\documentclass[a4paper,11pt]{article}
\usepackage{dsfont}
\usepackage{amssymb}
\usepackage{latexsym}
\usepackage{amsmath}
\usepackage{xcolor}
\usepackage{comment}
\usepackage{amsthm}
\usepackage{layout} 
\usepackage{ulem}
\usepackage{enumerate}
\usepackage{bm}
\usepackage{bbm} 
\usepackage[bbgreekl]{mathbbol} 
\usepackage{authblk}
\usepackage{setspace} 
\newif\ifrs
\rstrue
\ifrs \usepackage{mathrsfs} \fi  
\newif\ifcol
\coltrue 

\def\tjp{t_{j+1}}
\def\mfh{{\mathfrak H}}
\def\wh{\widehat}
\def\wt{\widetilde}
\newtheorem{theorem}{Theorem}[section]

\newtheorem{lemma}[theorem]{Lemma}

\newtheorem{proposition}[theorem]{Proposition}
\newtheorem{corollary}[theorem]{Corollary}
\newtheorem{remark}[theorem]{Remark}

\numberwithin{equation}{section}
\newtheorem{theorem*}{Theorem}
\newtheorem{ass*}[theorem*]{Assumption}
\newtheorem{note*}[theorem*]{Note}
\newtheorem{lemma*}[theorem*]{Lemma}
\newtheorem{definition*}[theorem*]{Definition}
\newtheorem{proposition*}[theorem*]{Proposition}
\newtheorem{corollary*}[theorem*]{Corollary}
\newtheorem{remark*}[theorem*]{Remark}
\newtheorem{example*}[theorem*]{Example}
\numberwithin{equation}{section}
\newif\ifcol
\colfalse
\ifcol
\newcommand{\colorr}{\color[rgb]{0.8,0,0}}
\newcommand{\colorg}{\color[rgb]{0,0.5,0}}

\newcommand{\colorn}{\color[rgb]{1,1,1}}

\newcommand{\coloroy}{\color[rgb]{1,0.95,0}}

\else

\newcommand{\colorr}{\color{black}}
\newcommand{\colorg}{\color{black}}

\newcommand{\colorn}{\color{black}}

\newcommand{\coloroy}{\color{black}}
\fi
\newif\ifcol
\colfalse
\ifcol
\newcommand{\sred}{\color[rgb]{0.8,0,0}}

\newcommand{\sblue}{\color[rgb]{0,0,0.8}}
\else
\newcommand{\sred}{\color{black}}

\newcommand{\sblue}{\color{black}}
\fi
\newif\ifcol
\colfalse
\ifcol

\else

\fi
\excludecomment{en-text}
\includecomment{jp-text}
\includecomment{comment}
\def\infm{{\infty\text{--}}}
\def\inftym{\infm}
\def\koko{{\coloroy{koko}}}
\def\bd{\begin{description}}
\def\ed{\end{description}}

\def\D2{\bbD_{2,\infty-}}

\def\tj{{t_j}}
\def\tjm{{t_{j-1}}}
\def\tjp{{t_{j+1}}}

\def\D{{\bf D}}

\def\cale{{\cal E}}
\def\calf{{\cal F}}

\def\calh{{\cal H}}

\def\ds{\displaystyle}
\def\yeq{\>=\>}
\def\yleq{\>\leq\>}
\def\ygeq{\>\geq\>}

\def\simleq{\ \raisebox{-.7ex}{$\stackrel{{\textstyle <}}{\sim}$}\ }

\def\half{\frac{1}{2}}

\def\up{\uparrow}

\def\y{\vspace*{3mm}\\}
\def\halflineskip{\vspace*{3mm}}
\def\nn{\nonumber}
\def\be{\begin{equation}}
\def\ee{\end{equation}}
\def\bea{\begin{eqnarray}}
\def\eea{\end{eqnarray}}
\def\beas{\begin{eqnarray*}}
\def\eeas{\end{eqnarray*}}
\def\bi{\begin{itemize}}
\def\ei{\end{itemize}}
\def\im{\item}
\def\bd{\begin{description}}
\def\ed{\end{description}}
\def\l{\left}
\def\r{\right}

\newcommand{\bbD}{{\mathbb D}}
\newcommand{\bbE}{{\mathbb E}}

\newcommand{\bbI}{{\mathbb I}}

\newcommand{\bbN}{{\mathbb N}}

\newcommand{\bbR}{{\mathbb R}}
\newcommand{\bbS}{{\mathbb S}}

\newcommand{\bbV}{{\mathbb V}}
\newcommand{\bbW}{{\mathbb W}}
\newcommand{\bbX}{{\mathbb X}}
\newcommand{\bbY}{{\mathbb Y}}
\newcommand{\bbZ}{{\mathbb Z}}

\def\tti{{\tt i}}

\makeatletter
\newsavebox{\@brx}
\newcommand{\llangle}[1][]{\savebox{\@brx}{\(\m@th{#1\langle}\)}%
  \mathopen{\copy\@brx\kern-0.5\wd\@brx\usebox{\@brx}}}
\newcommand{\rrangle}[1][]{\savebox{\@brx}{\(\m@th{#1\rangle}\)}%
  \mathclose{\copy\@brx\kern-0.5\wd\@brx\usebox{\@brx}}}
\makeatother

\newcommand{\stout}[1]{\ifmmode\text{\sout{\ensuremath{#1}}}\else\sout{#1}\fi}

\newcommand{\redc}[1]{{\myred #1}}

\newcommand{\grnc}[1]{{\mygreen #1}}
\newcommand{\grnb}[1]{{\mygreen [#1]}}
\newcommand{\delc}[1]{{}}
\newcommand{\delb}[1]{{[]}}
\newcommand{\bluc}[1]{{\myblue #1}}

\newcommand{\abs}[1]{{\left| #1 \right|}}

\newcommand{\rbr}[1]{{\left( #1 \right) }}

\usepackage{graphicx}
\usepackage{subcaption}
\usepackage{here}
\newif\ifcol
\colfalse
\ifcol

\renewcommand{\redc}[1]{{\myred #1}}

\renewcommand{\grnc}[1]{{\mygreen #1}}
\renewcommand{\grnb}[1]{{\mygreen [#1]}}
\renewcommand{\delc}[1]{{}}
\renewcommand{\delb}[1]{{[]}}
\renewcommand{\bluc}[1]{{\myblue #1}}

\else

\renewcommand{\redc}[1]{{#1}}

\renewcommand{\grnc}[1]{{#1}}
\renewcommand{\grnb}[1]{{[#1]}}
\renewcommand{\delc}[1]{{}}
\renewcommand{\delb}[1]{{[]}}
\renewcommand{\bluc}[1]{{#1}}

\fi
\begin{document}

\title{
Asymptotic expansion of an estimator for the Hurst coefficient
\footnote{
This work was in part supported by 
Japan Science and Technology Agency CREST JPMJCR2115; 
Japan Society for the Promotion of Science Grants-in-Aid for Scientific Research 
No. 17H01702 (Scientific Research);  
and by a Cooperative Research Program of the Institute of Statistical Mathematics. 
}
}
\author[1]{Yuliya Mishura}
\author[2,3]{Hayate Yamagishi}
\author[2,3]{Nakahiro Yoshida}
\affil[1]{Taras Shevchenko National University of Kyiv
        }
\affil[2]{Graduate School of Mathematical Sciences, University of Tokyo
        }
\affil[3]{CREST, Japan Science and Technology Agency
        }
\maketitle
\ \\
{\it Summary} 
Asymptotic expansion is presented for an estimator of the Hurst coefficient of a fractional Brownian motion. 
For this, a recently developed theory of asymptotic expansion of the distribution of Wiener functionals is applied. 
The effects of the asymptotic expansion are demonstrated by numerical studies. 
\begin{en-text}
{\color{gray}
The quasi-likelihood estimator and the Bayesian type estimator of the volatility parameter 
are in general asymptotically mixed normal. In case the limit is normal, the asymptotic expansion 
was derived in \cite{Yoshida1997} as an application of the martingale expansion. 
The expansion for the asymptotically mixed normal distribution is then indispensable 
to develop the higher-order approximation and inference for the volatility. 
The classical approaches in limit theorems, where the limit is a process with independent increments 
or a simple mixture, 
do not work.  
We present asymptotic expansion of a martingale with asymptotically mixed normal distribution. 
The expansion formula is expressed by the adjoint of a random symbol with coefficients 
described by the Malliavin calculus, 
differently from the standard invariance principle. 
Applications to a quadratic form of a diffusion process (``realized volatility'') is discussed. 
}
\end{en-text}
\ \\
\ \\
{\it Keywords and phrases } 
Asymptotic expansion, Hurst coefficient, fractional Brownian motion, 
Malliavin calculus, central limit theorem, Edgeworth expansion. 
\ \\


\section{Introduction}
Let $B=(B_t)_{t\in[0,T]}$ be a fractional Brownian motion with Hurst coefficient $H\in(0,1)$. 
For a fixed positive number $T$, let $\tj=t_j^n=jT/n$ for $n\in\bbN$ and $j\in\{0,1,...,n\}$. 
We are interested in estimation of the Hurst parameter $H$ 
from the sampled data $(B_\tj)_{j=0,1,...,n}$. 
The second-order difference $d_{n,j} $ of $B$ is denoted by 
\beas 
d_{n,j}&=& B_\tjp-2B_\tj+B_\tjm. 
\eeas
Note that $t_j$ depends on $n$. 
We make the sum of squares $V_{n,T}^{(2)}$ of $d_{n,j}$ by 
\beas 
V_{n,T}^{(2)} &=& \sum_{j=1}^{n-1} d_{n,j}^2.
\eeas
In Kubilius, Mishura and Ralchenko \cite{kubilius2018parameter} 
this estimator is denoted by $V_{n,T}^{(2)B}$. 
For estimation of $H$, 
Benassi et al. \cite{benassi1998identification} and Istas and Lang \cite{istas1997quadratic} introduced the estimator 
\bea\label{2108101802}
\wh{H}_n^{(2)}
&=& 
0\vee\bigg(
\half-\frac{1}{2\log2}\log\frac{V_{2n,T}^{(2)}}{V_{n,T}^{(2)}}
\bigg)\wedge1.
\eea

{\sred
The estimator $\wh{H}_n^{(2)}$ is preferable since it is consistent and asymptotically normal 
as $n\to\infty$. 
On the other hand, 
it is known that normal approximation to the distribution of a statistic is not necessarily satisfactory. 
Section \ref{2208290905} shows that the error of the normal approximation to the histogram of the estimator is not negligible numerically. 
In this paper, 
we will consider a higher-order approximation of the distribution of the error of $\wh{H}_n^{(2)}$ by mean of the asymptotic expansion. 

Asymptotic expansion is a standard concept in statistics and well developed for independent models. 
As for basic literature, we refer the reader to 
Cram\'er \cite{cramer1928composition, cramer2004random}, 
Gnedenko and Kolmogorov \cite{gnedenko1954limit}, 
Bhattacharya \cite{Bhattacharya1971}, 
Petrov \cite{Petrov1975}, 
Bhattacharya and Ranga Rao \cite{BhattacharyaRanga1976}, 
and to a recent textbook by Bhattacharya et al. \cite{bhattacharya2016course}. 
The theory of asymptotic expansion has been extended to dependent models such as mixing Markov processes and martingales, and 
an amount of literature is available today. 
See Yoshida \cite{yoshida2016asymptotic} and references therein. 

The theory of asymptotic expansion for Wiener functionals has been developed recently. 
In the central limit case, Tudor and Yoshida \cite{tudor2019asymptotic} derived 
the first-order expansion for vector-valued sequences of random variables, and 
Tudor and Yoshida \cite{tudor2019high} did 
the arbitrary order of asymptotic expansion for Wiener functionals. 
The latter has been applied to 
asymptotic expansion of the quadratic variation of a mixed fractional Brownian motion by 
Tudor and Yoshida \cite{tudor2020asymptotic}. 
In the mixed Gaussian limit case, Nualart and Yoshida \cite{nualart2019asymptotic} presented 
asymptotic expansion of Skorohod integrals. 
This method is used in Yoshida \cite{yoshida2020asymptotic} for Brownian functionals having anticipative weights. 
Yamagishi and Yoshida \cite{yamagishi2022order} provided 
order estimates of functionals related to fractional Brownian motion, with weighted graphs, and asymptotic expansion of the quadratic variation of fractional diffusions. 
In this paper, we will apply the scheme by Tudor and Yoshida \cite{tudor2019asymptotic} combined with a perturbation method. 

The organization of this paper is as follows. 
Section \ref{2108110853} gives the main results, and 
Section \ref{2208290939} proves them. 
The real performance of asymptotic expansion is numerically studied in Section \ref{2208290905}.

}

\section{Asymptotic expansion for $\wh{H}_n^{(2)}$}\label{2108110853}
To carry out the computation, we will work with the Malliavin calculus. 
{\sblue Suppose that $H\in(0,1)$.}
Denote by $\cale$ the set of step functions on $[0,T]$. 
Introduce an inner product in $\cale$ by 
\beas 
\langle 1_{[0,t]},1_{[0,s]}\rangle_{\sblue\mfh} &=& \half\big(t^{2H}+s^{2H}-|t-s|^{2H}\big)
\quad (t,s\in[0,T])
\eeas
Let $\|h\|_\mfh=\langle h,h\rangle_\mfh^{1/2}$ for $h\in\cale$. 
The closure of $\cale$ with respect to $\|\cdot\|_\mfh$ is denoted by $\mfh$. 
It is a real separable Hilbert space with an extended 
inner product $\langle\cdot,\cdot\rangle_\mfh$ and the corresponding norm 
denoted by $\|\cdot\|_\mfh$. 
With an isometric Gaussian process $\bbW=(\bbW(h))_{h\in\mfh}$ 
on some probability space $(\Omega,\calf,P)$, 
it is possible to realize a fractional Brownian motion $B$ by $B_t=\bbW(1_{[0,t]})$. 
Denote by $(\bbD_{s,p},\|\cdot\|_{s,p})$ the Sobolev space of Wiener functionals 
with differentiability index $s\in\bbR$ and integrability index $p>1$. 
We write $\bbD_\infty=\cap_{s,p}\bbD_{s,p}$. 
The Malliavin derivative is denoted by $D$, and 
the divergence operator (Skorokhod integral) by $\delta$. 
The $k$-fold Wiener integral of $h^{\otimes k}$ is denoted by 
$I_k(h^{\otimes k})=\delta^k(h^{\otimes k})$ for $h\in\mfh$. 
{\sred For basic elements in the Malliavin calculus, we refer the reader to 
Watanabe \cite{watanabe1984lectures}, Ikeda and Watanabe \cite{IkedaWatanabe1989}, 
Nualart \cite{Nualart2006} and Nourdin an Peccati \cite{nourdin2012normal}. 
}

{\sblue Let $1_j^n=1_{(t_{j-1}^n,t_j^n]}$. The function $1_j^n$ will simply be denoted by $1_j$.} 
Then 
\beas 
d_{n,j} = \delta(1_{j+1}-1_j)=\bbW(1_{j+1}-1_j)=I_1(1_{j+1}-1_j).
\eeas
By the product formula, we have 
\beas 
d_{n,j}^2 
&=& 
I_1(1_{j+1}-1_j)^2
\\&=&
I_2\big((1_{j+1}-1_j)^{\otimes2}\big)+\|1_{j+1}-1_j\|_\mfh^2.
\eeas
By definition, 
\beas
\|1_{j+1}-1_j\|_\mfh^2
&=& 
2\|1_j\|_\mfh^2-2\langle 1_{j+1},1_j\rangle
\\&=&
2(T/n)^{2H}-2E[\bbW(1_{j+1})\bbW(1_j)]
\\&=& 
2(T/n)^{2H}-\big((2T/n)^{2H}-2(T/n)^{2H}\big)
\\&=&
\big(4-2^{2H}\big)(T/n)^{2H}. 
\eeas
Thus, 
\beas 
E[d_{n,j}^2] &=& \big(4-2^{2H}\big)(T/n)^{2H}
\eeas
for $j=1,...,n-1$. 
In particular, 
\bea\label{2108101801} 
E\big[V_{n,T}^{(2)}\big]
&=& 
\big(4-2^{2H}\big)(T/n)^{2H}(n-1).
\eea
The use of $V_{n,T}^{(2)}$ for the estimator $\wh{H}_n^{(2)}$ of (\ref{2108101802}) 
is accounted by 
(\ref{2108101801}), that entails 
\beas 
\log\frac{E\big[V_{2n,T}^{(2)}\big]}{E\big[V_{n,T}^{(2)}\big]}
&=& 
\log\frac{\big(4-2^{2H}\big)(T/(2n))^{2H}(2n-1)}{\big(4-2^{2H}\big)(T/n)^{2H}(n-1)}
\\&=&
-2H\log2+\log\frac{2(1-1/(2n))}{1-1/n}
\\&=&
-2H\log2+\log2+\frac{1}{2n}+O(n^{-2})
\eeas
as $n\to\infty$. 

We say $X_n=o_M(n^\nu)$ [resp. $X_n=O_M(n^\nu)$] 
for a sequence $(X_n)_{n\in\bbN}$ of random variables and $\nu\in\bbR$ 
if $X_n\in\bbD_\infty$ and 
$\|X_n\|_{s,p}=o(n^\nu)$ [resp. $\|X_n\|_{s,p}=O(n^\nu)$] 
as $n\to\infty$ for any $s\in\bbR$ and $p\in(1,\infty)$. 
We shall need a high-order of stochastic expansion for $\wh{H}_n^{(2)}$ 
to go up to the asymptotic expansion from the central limit theorem. 
\begin{en-text}
\begin{proposition}\label{2108101817}
There exist 
a positive constant $G_\infty$ and 
a sequence $(\psi_n)_{n\in\bbN}$ of Wiener functionals 
satisfying the following conditions. 
\bd
\im[(i)] $\psi_n:\Omega\to[0,1]$ and $\psi_n\in\bbD_\infty$ for $n\in\bbN$. 
\im[(ii)] 
$\|G_n-G_\infty\|_{s,p}=O(n^{-1/2})$ and 
$\|\psi_n-1\|_{s,p}=O(n^{-L})$ 
as $n\to\infty$ 
for every $(s,p,L)\in\bbR\times(1,\infty)\times(0,\infty)$. 
\im[(iii)] 
The sequence $(Z_n)_{n\in\bbN}$ of random variables defined by 
\bea\label{2108101820}
Z_n 
&=& 
\bigg\{(2\log2)\sqrt{n}\big(\wh{H}_n^{(2)}-H\big)+\frac{1}{4\sqrt{n}}\bigg\}\psi_n
\eea
for $n\in\bbN$ admits a stochastic expansion 
$Z_n=M_n+n^{-1/2}N_n$, where 
\beas 
M_n
&=& 
n^{1/2}\bigg\{\bigg[\frac{V^{(2)}_{2n,T}}{E[V^{(2)}_{2n,T}]}-1\bigg]
-\bigg[\frac{V^{(2)}_{n,T}}{E[V^{(2)}_{n,T}]}-1\bigg]\bigg\}
\eeas
and 
\beas 
N_n 
&=& 
n\bigg\{
\half\bigg[\frac{V^{(2)}_{2n,T}}{E[V^{(2)}_{2n,T}]}-1\bigg]^2
{\colorr-}\half\bigg[\frac{V^{(2)}_{n,T}}{E[V^{(2)}_{n,T}]}-1\bigg]^2
\bigg\}+o_M(1)
\eeas
as $n\to\infty$. In particular, $Z_n\in\bbD_\infty$ for $n\in\bbN$. 
\ed
\end{proposition}
\end{en-text}
\begin{proposition}\label{2108101817}
There exists a sequence $(\psi_n)_{n\in\bbN}$ of Wiener functionals 
satisfying the following conditions. 
\bd
\im[(i)] $\psi_n:\Omega\to[0,1]$, $\psi_n\in\bbD_\infty$ for $n\in\bbN$, 
and 
$\|\psi_n-1\|_{s,p}=O(n^{-L})$ 
as $n\to\infty$ 
for every $(s,p,L)\in\bbR\times(1,\infty)\times(0,\infty)$. 
\im[(ii)] 
The sequence $(Z_n)_{n\in\bbN}$ of random variables defined by 
\bea\label{2108101820}
Z_n 
&=& 
(2\log2)\sqrt{n}\big(\wh{H}_n^{(2)}-H\big)\psi_n
\eea
for $n\in\bbN$ admits a stochastic expansion 
\bea\label{2108141223}
Z_n&=&M_n+n^{-1/2}N_n,
\eea 
where 
\grnc{\bea\label{210811857} 
M_n
&=& 
-n^{1/2}\bigg\{\bigg[\frac{V^{(2)}_{2n,T}}{E[V^{(2)}_{2n,T}]}-1\bigg]
-\bigg[\frac{V^{(2)}_{n,T}}{E[V^{(2)}_{n,T}]}-1\bigg]\bigg\}
\eea}
and 
\bea\label{210811858}
N_n 
&=& 
n\bigg\{
\half\bigg[\frac{V^{(2)}_{2n,T}}{E[V^{(2)}_{2n,T}]}-1\bigg]^2
{\colorr-}\half\bigg[\frac{V^{(2)}_{n,T}}{E[V^{(2)}_{n,T}]}-1\bigg]^2
\bigg\}-\half+O_M(n^{-1/2})
\eea
as $n\to\infty$. In particular, $Z_n\in\bbD_\infty$ for $n\in\bbN$. 
\ed
\end{proposition}
See Section \ref{2108110837} for proof of Proposition \ref{2108101817}. 

\noindent{\sblue
From now on, our approach will be as follows.} 
We will first validate asymptotic expansion of the distribution of $M_n$, 
and next apply the perturbation method to the stochastic expansion 
(\ref{2108141223}) with (\ref{210811857}) and (\ref{210811858}) 
in order to derive an Edgeworth expansion formula for $\wh{H}_n^{(2)}$. 
%
{\sblue For positive numbers $K$ and $\gamma$, 
$\cale(K,\gamma)$ denotes} the set of measurable functions on $\bbR$ 
such that $|f(z)|\leq K(1+|z|^\gamma)$ for all $z\in\bbR$. 
{\sblue 
The density function of the normal distribution with mean $\mu$ and variance $C$ is denoted by $\phi(z;\mu,C)$.} 
The main interest of this paper is in the following theorem. 
\begin{theorem}\label{2108141203}
There exist a positive constant $v_H$ and {\sred an odd} polynomial $q(z)$ of degree $3$ such that 
the function 
\bea\label{2108141155}
p_n(z)
&=& 
\big(1+n^{-1/2}q(z)\big)\phi(z;0,v_H)
\eea
satisfies 
\bea\label{2108150241} 
\sup_{f\in\cale(K,\gamma)}
\bigg|E\big[f\big(\sqrt{n}(\wh{H}_n^{(2)}-H)\big)\big]-\int_\bbR f(z)p_n(z)dz\bigg|
&=& 
o(n^{-1/2})\quad(n\to\infty)
\eea
for every positive numbers $K$ and $\gamma$. 
\end{theorem}

See {\sblue (\ref{2108141156}) of} Section \ref{2108140213} for {\sblue the definition} of $q(z)$ and $v_H$. 

{\sred The wedge $\wedge$ and the vee $\vee$ stand for $\min$ and $\max$, respectively.}
The second-order modification of $\wh{H}_n^{(2)}$ is given by 
\bea\label{2108150237} 
\wh{H}_n^{(b)}
&=& 
\sred{
0\vee\bigg\{\wh{H}_n^{(2)}-\frac{1}{n}b\big(\wh{H}_n^{(2)}\big)\bigg\}\wedge1,}
\eea
where the function $b:[0,1]\to\bbR$ is assumed to be smooth on $(0,1)$. 
Asymptotic expansion of the distribution of $\wh{H}_n^{(b)}$ can be obtained easily 
by a trivial modification of that of $\wh{H}_n^{(2)}$. 
Define $p_n^{(b)}(z)$ by 
\bea\label{2108150212}
p_n^{(b)}(z)
&=& 
\big(1+n^{-1/2}q^{(b)}(z)\big)\phi(z;0,v_H)
\eea
with 
\beas 
q^{(b)}(z)
&=& 
q(z) - \frac{b(H)}{v_H}z
\eeas

\begin{theorem}\label{2108150211}
For any positive numbers $K$ and $\gamma$,  
\bea\label{2108150242} 
\sup_{f\in\cale(K,\gamma)}
\bigg|E\big[f\big(\sqrt{n}(\wh{H}_n^{(b)}-H)\big)\big]-\int_\bbR f(z)p_n^{(b)}(z)dz\bigg|
&=& 
o(n^{-1/2})
\eea
as $n\to\infty$. 
\end{theorem}

As an application of Theorem \ref{2108150211}, we obtain 
a second-order unbiased estimator $\wh{H}_n^{(*)}$ 
by choosing the function 
\beas 
b^*(H) 
&=& 
\int_\bbR zq(z)\phi(z;0,v_H)dz
\eeas
as $b(H)$, namely, $\wh{H}_n^{(*)}=\wh{H}_n^{(b^*)}$. 
Then 
\beas 
E\big[\sqrt{n}(\wh{H}_n^{(*)}-H)\big]
&=& 
o(n^{-1/2})
\eeas
as $n\to\infty$. 
If taking the function 
\beas 
b^{**}(H) 
&=& 
\int_0^\infty q(z)e^{-\frac{z^2}{2v_H}}dz
\eeas
as $b(H)$, then the estimator $\wh{H}_n^{(**)}=\wh{H}_n^{(b^{**})}$ satisfies 
\beas 
P\big[\sqrt{n}\big(\wh{H}_n^{(**)}-H\big)\leq0\big]
&=& 
\half+o(n^{-1/2})
\eeas
and 
\beas 
P\big[\sqrt{n}\big(\wh{H}_n^{(**)}-H\big)\geq0\big]
&=& 
\half+o(n^{-1/2})
\eeas
as $n\to\infty$. 
{\sblue That is,} the estimator $\wh{H}_n^{(**)}$ is 
second-order median unbiased. 
{\sred 
See Remark \ref{2109051131} about regularity of the funcitons $b^*$ and $b^{**}$.}

The rest of this paper will be devoted to proof of 
Theorems \ref{2108141203} and \ref{2108150211}.

\section{Proof}\label{2208290939}

\subsection{Preliminary lemmas}
Let $k\in\bbN$ with $k\geq2$. 
Consider $k+1$ sequences of positive integers
\beas 
\big(\nu(n)\big)_{n\in\bbN},\>
\big(\nu_1(n)\big)_{n\in\bbN},...,\big(\nu_k(n)\big)_{n\in\bbN}
\eeas
such that 
\bea\label{2108111652}
\lim_{n\to\infty}\nu(n)=\infty\quad\text{and}\quad
\lim_{n\to\infty}\frac{\nu_\alpha(n)}{\nu(n)} \yeq p_\alpha
\eea
for some constant $p_\alpha\in(0,\infty)$ for every $\alpha=1,...,k$.

For functions $\varrho_1,...,\varrho_k:\bbZ\to\bbR$, let 
\beas 
A_n(\varrho_1,...,\varrho_k) 
&=& 
\nu(n)^{-1}\sum_{j_1=1}^{\nu_1(n)}\cdots\sum_{j_k=1}^{\nu_k(n)}
\varrho_1(j_1-j_2)\varrho_2(j_2-j_3)\cdots\varrho_{k-1}(j_{k-1}-j_k)\varrho_k(j_k-j_1)
\eeas
and also let 
\beas 
\bbV(\varrho_1,...,\varrho_k) &=& 
\bigg|\sum_{i_1,...,i_{k-1}\in\bbZ}
\varrho_1(i_1)\varrho_2(i_1-i_2)\varrho_3(i_2-i_3)\cdots\varrho_{k-1}(i_{k-2}-i_{k-1})\varrho_k(i_{k-1})
\bigg|.
\eeas
{\sred
For $\bbI=\{1,...,k-1\}$}, let 
\beas 
p_n(i_1,...,i_{k-1})
&=&
\bigg[
\bigg\{\bigg(\bigwedge_{\alpha\in\bbI}\frac{\nu_{\alpha+1}(n)-i_\alpha}{\nu(n)}\bigg)
\wedge\frac{\nu_1(n)}{\nu(n)}\bigg\}
-
\bigg\{\bigg(\bigvee_{\alpha\in\bbI}\frac{-i_\alpha}{\nu(n)}\bigg)\vee0\bigg\}
\bigg]
\\&&\hspace{50pt}\times
1_{\{1-\nu_1(n)\leq i_\alpha\leq\nu_{\alpha+1}(n)\>
(\alpha\in\bbI)
\}}.
\eeas
%
\begin{lemma}\label{2108111404}
\bd\im[(a)] 
$\ds
\bbV(\varrho_1,...,\varrho_k) \yleq 
\bbV(|\varrho_1|,...,|\varrho_k|) \yleq 
\big\|\varrho_1\big\|_{\ell_{\frac{k}{k-1}}}\cdots\big\|\varrho_k\big\|_{\ell_{\frac{k}{k-1}}}
$
\im[(b)] 
\begin{en-text}
\beas 
A_n(\varrho_1,...,\varrho_k) 
&=& 
\sum_{i_1,...,i_{k-1}\in\bbZ}
\varrho_1(-i_1)\varrho_2(i_1-i_2)\varrho(i_2-i_3)\cdots\varrho_{k-1}(i_{k-2}-i_{k-1})\varrho_k(i_{k-1})
\\&&\hspace{50pt}\times
\bigg[1+\frac{1}{n}\bigg\{0\wedge\big(-\vee_{\alpha\in\bbI}i_\alpha\big)\bigg\}
-\frac{1}{n}\bigg\{0\vee\big(-\wedge_{\alpha\in\bbI}i_\alpha\big)\bigg\}\bigg]
1_{\{\vee_{\alpha\in\bbI}|i_\alpha|<n\}}
\eeas
\end{en-text}
\beas 
A_n(\varrho_1,...,\varrho_k) 
&=& 
\sum_{i_1,...,i_{k-1}\in\bbZ}
\varrho_1(-i_1)\varrho_2(i_1-i_2)\varrho_{\sblue3}(i_2-i_3)\cdots
\\&&\hspace{60pt}
\cdots
\varrho_{k-1}(i_{k-2}-i_{k-1})\varrho_k(i_{k-1})
p_n(i_1,...,i_{k-1}).
\eeas

\im[(c)] Suppose that the condition (\ref{2108111652}) is satisfied. 
If $\bbV(|\varrho_1(-\cdot)|,|\varrho_2|,...,|\varrho_k|) <\infty$, then the limit
\bea\label{2108111426}
\lim_{n\to\infty}A_n(\varrho_1,...,\varrho_k)
&=& 
\bigg(\bigwedge_{\alpha=1,...,k} p_\alpha\bigg)
\sum_{i_1,...,i_{k-1}\in\bbZ}
\varrho_1(-i_1)\varrho_2(i_1-i_2)\varrho(i_2-i_3)\cdots
\nn\\&&\hspace{120pt}\times
\varrho_{k-1}(i_{k-2}-i_{k-1})\varrho_k(i_{k-1})
\nn\\&&
\eea
exists, that is, the sum on the right-hand side of (\ref{2108111426}) is finite. 
In particular, if $\big\|\varrho_\alpha\big\|_{\ell_{\frac{k}{k-1}}}<\infty$ for all $\alpha\in\{1,...,k\}$, then the convergence (\ref{2108111426}) holds under (\ref{2108111652}). 
\im[(d)] Suppose that the functions $\varrho_{1},...,\varrho_{k}:\bbZ\to\bbR$ satisfy
\beas 
|\varrho_{\alpha}(j)| &\leq& C(1+|j|)^{-\gamma}\quad(j\in\bbZ;\>\alpha=1,...,k)
\eeas
for some constant $\gamma>\frac{k}{k-1}$. 
Suppose that the condition (\ref{2108111652}) is satisfied. 
Then 
\beas 
\bbV(\varrho_{1},...,\varrho_{k}) \leq 
\bbV(|\varrho_{1}|,...,|\varrho_{k}|)
<
\infty
\eeas
In particular, 
\beas 
\sup_{n\in\bbN}\big|A_n(\varrho_{1},...,\varrho_{k})\big|
\leq 
\sup_{n\in\bbN}\big|A_n(|\varrho_{1}|,...,|\varrho_{k}|)\big|<\infty, 
\eeas
and the convergence (\ref{2108111426}) holds. 
\ed
\end{lemma}
\proof
(a) 
The system of linear equations 
\beas 
\l\{\begin{array}{l}
\frac{k-1}{p}\yeq\frac{1}{q}+(k-2)
\y
\frac{1}{p}+\frac{1}{q}\yeq1
\end{array}\r.
\eeas
is solved by $p=\frac{k}{k-1}$ and $q=k$. 
Therefore, Young's inequality yields 
\beas 
\bbV(|\varrho_1|,...,|\varrho_k|) &\leq& 
\big\||\varrho|_1*\cdots*|\varrho_{k-1}|\big\|_{\ell_k}\>\big\|\varrho_k\big\|_{\ell_{\frac{k}{k-1}}}
\\&\leq&
\big\|\varrho_1\big\|_{\ell_{\frac{k}{k-1}}}\cdots\big\|\varrho_k\big\|_{\ell_{\frac{k}{k-1}}}.
\eeas

\noindent 
(b)
By the change of variables 
\beas 
i_\alpha \yeq j_{\alpha+1} - j_1\quad(\alpha=1,...,k-1)
\eeas
for given $j_1$, we obtain 
\beas 
A_n(\varrho_1,...,\varrho_k) 
&=& 
\nu(n)^{-1}
\sum_{j_1=1}^{\nu_1(n)}\sum_{i_1,...,i_{k-1}\in\bbZ}
\varrho_1(-i_1)\varrho_2(i_1-i_2)\varrho(i_2-i_3)\cdots\varrho_{k-1}(i_{k-2}-i_{k-1})\varrho_k(i_{k-1})
\\&&\hspace{100pt}\times
1_{\{1-i_\alpha\leq j_1\leq \nu_\alpha(n)-i_\alpha\>(\alpha=1,...,k-1)\}}
\\&&\hspace{100pt}\times
1_{\{1-\nu_1(n)\leq i_\alpha\leq\nu_{\alpha+1}(n)\>(\alpha=1,...,k-1)\}}
\\&=&
\sum_{i_1,...,i_{k-1}\in\bbZ}
\varrho_1(-i_1)\varrho_2(i_1-i_2)\varrho(i_2-i_3)\cdots\varrho_{k-1}(i_{k-2}-i_{k-1})\varrho_k(i_{k-1})
p_n(i_1,...,i_{k-1})
\eeas
%

\noindent
(c) 
Under the condition (\ref{2108111652}), 
\beas 
\sup_{n\in\bbN}\sup_{i_1,...,i_{k-1}\in\bbZ}|p_n(i_1,...,i_{k-1})|<\infty
\quad\text{and}\quad
\lim_{n\to\infty}p_n(i_1,...,i_{k-1})\yeq \bigwedge_{\alpha=1,...,k} p_\alpha.
\eeas
Then Lebesgue's theorem implies the convergence (\ref{2108111426}). \y
\noindent
(d) Use (a) and (c) to prove (d). 
\qed\halflineskip

\begin{en-text}
Suppose that the even functions $\varrho_{1,n},...,\varrho_{k,n}:\bbZ\to\bbR$ 
possibly depending on $n\in\bbN$ 
satisfy
\beas 
|\varrho_{m,n}(j)| &\leq& C(1+|j|)^{-\gamma}\quad(j\in\bbZ;\>m=1,...,k;\>n\in\bbN).
\eeas
Then 
\beas 
\bbV(\varrho_{1,n},...,\varrho_{k,n}) \leq 
\bbV(|\varrho_{1,n}|,...,|\varrho_{k,n}|)
<
\infty
\eeas
whenever 
$\gamma>\frac{k}{k-1}$, in particular, 
\beas 
\sup_{n\in\bbN}\big|A_n(\varrho_{1,n},...,\varrho_{k,n})\big|
\leq 
\sup_{n\in\bbN}\big|A_n(|\varrho_{1,n}|,...,|\varrho_{k,n}|)\big|<\infty. 
\eeas
\end{en-text}
\begin{remark}\label{2108111803}\rm
(i) 
We have $\bbV\big(\overbrace{|\rho_1(-\cdot)|,...,|\rho_k|}^{k}\big) < \infty$ 
for $(\rho_1,...,\rho_k)\in\{\wh{\rho},\wt{\rho}\}^k$ 
since $4-2H>2>\frac{k-1}{k}$ 
for every $k\geq2$, thanks to the estimates 
(\ref{2108110926}) and (\ref{2108111821}) below.
(ii) It is possible to strengthen the result (\ref{2108111426}) 
to a representation of the form 
$A_n(\varrho_1,...,\varrho_k)\yeq \text{constant}+O(n^{-1})$
by putting a more restrictive condition on $\nu(n)$ and $\nu_\alpha(n)$ 
than (\ref{2108111652}). 
In this paper, such a convergence is necessary only in the case $k=2$, and 
we will give estimate for the error term directly without the help of Lemma \ref{2108111404}. 
\end{remark}

{\sblue 
Recall that $1^n_j=1_{(t_{j-1}^n,t_j^n]}$.}
Let 
\bea\label{2108121842}
f^{[\alpha]}_n
&=& 
\frac{\sqrt{n}\sum_{j=1}^{\alpha n-1}(1^{\alpha n}_{j+1}-1^{\alpha n}_j)^{\otimes2}}{E[V^{(2)}_{\alpha n,T}]}
\eea
for $\alpha\in\{1,2\}$. 
Then 
\grnc{$M_n=I_2(f_n^{[1]})-I_2(f_n^{[2]})$} for $M_n$ defined by (\ref{210811857}). 
{\sred 
The operator $L=-\delta D$ is the Malliavin operator (Ornstein-Uhlenbeck operator). 
It is a numeric operator such that $LF=(-q)F$ for elements $F$ of the $q$-th Wiener chaos.} 
Since 
\grnc{$
(-L)^{-1}M_n
\yeq 
2^{-1}I_2(f_n^{[1]})-2^{-1}I_2(f_n^{[2]})
$}
and 
\grnc{$
D(-L)^{-1}M_n
\yeq
I_1(f_n^{[1]})-I_1(f_n^{[2]})
$}, 
the second-order $\Gamma$-factor $\Gamma_n^{(2)}(M_n)$ of $M_n$ is 
given by 
\beas
G_n
\yeq
\Gamma_n^{(2)}(M_n)
\yeq
\langle DM_n,D(-L)^{-1}M_n\rangle_\mfh
\grnc{\yeq
\big\langle 2I_1(f_n^{[1]})-2I_1(f_n^{[2]}),\>I_1(f_n^{[1]})-I_1(f_n^{[2]})\big\rangle_\mfh}
\eeas
for $n\in\bbN$. 
%
The product formula gives 
\bea\label{2108121130}
G_n 
&=& 
2I_2\big(f_n^{[2]}\otimes_1f_n^{[2]}\big)+2\|f_n^{[2]}\|_{\mfh^{\otimes2}}^2
-4I_2\big(f_n^{[2]}\odot_1f_n^{[1]}\big)-4\langle f_n^{[2]},f_n^{[1]}\rangle_{\mfh^{\otimes2}}
\nn\\&&
+2I_2\big(f_n^{[1]}\otimes_1f_n^{[1]}\big)+2\|f_n^{[1]}\|_{\mfh^{\otimes2}}^2,
\eea
where $f_n^{[2]}\odot_1f_n^{[1]}$ is the symmetrized 
$1$-contruction of $f_n^{[2]}\otimes f_n^{[1]}$. 
\begin{en-text}
\beas 
f_n^{[2]}\odot_1f_n^{[1]} 
&=& 
\big(\langle f_n^{[2]},\ f_n^{[1]}\rangle_\mfh\big)^{sym}
\eeas
\end{en-text}
\halflineskip

{\sred 
Backward shift operator ${\sf{B}}$ 
defined by $B\theta(x)=\theta(x-1)$ for 
$x\in\bbR$ and 
a sequence $\theta=(\theta(x+t))_{t\in\bbZ}$ of numbers. 
The following lemma is an exercise. 
\begin{lemma}\label{2108301908}
Let $x\in\bbR$, $I\in\bbN$ and $k_1,...,k_I\in\bbN$. Then 
\beas 
\prod_{i=1}^{I}(1-B^{k_i})f(x)
&=& 
(-1)^I\int_0^{k_1}\cdots\int_0^{k_I}f^{(I)}\big(x-s_1-\cdots-s_I)ds_1\cdots ds_I
\eeas
for any function $f\in C^I\big([x-\sum_{i=1}^Ik_i,x]\big)$, {\sblue the set of functions of class $C^I$ on $[x-\sum_{i=1}^Ik_i,x]$.}
\end{lemma}
}

Let 
\bea\label{2108110945}
\wh{\rho}(j)
&=& 
\half
\big\{-|j-2|^{2H}+4|j-1|^{2H}-6|j|^{2H}+4|j+1|^{2H}-|j+2|^{2H}\big\}
\eea
\begin{lemma}\label{2108110915}
\bea\label{2108110926}
\wh{\rho}(i)
&\sim& 
H(2H-1)(2H-2)(2H-3)|i|^{2H-4}\quad(|i|\to\infty)
\eea
In particular, 
$\sum_{i=1}^\infty|i|^3\wh{\rho}(i)^2<\infty$. 
\end{lemma}
\proof 
With the backward shift operator $B$, we have  
\beas 
(1-{\sf{B}})^4
&=& 
1-4{\sf{B}}+6{\sf{B}}^2-4{\sf B}^3+{\sf{B}}^4.
\eeas
Therefore 
\beas 
\wh{\rho}(j)
&=& -\half \big((1-{\sf{B}})^4 |\cdot+2|^{{\sred 2H}}\big)(j)
\eeas
and hence (\ref{2108110926}) follows {\sred from Lemma \ref{2108301908}}. 
\qed

Let 
\beas
\Sigma_{11}
&=&
2\bigg[1+\frac{2}{\big(4-2^{2H}\big)^2}\sum_{i=1}^\infty\wh{\rho}(i)^2\bigg],
\eeas
and let 
\bea\label{2108211148}
\Sigma_{22}=\half\Sigma_{11}. 
\eea
\begin{en-text}
\beas 
\Sigma_{22}
&=& 
1+\frac{2}{\big(4-2^{2H}\big)^2}\sum_{i=1}^\infty\wh{\rho}(i)^2
\yeq
\half\Sigma_{11}
\eeas
\end{en-text}
\begin{lemma}\label{2108110931}
\bd
\im[(a)] 
For $j,j'\in\{1,...,n-1\}$, 
\bea\label{2108110941}
\big\langle 1_{j+1}^n-1_j^n, 1_{j'+1}^n-1_{j'}^n\big\rangle_\mfh
&=&
\frac{T^{2H}}{n^{2H}}\wh{\rho}(j-j'). 
\eea
\im[(b)] 
For $k,k'\in\{1,...,2n-1\}$, 
\bea\label{2108110942}
\big\langle 1_{k+1}^{2n}-1_k^{2n}, 1_{k'+1}^{2n}-1_{k'}^{2n}\big\rangle_\mfh
&=&
\frac{T^{2H}}{(2n)^{2H}}\wh{\rho}(k-k'). 
\eea
\im[(c)] As $n\to\infty$, 
\bea\label{2108110943}
\|f_n^{[1]}\|_{\mfh^{\otimes2}}^2
&=& 
\half\Sigma_{11}+O(n^{-1}).
\eea
\im[(d)] As $n\to\infty$, 
\bea\label{2108110944}
\|f_n^{[2]}\|_{\mfh^{\otimes2}}^2
&=& 
\half\Sigma_{22}+O(n^{-1}).
\eea
\ed
\end{lemma}
\proof 
{\sred\noindent
(a) and (b):}
We will use an elementary formula  
\bea\label{2108111320} 
E\big[(B_d-B_c)(B_b-B_a)\big]
&=& 
\half\big(|d-a|^{2H}+|c-b|^{2H}-|d-b|^{2H}-|c-a|^{2H}\big)
\eea
for $a,b,c,d\in[0,T]$. 
We obtain (\ref{2108110941}) and(\ref{2108110942}) since 
\beas 
\big\langle 1_{j+1}^n-1_j^n, 1_{j'+1}^n-1_{j'}^n\big\rangle_\mfh
&=&
\langle 1_{j+1}^n,1_{j'+1}^n\rangle_\mfh
+\langle 1_j^n,1_{j'}^n\rangle_\mfh
-\langle 1_{j+1}^n,1_{j'}^n\rangle_\mfh
-\langle 1_j^n,1_{j'+1}^n\rangle_\mfh
\\&=&
2\times\frac{T^{2H}}{2n^{2H}}\big\{|j'-j+1|^{2H}+|j'-j-1|^{2H}-2|j'-j|^{2H}\big\}
\\&&
-\frac{T^{2H}}{2n^{2H}}\big\{|j'-j|^{2H}+|j'-j-2|^{2H}-2|j'-j-1|^{2H}\big\}
\\&&
-\frac{T^{2H}}{2n^{2H}}\big\{|j'-j|^{2H}+|j'-j+2|^{2H}-2|j'-j+1|^{2H}\big\}
\\&=&
\frac{T^{2H}}{n^{2H}}\wh{\rho}(j'-j)
\eeas
from (\ref{2108110945}). 
{\sred (\ref{2108110942}) is now trivial. }

\begin{en-text}
\beas &&
\sum_{j,k=1}^{n-1}\big\langle (1_{j+1}^n-1_j^n)^{\otimes2}, (1_{k+1}^n-1_k^n)^{\otimes2}
\big\rangle_{\mfh^{\otimes2}}
\\&=& 
\sum_{j,k=1}^{n-1}\big\langle 1_{j+1}^n-1_j^n, 1_{k+1}^n-1_k^n\big\rangle_\mfh^2
\\&=&
\sum_{j,k=1}^{n-1}\big\{
\langle 1_{j+1}^n,1_{k+1}^n\rangle_\mfh
+\langle 1_j^n,1_k^n\rangle_\mfh
-\langle 1_{j+1}^n,1_k^n\rangle_\mfh
-\langle 1_j^n,1_{k+1}^n\rangle_\mfh\big\}^2
\eeas
\end{en-text}
{\sred 
\noindent
(c) and (d):}
For $j,j'\in\{1,...,n-1\}$, 
\beas 
\sum_{j,j'=1}^{n-1}\big\langle (1_{j+1}^n-1_j^n)^{\otimes2}, (1_{j'+1}^n-1_{j'}^n)^{\otimes2}
\big\rangle_{\mfh^{\otimes2}}
&=& 
\sum_{j,j'=1}^{n-1}\big\langle 1_{j+1}^n-1_j^n, 1_{j+1}^n-1_{j'}^n\big\rangle_\mfh^2
\\&=&
\l(\frac{T^{2H}}{n^{2H}}\r)^2\wh{\rho}(j-j')^2. 
\eeas
\begin{en-text}
\im 
\beas 
\big\langle 1_{j+1}^n-1_j^n, 1_{k+1}^n-1_k^n\big\rangle_\mfh
&=&
\langle 1_j^n,1_k^n\rangle_\mfh
\\&=& 
\frac{T^{2H}}{2n^{2H}}\big\{(|k-j|+1)^{2H}+(|k-j|-1)^{2H}-2|k-j|^{2H}\big\}
\\&\simleq&
H(2H-1)T^{2H}n^{-2H}(|k-j|^{2H-2}+1_{\{k=j\}})
\eeas
\end{en-text}
\begin{en-text}
\im In particular, 
\beas 
\big\langle 1_{j+1}^{2n}-1_j^{2n}, 1_{k+1}^{2n}-1_k^{2n}\big\rangle_\mfh
&=&
\frac{T^{2H}}{(2n)^{2H}}\wh{\rho}(k-j)
\eeas
\end{en-text}
\begin{en-text}
\beas
\sum_{j,k=1,..,n\atop j<k}
\big\{\langle 1_{j+1}^n,1_{k+1}^n\rangle_\mfh
+\langle 1_j^n,1_k^n\rangle_\mfh
-\langle 1_{j+1}^n,1_k^n\rangle_\mfh
-\langle 1_j^n,1_{k+1}^n\rangle_\mfh\big\}^2
\eeas
\end{en-text}
Therefore
\beas 
\|f_n^{[1]}\|_{\mfh^{\otimes2}}^2
&=& 
\bigg\|\frac{\sqrt{n}\sum_{j=1}^{n-1}(1^{n}_{j+1}-1^{n}_j)^{\otimes2}}{E[V^{(2)}_{n,T}]}
\bigg\|_{\mfh^{\otimes2}}^2
\\&=&
\frac{n}{\big(4-2^{2H}\big)^2(T/n)^{4H}(n-1)^2}
\l(\frac{T^{2H}}{n^{2H}}\r)^2
\sum_{j,k=1}^{n-1}\wh{\rho}(k-j)^2
\\&=&
\frac{n}{\big(4-2^{2H}\big)^2 (n-1)^2}\sum_{j,k=1}^{n-1}\wh{\rho}(k-j)^2
\\&=&
\frac{n^2}{\big(4-2^{2H}\big)^2 (n-1)^2}
\bigg[\big(4-2^{2H}\big)^2\>\frac{n-1}{n}+2\sum_{i=1}^{n-2} \l(1-\frac{i-1}{n}\r)\wh{\rho}(i)^2\bigg]
\\&=&
1+\frac{2}{\big(4-2^{2H}\big)^2}\sum_{i=1}^\infty\wh{\rho}(i)^2+O(n^{-1})
\\&=&
\half\Sigma_{11}+O(n^{-1}).
\eeas
With this result, we also obtain
\beas 
\|f_n^{[2]}\|_{\mfh^{\otimes2}}^2
&=& 
\bigg\|\frac{\sqrt{2^{-1}\cdot2n}\sum_{j=1}^{2n-1}(1^{2n}_{j+1}-1^{2n}_j)^{\otimes2}}{E[V^{(2)}_{2n,T}]}
\bigg\|_{\mfh^{\otimes2}}^2
\\&=&
\half\bigg[\half\Sigma_{11}+O(n^{-1})\bigg]
\yeq
\half\Sigma_{22}+O(n^{-1}). 
\eeas
\qed

Let 
\beas 
\wt{\rho}(j) 
&=& 
\frac{1}{2^{2H+1}}\bigg\{
-|j-3|^{2H}+2|j-2|^{2H}+|j-1|^{2H}-4|j|^{2H}+|j+1|^{2H}
\\&&
+2|j+2|^{2H}-|j+3|^{2H}\bigg\}.
\eeas
for $j\in\bbZ$. 
We will write 
\begin{en-text}
\beas 
\wt{\rho}(j,k)
&=& 
\frac{1}{2^{2H+1}}\bigg\{
-|j-2k-3|^{2H}+2|j-2k-2|^{2H}+|j-2k-1|^{2H}-4|j-2k|^{2H}+|j-2k+1|^{2H}
\\&&
+2|j-2k+2|^{2H}-|j-2k+3|^{2H}\bigg\}. 
\eeas
\end{en-text}
{\sred
\beas 
\wt{\rho}(k,j)
&=& 
\frac{1}{2^{2H+1}}\bigg\{
-|k-2j-3|^{2H}+2|k-2j-2|^{2H}+|k-2j-1|^{2H}-4|k-2j|^{2H}+|k-2j+1|^{2H}
\\&&
+2|k-2j+2|^{2H}-|k-2j+3|^{2H}\bigg\}. 
\eeas
}
Then $\wt{\rho}(k,j)=\wt{\rho}(k-2j)$. 
\begin{lemma}\label{2108121121}
\bea\label{2108111821} 
\wt{\rho}(j)
&\sim& 
2^{2-2H}H(2H-1)(2H-2)(2H-3)|j|^{2H-4}\quad(|j|\to\infty). 
\eea
In particular, $\sum_{j\in\bbZ}|i|^3\wt{\rho}(j)^2<\infty$.
\end{lemma}
\proof 
We have 
\beas 
(1-{\sf B})^2(1-{\sf B}^2)^2
&=&
1-2{\sf B}-{\sf B}^2+4{\sf B}^3-{\sf B}^4-2{\sf B}^5+{\sf B}^6
\eeas
for the backward shift operator ${\sf B}$. 
Therefore we obtain (\ref{2108111821}). 
\qed\halflineskip

Let 
\beas 
\Sigma_{12}
&=& 
\frac{{\sred 2^{2H}}}{
(4-2^{2H})^2
}
\>\sum_{\ell=-\infty}^{\infty}
\wt{\rho}(\ell)^2. 
\eeas
\begin{en-text}
{\sred The symbols $\Sigma_{ij}$ $(i,j=1,2)$ are used in 
Kubilius, Mishura and Ralchenko \cite{kubilius2018parameter}. 

$\up$ Question. Is this factor ``$2^{2H}$'' correct? 
Your book's $\Sigma_{12}$ seems to have a different value. }
\end{en-text}
\begin{lemma}\label{2108111311}
\bd
\im[(a)] 
For $j\in\{1,...,n-1\}$ and $k\in\{1,...,2n-1\}$, 
\bea\label{2108111328}
\big\langle 1_{j+1}^{n}-1_j^{n}, 1_{k+1}^{2n}-1_k^{2n}\big\rangle_\mfh
&=&
\frac{T^{2H}}{n^{2H}}\wt{\rho}(k,j).
\eea
\im[(b)] As $n\to\infty$, 
\bea\label{2108111329}
\langle f_n^{[2]},f_n^{[1]}\rangle_{\mfh^{\otimes2}}
&=&
\half\Sigma_{12}+O(n^{-1}).
\eea
\ed
\end{lemma}
\proof 
We have 
\beas &&
\big\langle 1_{k+1}^{2n}-1_k^{2n}, 1_{j+1}^n-1_j^n\big\rangle_\mfh
\\&=&
\langle 1_{k+1}^{2n},1_{j+1}^n\rangle_\mfh
+\langle 1_k^{2n},1_j^n\rangle_\mfh
-\langle 1_{k+1}^{2n},1_j^n\rangle_\mfh
-\langle 1_k^{2n},1_{j+1}^n\rangle_\mfh
\\&=&
\frac{T^{2H}}{2(2n)^{2H}}\bigg\{
|k-2j-2|^{2H}+|k-2j+1|^{2H}-|k-2j|^{2H}-|k-2j-1|^{2H}
\\&&
+|k-2j-1|^{2H}+|k-2j+2|^{2H}-|k-2j+1|^{2H}-|k-2j|^{2H}
\\&&
-|k-2j|^{2H}-|k-2j+3|^{2H}+|k-2j+1|^{2H}+|k-2j+2|^{2H}
\\&&
-|k-2j-3|^{2H}-|k-2j|^{2H}+|k-2j-2|^{2H}+|k-2j-1|^{2H}\bigg\}
\\&=&
\frac{T^{2H}}{2(2n)^{2H}}\bigg\{
-|k-2j-3|^{2H}+2|k-2j-2|^{2H}+|k-2j-1|^{2H}-4|k-2j|^{2H}+|k-2j+1|^{2H}
\\&&
+2|k-2j+2|^{2H}-|k-2j+3|^{2H}\bigg\}
\\&=&
\frac{T^{2H}}{n^{2H}}\wt{\rho}(k,j)
\eeas
This shows (\ref{2108111328}). 
By using (\ref{2108101801}) and (\ref{2108111328}), we obtain
\beas 
\langle f_n^{[2]},f_n^{[1]}\rangle_{\mfh^{\otimes2}}
&=&
\frac{n}{E[V^{(2)}_{2n,T}]E[V^{(2)}_{n,T}]}
\sum_{k=1}^{2n-1}
\sum_{j=1}^{n-1}
\big\langle (1^{2n}_{k+1}-1^{2n}_k)^{\otimes2},\>
(1^{n}_{j+1}-1^{n}_j)^{\otimes2}\big\rangle_{\mfh^{\otimes2}}
\\&=&
\frac{n}{E[V^{(2)}_{2n,T}]E[V^{(2)}_{n,T}]}
\sum_{k=1}^{2n-1}
\sum_{j=1}^{n-1}
\big\langle 1^{2n}_{k+1}-1^{2n}_k,\>1^{n}_{j+1}-1^{n}_j\big\rangle_{\mfh}^2
\\&=&
\frac{n}{E[V^{(2)}_{2n,T}]E[V^{(2)}_{n,T}]}\l(\frac{T^{2H}}{n^{2H}}\r)^2
\>\sum_{k=1}^{2n-1}
\sum_{j=1}^{n-1}\wt{\rho}(k,j)^2
\eeas
\beas
&=&
\frac{n}{
(4-2^{2H})^2(T/n)^{2H}(T/(2n))^{2H}(n-1)(2n-1)
}\l(\frac{T^{2H}}{n^{2H}}\r)^2
\>\sum_{k=1}^{2n-1}
\sum_{j=1}^{n-1}\wt{\rho}(k,j)^2
\\&=&
\frac{2^{2H}n}{
(4-2^{2H})^2(n-1)(2n-1)
}
\>\sum_{k=1}^{2n-1}
\sum_{j=1}^{n-1}\wt{\rho}(k-2j)^2
\\&=&
\frac{2^{2H}n}{
(4-2^{2H})^2(n-1)(2n-1)
}
\>\sum_{\ell=-2n+3}^{2n-3}
\grnb{\rbr{n-1-\lfloor\abs{\ell}/2\rfloor}}
\wt{\rho}(\ell)^2
\\&=&
\frac{2^{2H}}{
2(4-2^{2H})^2
}
\>\sum_{\ell=-\infty}^{\infty}
\wt{\rho}(\ell)^2
+O(n^{-1})
\\&=&
\half\Sigma_{12}+O(n^{-1})
\eeas
with the help of Lemma \ref{2108121121}. 
\begin{en-text}
Lemma \ref{2108111404} is applied to show  {\colorr ここ直接やって$n^{-1}$出す}
\beas
\frac{1}{n}
\sum_{k=1}^{2n-1}
\sum_{j=1}^{n-1}\wt{\rho}(k-2j)^2
&=&
\frac{1}{n}\sum_{k=1}^{n-1}
\sum_{j=1}^{n-1}\wt{\rho}(2k-2j)^2
+\frac{1}{n}\sum_{k=1}^{n}
\sum_{j=1}^{n-1}\wt{\rho}(2k-2j-1)^2
\\&\to&
\sum_{i\in\bbZ}\wt{\rho}(2i)^2+\sum_{i\in\bbZ}\wt{\rho}(2i-1)^2
\\&=&
\sum_{i\in\bbZ}\wt{\rho}(i)^2
\eeas
as $n\to\infty$. 
Therefore
\beas 
\langle f_n^{[2]},f_n^{[1]}\rangle_{\mfh^{\otimes2}}
&=&
\frac{2^{2H}}{
2(4-2^{2H})^2
}
\>\sum_{\ell=-\infty}^{\infty}
\wt{\rho}(\ell)^2
+o(1)
\\&=&
\half\Sigma_{12}+O(n^{-1}).
\eeas
\end{en-text}
\qed

Let 
\bea\label{2108211136}
G_\infty
\yeq
\Sigma_{22}-2\Sigma_{12}+\Sigma_{11}
\yeq 
\frac{3}{2}\Sigma_{11}-2\Sigma_{12}.
\eea

\begin{en-text}
\noindent
{\sred Question. $G_\infty>0$?
$G_\infty=0$ is equivalent to 
\beas 
1\yeq\frac{\Sigma_{12}}{\Sigma_{11}^{1/2}\Sigma_{22}^{1/2}}
\yeq
\frac{\frac{2^{2H}}{(4-2^{2H})^2}\sum_{i\in\bbZ}\wt{\rho}(i)^2}
{\sqrt{2}\bigg[1+\frac{2}{(4-2^{2H})^2}\sum_{i=1}^\infty\wh{\rho}(i)^2\bigg]}
\yeq
\frac{2^{2H}\sum_{i\in\bbZ}\wt{\rho}(i)^2}
{\sqrt{2}\sum_{i\in\bbZ}\wh{\rho}(i)^2}.
\eeas}

\noindent{\bf \colorr We will assume that $G_\infty>0$ in what follows. }
\end{en-text}

\subsection{A central limit theorem toward the asymptotic expansion}
From (\ref{2108121130}), Lemma \ref{2108110931} (c), (d) and Lemma \ref{2108111311} (b), we have 
\bea\label{2108121205}
G_n 
&=& 
\wt{G}_n+G_\infty+O(n^{-1})
\eea
where 
\bea\label{2108121206}
\wt{G}_n
&=& 
2I_2\big(f_n^{[2]}\otimes_1f_n^{[2]}\big)
-4I_2\big(f_n^{[2]}\odot_1f_n^{[1]}\big)
+2I_2\big(f_n^{[1]}\otimes_1f_n^{[1]}\big). 
\eea
\begin{en-text}
{\colorg Remove 
\beas
G_n 
&=& 
2I_2\big(f_n^{[2]}\otimes_1f_n^{[2]}\big)+2\|f_n^{[2]}\|_{\mfh^{\otimes2}}^2
-4I_2\big(f_n^{[2]}\odot_1f_n^{[1]}\big)-4\langle f_n^{[2]},f_n^{[1]}\rangle_{\mfh^{\otimes2}}
\nn\\&&
+2I_2\big(f_n^{[1]}\otimes_1f_n^{[1]}\big)+2\|f_n^{[1]}\|_{\mfh^{\otimes2}}^2
\eeas
}
\end{en-text}

Write $M_n^{[i]}=I_2(f_n^{[i]})$ for $i=1,2$. 
We will derive a central limit theorem for the sequence 
\beas 
U_n&=& \big(M_n^{[1]},M_n^{[2]},n^{1/2}\wt{G}_n\big).
\eeas
For $f,g\in\mfh^{\odot2}$, symmetric tensors, we have 
\beas 
E\big[I_2(f)I_2(g)\big]
&=& 
2f\odot_2g
\yeq 
2\langle f,g\rangle_{\mfh^{\otimes2}}
\eeas
Thus, 
\bea\label{2108140111} 
E\big[M_n^{[\alpha]}\wt{G}_n\big]
&=&
E\big[I_2(f_n^{[\alpha]})\big\{
2I_2\big(f_n^{[2]}\otimes_1f_n^{[2]}\big)
-4I_2\big(f_n^{[2]}\odot_1f_n^{[1]}\big)
+2I_2\big(f_n^{[1]}\otimes_1f_n^{[1]}\big)\big\}
\big]
\nn\\&=&
{\colorr 4}\big\langle f_n^{[\alpha]},f_n^{[2]}\otimes_1f_n^{[2]}\big\rangle_{\mfh^{\otimes2}}
-{\colorr 8}\big\langle f_n^{[\alpha]},f_n^{[2]}\odot_1f_n^{[1]}\big\rangle_{\mfh^{\otimes2}}
+{\colorr 4}\big\langle f_n^{[\alpha]},f_n^{[1]}\otimes_1f_n^{[1]}\big\rangle_{\mfh^{\otimes2}}
\eea
for $\alpha=1,2$. 

\subsubsection{Cubic formulas}
Define $\kappa(\alpha_1;\alpha_2,\alpha_2)$ ($\alpha_1,\alpha_2\in\{1,2\}$) as follows. 
\beas 
\kappa(1;1,1)
&=& 
\frac{1}{\big(4-2^{2H}\big)^3}
\sum_{i_1,i_2\in\bbZ}
\wh{\rho}(i_1)\wh{\rho}(i_1-i_2)\wh{\rho}(i_2),
\eeas
\beas 
\kappa(2;1,1)
&=&
\frac{2^{2H-1}}{\big(4-2^{2H}\big)^3}
\sum_{i_1,i_2\in\bbZ}
\wh{\rho}(i_1)\wt{\rho}(i_2-2i_1)\wt{\rho}(i_2),
\eeas
\begin{en-text}
\beas
\kappa(1;2,2)
&=&
\frac{2^{2H-2}}{\big(4-2^{2H}\big)^3}\bigg\{
\sum_{i_1,i_2\in\bbZ}\wh{\rho}(2i_1)\wt{\rho}(2i_1-2i_2)\wt{\rho}(2i_2)
+\sum_{i_1,i_2\in\bbZ}\wh{\rho}(2i_1+1)\wt{\rho}(2i_1-2i_2-1)\wt{\rho}(2i_2)
\\&&
+\sum_{i_1,i_2\in\bbZ}\wh{\rho}(2i_1-1)\wt{\rho}(2i_1-2i_2)\wt{\rho}(2i_2+1)
+\sum_{i_1,i_2\in\bbZ}\wh{\rho}(2i_1)\wt{\rho}(2i_1-2i_2-1)\wt{\rho}(2i_2+1)\bigg\}
\eeas
\end{en-text}
\beas
\kappa(1;2,2)
&=&
\bluc{\frac{2^{2H-2}}{\big(4-2^{2H}\big)^3}\bigg\{\sum_{i_1,i_2\in\bbZ}\wh{\rho}(2i_1)\wt{\rho}(2i_1-i_2)\wt{\rho}(i_2)
+\sum_{i_1,i_2\in\bbZ}\wh{\rho}(2i_1{\colorr+}1)\wt{\rho}(2i_1-i_2)\wt{\rho}(i_2+1)\bigg\}}
\\&=&
\grnc{\frac{2^{2H-2}}{\big(4-2^{2H}\big)^3}\sum_{i_1,i_2\in\bbZ}\wh{\rho}(i_1)\wt{\rho}(i_1-i_2)\wt{\rho}(i_2)}
\eeas
and
\beas
\kappa(2;2,2)\yeq \frac{1}{8}\kappa(1;1,1).
\eeas

\begin{lemma}\label{2108121234}
\bd\im[(a)] 
$\ds 
\big\langle f_n^{[1]},n^{1/2}f_n^{[1]}\otimes_1f_n^{[1]}\big\rangle_{\mfh^{\otimes2}}
\yeq
\kappa(1;1,1)+o(1)
$.
\im[(b)]
$\ds
\big\langle f_n^{[2]},n^{1/2}f_n^{[1]}\otimes_1f_n^{[1]}\big\rangle_{\mfh^{\otimes2}}
\yeq
\kappa(2;1,1)+o(1)
$.
\im[(c)] 
$\ds 
\big\langle f_n^{[1]},n^{1/2}f_n^{[2]}\otimes_1f_n^{[2]}\big\rangle_{\mfh^{\otimes2}}
\yeq
\kappa(1;2,2)+o(1)
$.
\im[(d)] 
$\ds 
\big\langle f_n^{[2]},n^{1/2}f_n^{[2]}\otimes_1f_n^{[2]}\big\rangle_{\mfh^{\otimes2}}
\yeq
\kappa(2;2,2)+o(1)
$.
\ed
\end{lemma}
\proof 
{\sblue (a):} 
By (\ref{2108121842}) and Lemma \ref{2108110931} (a), (b), we have 
\bea\label{2108121817}
f_n^{[\alpha]}\otimes_1f_n^{[\alpha]}
&=&
\frac{n\sum_{j=1}^{\alpha n-1}\sum_{j'=1}^{\alpha n-1}
(1^{\alpha n}_{j+1}-1^{\alpha n}_j)\odot(1^{\alpha n}_{j'+1}-1^{\alpha n}_{j'})}{E[V^{(2)}_{\alpha n,T}]^2}
\langle 1^{\alpha n}_{j+1}-1^{\alpha n}_j,1^{\alpha n}_{j'+1}-1^{\alpha n}_{j'}\rangle_\mfh
\nn\\&=&
\frac{n}{E[V^{(2)}_{\alpha n,T}]^2}\frac{T^{2H}}{(\alpha n)^{2H}}
\sum_{j=1}^{\alpha n-1}\sum_{j'=1}^{\alpha n-1}
(1^{\alpha n}_{j+1}-1^{\alpha n}_j)\odot(1^{\alpha n}_{j'+1}-1^{\alpha n}_{j'})
\wh{\rho}(j-j')
\eea
for $\alpha=1,2$. 
\begin{en-text}
{\colorg 
where 
\beas 
\gamma^{11}(j,,j')
&=& 
\langle 1^{n}_{j+1}-1^{n}_j,1^{n}_{j'+1}-1^{n}_{j'}\rangle_\mfh
\yeq
\frac{T^{2H}}{n^{2H}}\wh{\rho}(j-j')
\eeas
}
\end{en-text}
Therefore, from (\ref{2108121842}), (\ref{2108121817}) and (\ref{2108101801}), 
we obtain 
\beas &&
\big\langle f_n^{[1]},n^{1/2}f_n^{[1]}\otimes_1f_n^{[1]}\big\rangle_{\mfh^{\otimes2}}
\\&=&
\frac{n^2}{E[V^{(2)}_{n,T}]^3}\frac{T^{2H}}{n^{2H}}
\sum_{j,j',j''=1}^{n-1}
\langle 1^{n}_{j''+1}-1^{n}_{j''}, 1^{n}_{j+1}-1^{n}_j\rangle_\mfh
\langle 1^{n}_{j''+1}-1^{n}_{j''}, 1^{n}_{j'+1}-1^{n}_{j'}\rangle_\mfh
\wh{\rho}(j-j')
\\&=&
\frac{n^{2}}{\big(4-2^{2H}\big)^3(n-1)^3}
\sum_{j=1}^{n-1}\sum_{j'=1}^{n-1}\sum_{j''=1}^{n-1}
\wh{\rho}(j-j')\wh{\rho}(j'-j'')\wh{\rho}(j''-j)
\begin{en-text}
\\&=&
\frac{n^{2}}{\big(4-2^{2H}\big)^3(n-1)^3}
\sum_{i=-n+2}^{n-2}\sum_{i'=-n+2}^{n-2}
\wh{\rho}(i)\wh{\rho}(i')\wh{\rho}(i-i')
\\&=&
\frac{n^{3}}{\big(4-2^{2H}\big)^3(n-1)^3}
\sum_{i=-n+2}^{n-2}\sum_{i'=-n+2}^{n-2}
\wh{\rho}(i)\wh{\rho}(i')\wh{\rho}(i-i')
\\&&\times
\frac{\min\{n-1-i,n-1-i',n-1\}-\max\{1-i,1-i',1\}}{n}
\end{en-text}
\\&=&
\kappa(1;1,1)+o(1)
\eeas
thanks to Lemma \ref{2108111404}. 
\begin{en-text}
Remark
\beas &&
\sum_{j=1}^{n-1}\sum_{j'=1}^{n-1}\sum_{j''=1}^{n-1}
\wh{\rho}(j-j')\wh{\rho}(j'-j'')\wh{\rho}(j''-j)
\\&=&
\sum_{j=1}^{n-1}\sum_{i=-n+2}^{n-2}\sum_{i'=-n+2}^{n-2}
\wh{\rho}(i)\wh{\rho}(i')\wh{\rho}(i-i')
1_{\{1-j\leq i\leq n-1-j\}}1_{\{1-j\leq i'\leq n-1-j\}}
\\&=&
\sum_{i=-n+2}^{n-2}\sum_{i'=-n+2}^{n-2}
\wh{\rho}(i)\wh{\rho}(i')\wh{\rho}(i-i')
\\&&\times
\bigg(\min\{n-1-i,n-1-i',n-1\}-\max\{1-i,1-i',1\}\bigg)
\eeas
Remark
\beas 
\sum_{i,i'\in\bbZ}|\wh{\rho}(i)||\wh{\rho}(i')||\wh{\rho}(i-i')|
&=&
\sum_{i\in\bbZ}|\wh{\rho}(i)|(|\wh{\rho}|*|\wh{\rho}|)(i)|
\\&\leq&
\|\wh{\rho}\|_{\ell_{\frac{3}{2}}}^3
\simleq
\bigg(\sum_{i\in\bbZ}(|i|+1)^{\frac{3}{2}(2H-4)}\bigg)^2
\><\>\infty
\eeas
\beas
\gamma^{12}(j,k)
&=&
\langle 1^{n}_{j+1}-1^{n}_j,1^{2n}_{k+1}-1^{2n}_{k}\rangle_\mfh
\yeq
\frac{T^{2H}}{n^{2H}}\wt{\rho}(k,j)
\eeas
\end{en-text}

\halflineskip\noindent
{\sblue (b): 
Similarly to the proof of (a),} 
\bea\label{2108130419}
\big\langle f_n^{[2]},n^{1/2}f_n^{[1]}\otimes_1f_n^{[1]}\big\rangle_{\mfh^{\otimes2}}
&=&
\frac{n^{2}}{E[V^{(2)}_{n,T}]^2E[V^{(2)}_{2n,T}]}
\sum_{k=1}^{2n-1}\sum_{j=1}^{n-1}\sum_{j'=1}^{n-1}
\langle 1^{2n}_{k+1}-1^{2n}_k,1^{n}_{j+1}-1^{n}_{j}\rangle_\mfh
\nn\\&&\hspace{140pt}\times
\langle 1^{2n}_{k+1}-1^{2n}_k,1^{n}_{j'+1}-1^{n}_{j'}\rangle_\mfh
\wh{\rho}(j-j')
\nn\\&=&
\frac{2^{2H}n^{2}}{\big(4-2^{2H}\big)^3
(n-1)^2(2n-1)}
\sum_{k=1}^{2n-1}\sum_{j=1}^{n-1}\sum_{j'=1}^{n-1}
\wh{\rho}(j-j')\wt{\rho}(2j'-k)\wt{\rho}(k-2j)
\nn\\&=&
\frac{2^{2H-1}}{\big(4-2^{2H}\big)^3}
\sum_{i\in\bbZ}\sum_{j\in\bbZ}
\wh{\rho}(i)\wt{\rho}(j)\wt{\rho}(j-2i)+o(1)
\nn\\&=&
\kappa(2;1,1)+o(1).
\eea
Ideed, 
\beas &&
n^{-1}\sum_{k=1}^{2n-1}\sum_{j=1}^{n-1}\sum_{j'=1}^{n-1}
\wh{\rho}(j-j')\wt{\rho}(2j'-k)\wt{\rho}(k-2j)
\\&=&
n^{-1}\sum_{k=1}^{n-1}\sum_{j=1}^{n-1}\sum_{j'=1}^{n-1}
\wh{\rho}(j-j')\wt{\rho}(2j'-2k)\wt{\rho}(2k-2j)
\\&&
+
n^{-1}\sum_{k=1}^{n}\sum_{j=1}^{n-1}\sum_{j'=1}^{n-1}
\wh{\rho}(j-j')\wt{\rho}(2j'-2k+1)\wt{\rho}(2k-2j-1)
\\&=&
\sum_{i\in\bbZ}\sum_{j\in\bbZ}
\wh{\rho}(i)\wt{\rho}(2i-2j)\wt{\rho}(2j)
+
\sum_{i\in\bbZ}\sum_{j\in\bbZ}
\wh{\rho}(i)\wt{\rho}(2i-2j+1)\wt{\rho}(2j-1)+o(1)
\\&=&
\sum_{i\in\bbZ}\sum_{j\in\bbZ}
\wh{\rho}(i)\wt{\rho}(2i-j)\wt{\rho}(j)+o(1)
\eeas
by Lemma \ref{2108111404}. 
\begin{en-text}
since 
\beas
\sum_{i\in\bbZ}\sum_{j\in\bbZ}
|\wh{\rho}(i)||\wt{\rho}(j)||\wt{\rho}(j-2i)|
&=&
\sum_{i\in\bbZ}(|\wt{\rho}|*|\wt{\rho}|)(2i)|\wh{\rho}(i)|
\><\>\infty
\eeas
Indeed, 
\beas &&
\sum_{k=1}^{2n-1}\sum_{j=1}^{n-1}\sum_{j'=1}^{n-1}
\wh{\rho}(j-j')\wt{\rho}(k-2j)\wt{\rho}(k-2j')
\\&=&
\sum_{k=1}^{2n-1}\sum_{j=1}^{n-1}\sum_{j'=1}^{n-1}
\bigg(\sum_{p=-n+2}^{n-2}\sum_{q=-2n+3}^{2n-3}1_{\{j-j'=p\}}1_{\{k-2j'=q\}}\bigg)
\wh{\rho}(j-j')\wt{\rho}(k-2j)\wt{\rho}(k-2j')
\\&=&
\sum_{p=-n+2}^{n-2}\sum_{q=-2n+3}^{2n-3}
\sum_{k=1}^{2n-1}\sum_{j=1}^{n-1}\sum_{j'=1}^{n-1}
1_{\{j-j'=p\}}1_{\{k-2j'=q\}}
\wh{\rho}(j-j')\wt{\rho}(k-2j)\wt{\rho}(k-2j')
\\&=&
\sum_{p=-n+2}^{n-2}\sum_{q=-2n+3}^{2n-3}
\sum_{k=1}^{2n-1}\sum_{j=1}^{n-1}\sum_{j'=1}^{n-1}
1_{\{j-j'=p\}}1_{\{k-2j'=q\}}
\wh{\rho}(p)\wt{\rho}(q)\wt{\rho}(q-2p)
\\&=&
\sum_{p=-n+2}^{n-2}\sum_{q=-2n+3}^{2n-3}
\bigg[
\sum_{j=1}^{n-1}\sum_{j'=1}^{n-1}
1_{\{j-j'=p\}}\bigg(\sum_{k=1}^{2n-1}1_{\{k-2j'=q\}}\bigg)\bigg]
\wh{\rho}(p)\wt{\rho}(q)\wt{\rho}(q-2p)
\\&=&
\sum_{p=-n+2}^{n-2}\sum_{q=-2n+3}^{2n-3}
\bigg[
\sum_{j=1}^{n-1}\sum_{j'=1}^{n-1}
1_{\{j-j'=p\}}1_{\{1\leq2j'+q\leq2n-1\}}\bigg]
\wh{\rho}(p)\wt{\rho}(q)\wt{\rho}(q-2p)
\\&=&
\sum_{p=-n+2}^{n-2}\sum_{q=-2n+3}^{2n-3}
\bigg[
\sum_{j'=1}^{n-1}\bigg(\sum_{j=1}^{n-1}
1_{\{j-j'=p\}}\bigg)1_{\{1\leq2j'+q\leq2n-1\}}\bigg]
\wh{\rho}(p)\wt{\rho}(q)\wt{\rho}(q-2p)
\\&=&
\sum_{p=-n+2}^{n-2}\sum_{q=-2n+3}^{2n-3}
\bigg[
\sum_{j'=1}^{n-1}
1_{\{1\leq j'+p\leq n-1\}}1_{\{1\leq2j'+q\leq2n-1\}}\bigg]
\wh{\rho}(p)\wt{\rho}(q)\wt{\rho}(q-2p)
\\&=&
\sum_{p=-n+2}^{n-2}\sum_{q=-2n+3}^{2n-3}
\bigg[
\min\big\{n-1,n-p-1,\lfloor n-2^{-1}(q+1)\rfloor\big\}
-\max\big\{1,1-p,\lceil 2^{-1}(1-q)\rceil\big\}
\bigg]
\\&&\hspace{80pt}\times
\wh{\rho}(p)\wt{\rho}(q)\wt{\rho}(q-2p)
\eeas
\end{en-text}

\halflineskip\noindent
{\sblue (c):} 
By using (\ref{2108121842}), (\ref{2108121817}) and (\ref{2108101801}), 
we obtain 
\bea\label{2108130427}
\big\langle f_n^{[1]},n^{1/2}f_n^{[2]}\otimes_1f_n^{[2]}\big\rangle_{\mfh^{\otimes2}}
&=&
\frac{n^2}{E[V^{(2)}_{n,T}]E[V^{(2)}_{2n,T}]^2}\frac{T^{2H}}{(2n)^{2H}}
\nn\\&&\times
\sum_{j=1}^{n-1}\sum_{k,k'=1}^{2n-1}
\langle 1^{n}_{j+1}-1^{n}_{j}, 1^{2n}_{k+1}-1^{n}_k\rangle_\mfh
\langle 1^{n}_{j+1}-1^{n}_{j}, 1^{2n}_{k'+1}-1^{n}_{k'}\rangle_\mfh
\wh{\rho}(k-k')
\nn\\&=&
\frac{n^{2}}{\big(4-2^{2H}\big)^3(T/n)^{2H}(T/(2n))^{4H}(n-1)(2n-1)^2}
\frac{T^{2H}}{(2n)^{2H}}\l(\frac{T^{2H}}{n^{2H}}\r)^2
\nn\\&&\times
\sum_{j=1}^{n-1}\sum_{k,k'=1}^{2n-1}
\wh{\rho}(k-k')\wt{\rho}(k'-2j)\wt{\rho}(2j-k)
\nn\\&=&
\frac{2^{2H-2}}{\big(4-2^{2H}\big)^3n}
\sum_{j=1}^{n-1}\sum_{k,k'=1}^{2n-1}
\wh{\rho}(k-k')\wt{\rho}(k'-2j)\wt{\rho}(2j-k)+o(1)
\nn\\&=&
\kappa(1;2,2)+o(1).
\eea
In fact, thanks to Lemma \ref{2108111404}, 
\beas &&
n^{-1}\sum_{j=1}^{n-1}\sum_{k,k'=1}^{2n-1}
\wh{\rho}(k-k')\wt{\rho}(k'-2j)\wt{\rho}(2j-k)
\\&=&
n^{-1}\sum_{j=1}^{n-1}\sum_{k=1}^{n-1}\sum_{k'=1}^{n-1}
\wh{\rho}(2k-2k')\wt{\rho}(2k'-2j)\wt{\rho}(2j-2k)
\\&&
+n^{-1}\sum_{j=1}^{n-1}\sum_{k=1}^{n-1}\sum_{k'=1}^{n}
\wh{\rho}(2k-2k'+1)\wt{\rho}(2k'-2j-1)\wt{\rho}(2j-2k)
\\&&
+n^{-1}\sum_{j=1}^{n-1}\sum_{k=1}^{n}\sum_{k'=1}^{n-1}
\wh{\rho}(2k-2k'-1)\wt{\rho}(2k'-2j)\wt{\rho}(2j-2k+1)
\\&&
+n^{-1}\sum_{j=1}^{n-1}\sum_{k=1}^{n}\sum_{k'=1}^{n}
\wh{\rho}(2k-2k')\wt{\rho}(2k'-2j-1)\wt{\rho}(2j-2k+1)
\eeas
\bluc{\beas
&=&
\sum_{i_1,i_2\in\bbZ}\wh{\rho}(2i_1)\wt{\rho}(2i_1-2i_2)\wt{\rho}(2i_2)
+\sum_{i_1,i_2\in\bbZ}\wh{\rho}(2i_1{\colorr-}1)\wt{\rho}(2i_1-2i_2-1)\wt{\rho}(2i_2)
\\&&
+\sum_{i_1,i_2\in\bbZ}\wh{\rho}(2i_1{\colorr+}1)\wt{\rho}(2i_1-2i_2)\wt{\rho}(2i_2+1)
+\sum_{i_1,i_2\in\bbZ}\wh{\rho}(2i_1)\wt{\rho}(2i_1-2i_2-1)\wt{\rho}(2i_2+1)
+o(1)
\eeas}
\grnc{
\beas
&=&
\sum_{i_1,i_2\in\bbZ}\wh{\rho}(2i_1)\wt{\rho}(2i_1-2i_2)\wt{\rho}(2i_2)
+\sum_{i_1,i_2\in\bbZ}\wh{\rho}(2i_1-1)\wt{\rho}((2i_1-1)-2i_2)\wt{\rho}(2i_2)+o(1)
\\&&
+\sum_{i_1,i_2\in\bbZ}\wh{\rho}(2i_1+1)\wt{\rho}((2i_1+1)-2i_2-1)\wt{\rho}(2i_2+1)
+\sum_{i_1,i_2\in\bbZ}\wh{\rho}(2i_1)\wt{\rho}(2i_1-2i_2-1)\wt{\rho}(2i_2+1)
+o(1)
\\&=&
\sum_{i_1,i_2\in\bbZ}\wh{\rho}(i_1)\wt{\rho}(i_1-2i_2)\wt{\rho}(2i_2)
+\sum_{i_1,i_2\in\bbZ}\wh{\rho}(i_1)\wt{\rho}(i_1-2i_2-1)\wt{\rho}(2i_2+1)
+o(1)
\\&=&
\sum_{i_1,i_2\in\bbZ}\wh{\rho}(i_1)\wt{\rho}(i_1-i_2)\wt{\rho}(i_2)
+o(1).
\eeas
}

\noindent
{\sblue (d):} 
Proof of (d) is similar to that of (a). 
\qed\halflineskip

\begin{lemma}\label{2108130356}
	\bd
	\im[(a)] 
	$\ds 
	n^{1/2}\big\langle f_n^{[1]},f_n^{[2]}\odot_1f_n^{[1]}\big\rangle_{\mfh^{\otimes2}}
	\yeq \kappa(2;1,1)+o(1)$. 
	\im[(b)] 
	$\ds
	n^{1/2}\big\langle f_n^{[2]},f_n^{[2]}\odot_1f_n^{[1]}\big\rangle_{\mfh^{\otimes2}}
	\yeq \kappa(1;2,2)+o(1)$. 
	\ed
\end{lemma}
\begin{proof}
We have
\begin{align*}
	\big\langle f_n^{[1]},f_n^{[2]}\odot_1f_n^{[1]}\big\rangle_{\mfh^{\otimes2}}
	&=
	\big\langle f_n^{[1]},f_n^{[2]}\otimes_1f_n^{[1]}\big\rangle_{\mfh^{\otimes2}}
	=
	\big\langle f_n^{[2]},f_n^{[1]}\otimes_1f_n^{[1]}\big\rangle_{\mfh^{\otimes2}}.
\end{align*}
The first equality stands since $f_n^{[1]}$ is symmetric 
(i.e. $\in\calh^{\odot2}$).
Hence we have (a) by Lemma \ref{2108121234} (b).
Similar arguments and Lemma \ref{2108121234} (c) prove (b).
\end{proof}
\begin{en-text}
\redc{
Let
\beas
\kappa(1;2,1)\yeq\kappa(2;1,1)\quad\text{and}\quad
\kappa(2;2,1)\yeq\kappa(1;2,2)
\eeas
\begin{lemma}
\bd
\im[(a)] 
$\ds 
n^{1/2}\big\langle f_n^{[1]},f_n^{[2]}\odot_1f_n^{[1]}\big\rangle_{\mfh^{\otimes2}}
\yeq \kappa(1;2,1)+o(1)$. 
\im[(b)] 
$\ds
n^{1/2}\big\langle f_n^{[2]},f_n^{[2]}\odot_1f_n^{[1]}\big\rangle_{\mfh^{\otimes2}}
\yeq \kappa(2;2,1)+o(1)$. 
\ed
\end{lemma}}
\proof
\redc{
From (\ref{2108121842}) and (\ref{2108111328}), we have
\beas 
f_n^{[2]}\otimes_1f_n^{[1]}
&=&
\frac{n}{E[V^{(2)}_{n,T}]E[V^{(2)}_{2n,T}]}
\sum_{j=1}^{n-1}\sum_{k=1}^{2n-1}
(1_{j+1}^n-1_j^n)\otimes(1_{k+1}^{2n}-1_k^{2n})
\langle 1_{j+1}^n-1_j^n,1_{k+1}^{2n}-1_k^{2n}\rangle_\mfh
\\&=&
\frac{n}{E[V^{(2)}_{n,T}]E[V^{(2)}_{2n,T}]}\frac{T^{2H}}{n^{2H}}
\sum_{j=1}^{n-1}\sum_{k=1}^{2n-1}
(1_{j+1}^n-1_j^n)\otimes(1_{k+1}^{2n}-1_k^{2n})
\wt{\rho}(k,j).
\eeas
Then we obtain 
\beas &&
n^{1/2}\big\langle f_n^{[1]},f_n^{[2]}\odot_1f_n^{[1]}\big\rangle_{\mfh^{\otimes2}}
\\&=&
n^{1/2}\bigg\langle 
\frac{\sqrt{n}\sum_{j=1}^{n-1}(1^{n}_{j+1}-1^{n}_j)^{\otimes2}}
{E[V^{(2)}_{n,T}]},\>f_n^{[2]}\otimes_1f_n^{[1]}
\bigg\rangle_{\mfh^{\otimes2}}
\\&=&
\frac{n^2}{E[V^{(2)}_{n,T}]^2E[V^{(2)}_{2n,T}]}\l(\frac{T^{2H}}{n^{2H}}\r)^3
\sum_{j=1}^{n-1}\sum_{j'=1}^{n-1}\sum_{k=1}^{2n-1}
\wh{\rho}(j-j')\wt{\rho}(k-2j')\wt{\rho}(k-2j)
\\&=&
\frac{n^2}{\big(4-2^{2H}\big)^2(T/n)^{4H}(n-1)^2\big(4-2^{2H}\big)(T/(2n))^{2H}(2n-1)}\l(\frac{T^{2H}}{n^{2H}}\r)^3
\\&&\times
\sum_{j=1}^{n-1}\sum_{j'=1}^{n-1}\sum_{k=1}^{2n-1}
\wh{\rho}(j-j')\wt{\rho}(k-2j')\wt{\rho}(k-2j)
\\&=&
\frac{2^{2H}n^2}{\big(4-2^{2H}\big)^3(n-1)^2(2n-1)}
\sum_{j=1}^{n-1}\sum_{j'=1}^{n-1}\sum_{k=1}^{2n-1}
\wh{\rho}(j-j')\wt{\rho}(k-2j')\wt{\rho}(k-2j)
\\&=&
\kappa(1;2,1)+o(1)
\eeas
by (\ref{2108130419}).
We have
\beas &&
n^{1/2}\big\langle f_n^{[2]},f_n^{[2]}\odot_1f_n^{[1]}\big\rangle_{\mfh^{\otimes2}}
\\&=&
n^{1/2}\big\langle f_n^{[2]},f_n^{[2]}\otimes_1f_n^{[1]}\big\rangle_{\mfh^{\otimes2}}
\\&=&
\frac{n^2}{E[V^{(2)}_{n,T}]E[V^{(2)}_{2n,T}]^2}
\frac{T^{2H}}{(2n)^{2H}}\l(\frac{T^{2H}}{n^{2H}}\r)^2
\sum_{j=1}^{n-1}\sum_{k=1}^{2n-1}\sum_{k'=1}^{2n-1}
\wh{\rho}(k-k')\wt{\rho}(k-2j)\wt{\rho}(k'-2j)
\\&=&
\frac{n^2}{\big(4-2^{2H}\big)(T/n)^{2H}(n-1)\big(4-2^{2H}\big)^2(T/(2n))^{4H}(2n-1)^2}\frac{T^{2H}}{(2n)^{2H}}\l(\frac{T^{2H}}{n^{2H}}\r)^2
\\&&\times
\sum_{j=1}^{n-1}\sum_{k=1}^{2n-1}\sum_{k'=1}^{2n-1}
\wh{\rho}(k-k')\wt{\rho}(k-2j)\wt{\rho}(k'-2j)
\\&=&
\frac{2^{2H}n^2}{\big(4-2^{2H}\big)^3(n-1)(2n-1)^2}
\sum_{j=1}^{n-1}\sum_{k=1}^{2n-1}\sum_{k'=1}^{2n-1}
\wh{\rho}(k-k')\wt{\rho}(k-2j)\wt{\rho}(k'-2j)
\\&=&
\kappa(2;2,1)+o(1)
\eeas
by observing (\ref{2108130427}).}
\end{en-text}
\begin{en-text}
where
\beas 
\kappa(2)\yeq 2^{-1}\kappa(1)
&=&
\frac{2^{2H-2}}{\big(4-2^{2H}\big)^3}
\sum_{i\in\bbZ}\sum_{j\in\bbZ}
\wh{\rho}(i)\wt{\rho}(j)\wt{\rho}(j-2i)
\eeas
\beas &&
n^{1/2}\big\langle f_n^{[1]},f_n^{[2]}\otimes_1f_n^{[2]}\big\rangle_{\mfh^{\otimes2}}
\\&=&
\frac{2^{2H-2}}{\big(4-2^{2H}\big)^3}
\sum_{i\in\bbZ}\sum_{j\in\bbZ}
\wh{\rho}(i)\wt{\rho}(j)\wt{\rho}(j-2i)+o(1)
\\&=&
\kappa(2)+o(1)
\eeas
\end{en-text}
\halflineskip

\subsubsection{Fourth power formulas}
Recall
\beas
f_n^{[1]}\otimes_1f_n^{[1]}
&=&
\frac{n}{E[V^{(2)}_{n,T}]^2}\frac{T^{2H}}{n^{2H}}
\sum_{j=1}^{n-1}\sum_{j'=1}^{n-1}
(1^{n}_{j+1}-1^{n}_j)\odot(1^{n}_{j'+1}-1^{n}_{j'})
\wh{\rho}(j-j'),
\eeas
\beas 
f_n^{[2]}\otimes_1f_n^{[1]}
&=&
\frac{n}{E[V^{(2)}_{n,T}]E[V^{(2)}_{2n,T}]}\frac{T^{2H}}{n^{2H}}
\sum_{j=1}^{n-1}\sum_{k=1}^{2n-1}
(1_{j+1}^n-1_j^n)\otimes(1_{k+1}^{2n}-1_k^{2n})
\wt{\rho}(k,j)
\eeas
and
\beas 
f_n^{[2]}\otimes_1f_n^{[2]}
&=&
\frac{n}{E[V^{(2)}_{2n,T}]^2}\frac{T^{2H}}{(2n)^{2H}}
\sum_{j=1}^{2n-1}\sum_{k=1}^{2n-1}
(1_{k+1}^{2n}-1_k^{2n})\otimes(1_{k'+1}^{2n}-1_{k'}^{2n})
\wh{\rho}(k-k').
\eeas

Let
\beas 
\kappa(1,1;1,1)
&=&
\frac{1}{\big(4-2^{2H}\big)^4}
\sum_{i_1,i_2,i_3\in\bbZ}
\wh{\rho}(i_1)\wh{\rho}(i_1-i_2)\wh{\rho}(i_2-i_3)\wh{\rho}(i_3),
\\
\kappa(2,2;2,2)&=&\frac{1}{16}\kappa(1,1;1,1),
\eeas
\beas 
\kappa(1,1;2,2)
&=& 
\frac{2^{2H-2}}{\big(4-2^{2H}\big)^4}
\sum_{i_1,i_2,i_3}\bigg\{
\wt{\rho}(2i_1)\wh{\rho}(i_1-i_2)\wt{\rho}(2i_2-i_3)\wh{\rho}(i_3)
\\&&\hspace{100pt}
+\wt{\rho}(2i_1+1)\wh{\rho}(i_1-i_2)\wt{\rho}(2i_2-i_3{\colorr+}1)\wh{\rho}(i_3)\bigg\}.
\eeas

\begin{lemma}\label{2108130456}
\bd\im[(a)] 
$\ds n\big\langle f_n^{[1]}\otimes_1f_n^{[1]},\>
f_n^{[1]}\otimes_1f_n^{[1]}\big\rangle_{\mfh^{\otimes2}}
\yeq\kappa(1,1;1,1)+o(1)$. 
\im[(b)] 
$\ds n\big\langle f_n^{[1]}\otimes_1f_n^{[1]},\>
f_n^{[2]}\otimes_1f_n^{[2]}\big\rangle_{\mfh^{\otimes2}}
\yeq \kappa(1,1;2,2)+o(1)$.
\im[(c)]
$\ds 
n\big\langle f_n^{[2]}\otimes_1f_n^{[2]},\>
f_n^{[2]}\otimes_1f_n^{[2]}\big\rangle_{\mfh^{\otimes2}}
\yeq
\kappa(2,2;2,2)+o(1)$.
\ed
\end{lemma}
\proof
It follows from Lemma \ref{2108111404} that
\beas&&
n\big\langle f_n^{[1]}\otimes_1f_n^{[1]},\>
f_n^{[1]}\otimes_1f_n^{[1]}\big\rangle_{\mfh^{\otimes2}}
\\&=&
\frac{n^3}{E[V^{(2)}_{n,T}]^4}\l(\frac{T^{2H}}{n^{2H}}\r)^2
\sum_{j_1,j_2,j_3,j_4=1}^{n-1}
\big\langle 
(1^{n}_{j_1+1}-1^{n}_{j_1})\otimes(1^{n}_{j_4+1}-1^{n}_{j_4}),\>
(1^{n}_{j_2+1}-1^{n}_{j_2})\otimes(1^{n}_{j_3+1}-1^{n}_{j_3})\big\rangle_{\mfh^{\otimes2}}
\\&&\hspace{140pt}\times
\wh{\rho}(j_1-j_4)\wh{\rho}(j_2-j_3)
\\&=&
\frac{n^3}{E[V^{(2)}_{n,T}]^4}\l(\frac{T^{2H}}{n^{2H}}\r)^4
\sum_{j_1,j_2,j_3,j_4=1}^{n-1}
\wh{\rho}(j_1-j_2)\wh{\rho}(j_2-j_3)\wh{\rho}(j_3-j_4)\wh{\rho}(j_4-j_1)
\\&=&
\frac{n^3}{\big(4-2^{2H}\big)^4(T/n)^{8H}(n-1)^4}\l(\frac{T^{2H}}{n^{2H}}\r)^4
\sum_{j_1,j_2,j_3,j_4=1}^{n-1}
\wh{\rho}(j_1-j_2)\wh{\rho}(j_2-j_3)\wh{\rho}(j_3-j_4)\wh{\rho}(j_4-j_1)
\\&=&
\kappa(1,1;1,1)+o(1).
\eeas
\begin{en-text}
Remark
\beas &&
\sum_{i_1,i_2,i_3\in\bbZ}
|\wh{\rho}(i_1)||\wh{\rho}(i_1-i_2)||\wh{\rho}(i_2-i_3)||\wh{\rho}(i_3)|
\yleq
\sum_{j\in\bbZ}
(|\wh{\rho}|*|\wh{\rho}|*|\wh{\rho}|)(j)|\wh{\rho}|(j)
\\&\leq&
\|\wh{\rho}\|_{\ell_{\frac{4}{3}}}^4
\simleq
\bigg(\sum_{i\in\bbN}|i|^{\frac{4}{3}(2H-4)}\bigg)^3
<\infty
\eeas
\end{en-text}
{\sblue Thus, (a) has been obtained.}
Similarly, we can show (c).

Now we have
\beas&&
n\big\langle f_n^{[1]}\otimes_1f_n^{[1]},\>
f_n^{[2]}\otimes_1f_n^{[2]}\big\rangle_{\mfh^{\otimes2}}
\\&=&
\frac{n^3}{E[V^{(2)}_{n,T}]^2E[V^{(2)}_{2n,T}]^2}
\l(\frac{T^{2H}}{n^{2H}}\r)^3\frac{T^{2H}}{(2n)^{2H}}
\sum_{j=1}^{n-1}\sum_{j'=1}^{n-1}
\sum_{k=1}^{2n-1}\sum_{k'=1}^{2n-1}
\wt{\rho}(k-2j)\wt{\rho}(k'-2j')\wh{\rho}(j-j')\wh{\rho}(k-k')
\\&=&
\frac{n^3}{\big(4-2^{2H}\big)^4(T/n)^{4H}(n-1)^2(T/(2n))^{4H}(2n-1)^2}
\l(\frac{T^{2H}}{n^{2H}}\r)^3\frac{T^{2H}}{(2n)^{2H}}
\\&&\times
\sum_{j=1}^{n-1}\sum_{j'=1}^{n-1}
\sum_{k=1}^{2n-1}\sum_{k'=1}^{2n-1}
\wt{\rho}(k-2j)\wt{\rho}(k'-2j')\wh{\rho}(j-j')\wh{\rho}(k-k')
\\&=&
\frac{2^{2H}n^3}{\big(4-2^{2H}\big)^4(n-1)^2(2n-1)^2}
\sum_{j=1}^{n-1}\sum_{j'=1}^{n-1}
\sum_{k=1}^{2n-1}\sum_{k'=1}^{2n-1}
\wt{\rho}(k-2j)\wh{\rho}(j-j')\wt{\rho}(2j'-k')\wh{\rho}(k-k')
\\&=&
\kappa(1,1;2,2)+o(1)
\eeas
by the following argument: 
\beas &&
n^{-1}
\sum_{j=1}^{n-1}\sum_{j'=1}^{n-1}\sum_{k=1}^{2n-1}\sum_{k'=1}^{2n-1}
\wt{\rho}(k-2j)\wh{\rho}(j-j')\wt{\rho}(2j'-k')\wh{\rho}(k'-k)
\\&=&
n^{-1}\sum_{j=1}^{n-1}\sum_{j'=1}^{n-1}\sum_{m=1}^{n-1}\sum_{m'=1}^{n-1}
\wt{\rho}(2m-2j)\wh{\rho}(j-j')\wt{\rho}(2j'-2m')\wh{\rho}(2m'-2m)
\\&&
+n^{-1}\sum_{j=1}^{n-1}\sum_{j'=1}^{n-1}\sum_{m=1}^{n}\sum_{m'=1}^{n-1}
\wt{\rho}(2m-1-2j)\wh{\rho}(j-j')\wt{\rho}(2j'-2m')\wh{\rho}(2m'-2m+1)
\\&&
+n^{-1}\sum_{j=1}^{n-1}\sum_{j'=1}^{n-1}\sum_{m=1}^{n-1}\sum_{m'=1}^{n}
\wt{\rho}(2m-2j)\wh{\rho}(j-j')\wt{\rho}(2j'-2m'+1)\wh{\rho}(2m'-1-2m)
\\&&
+n^{-1}\sum_{j=1}^{n-1}\sum_{j'=1}^{n-1}\sum_{m=1}^{n}\sum_{m'=1}^{n}
\wt{\rho}(2m-1-2j)\wh{\rho}(j-j')\wt{\rho}(2j'-2m'+1)\wh{\rho}(2m'-2m)
\eeas
and with Lemma \ref{2108111404}, we see the gap between 
the above expression and the one below is of $o(1)$ 
(substitute $-i_1$ into the first factor): 
\beas&&
\sum_{i_1,i_2,i_3\in\bbZ}
\wt{\rho}(2i_1)\wh{\rho}(i_1-i_2)\wt{\rho}(2i_2-2i_3)\wh{\rho}(2i_3)
\\&&
+\sum_{i_1,i_2,i_3\in\bbZ}
\wt{\rho}(2i_1{\colorr+}1)\wh{\rho}(i_1-i_2)\wt{\rho}(2i_2-2i_3)\wh{\rho}(2i_3+1)
\\&&
+\sum_{i_1,i_2,i_3\in\bbZ}
\wt{\rho}(2i_1)\wh{\rho}(i_1-i_2)\wt{\rho}(2i_2-2i_3+1)\wh{\rho}(2i_3-1)
\\&&
+\sum_{i_1,i_2,i_3\in\bbZ}
\wt{\rho}(2i_1{\colorr+}1)\wh{\rho}(i_1-i_2)\wt{\rho}(2i_2-2i_3+1)\wh{\rho}(2i_3)+o(1)
\eeas
\beas
&=&
\sum_{i_1,i_2,i_3\in\bbZ}
\wt{\rho}(2i_1)\wh{\rho}(i_1-i_2)\wt{\rho}(2i_2-i_3)\wh{\rho}(i_3)
\\&&
+\sum_{i_1,i_2,i_3\in\bbZ}
\wt{\rho}(2i_1{\colorr+}1)\wh{\rho}(i_1-i_2)\wt{\rho}(2i_2-2i_3)\wh{\rho}(2i_3+1)
\\&&
+\sum_{i_1,i_2,i_3\in\bbZ}
\wt{\rho}(2i_1{\colorr+}1)\wh{\rho}(i_1-i_2)\wt{\rho}(2i_2-2i_3+1)\wh{\rho}(2i_3)+o(1)
\\&=&
\sum_{i_1,i_2,i_3\in\bbZ}
\wt{\rho}(2i_1)\wh{\rho}(i_1-i_2)\wt{\rho}(2i_2-i_3)\wh{\rho}(i_3)
\\&&
+\sum_{i_1,i_2,i_3\in\bbZ}
\wt{\rho}(2i_1{\colorr+}1)\wh{\rho}(i_1-i_2)\wt{\rho}(2i_2-i_3+1)\wh{\rho}(i_3)+o(1).
\eeas
{\sblue Therefore we obtained (b). }
\qed\halflineskip
\begin{en-text}
\beas &&
n^{-1}
\sum_{j=1}^{n-1}\sum_{j'=1}^{n-1}
\sum_{k=1}^{2n-1}\sum_{k'=1}^{2n-1}
\wt{\rho}(k-2j)\wt{\rho}(k'-2j')\wh{\rho}(j-j')\wh{\rho}(k-k')
\\&=&
n^{-1}\sum_{j}\sum_{j'}\sum_{k_1}\sum_{k'_1}
\wt{\rho}(2k_1-2j)\wt{\rho}(2k'_1-2j')\wh{\rho}(j-j')\wh{\rho}(2k_1-2k'_1)
\\&&
+n^{-1}\sum_{j}\sum_{j'}\sum_{k_1}\sum_{k'_1}
\wt{\rho}(2k_1-2j)\wt{\rho}(2k'_1-2j'+1)\wh{\rho}(j-j')\wh{\rho}(2k_1-2k'_1-1)
\\&&
+n^{-1}\sum_{j}\sum_{j'}\sum_{k_1}\sum_{k'_1}
\wt{\rho}(2k_1-2j+1)\wt{\rho}(2k'_1-2j')\wh{\rho}(j-j')\wh{\rho}(2k_1-2k'_1+1)
\\&&
+n^{-1}\sum_{j}\sum_{j'}\sum_{k_1}\sum_{k'_1}
\wt{\rho}(2k_1-2j+1)\wt{\rho}(2k'_1-2j'+1)\wh{\rho}(j-j')\wh{\rho}(2k_1-2k'_1)
\\&=&
\sum_{i_1,i_2,i_3\in\bbZ}
\wt{\rho}(2i_1)\wt{\rho}(2i_1-2i_2)\wh{\rho}(i_2-i_3)\wh{\rho}(2i_3)
\\&&
+\sum_{i_1,i_2,i_3\in\bbZ}
\wt{\rho}(2i_1)\wt{\rho}(2i_1-2i_2+1)\wh{\rho}(i_2-i_3)\wh{\rho}(2i_3-1)
\\&&
+\sum_{i_1,i_2,i_3\in\bbZ}
\wt{\rho}(2i_1+1)\wt{\rho}(2i_1-2i_2)\wh{\rho}(i_2-i_3)\wh{\rho}(2i_3+1)
\\&&
+\sum_{i_1,i_2,i_3\in\bbZ}
\wt{\rho}(2i_1+1)\wt{\rho}(2i_1-2i_2+1)\wh{\rho}(i_2-i_3)\wh{\rho}(2i_3)+o(1)
\eeas

\beas &&
\sum_{i_1,i_2,i_3\in\bbZ}
\wt{\rho}(2i_1+1)\wt{\rho}(2i_1-2i_2)\wh{\rho}(i_2-i_3)\wh{\rho}(2i_3+1)
\\&=&
\sum_{i_1,i_2,i_3\in\bbZ}
\wt{\rho}(-2i_1+1)\wt{\rho}(-2i_1+2i_2)\wh{\rho}(i_2-i_3)\wh{\rho}(-2i_3+1)
\\&=&
\sum_{i_1,i_2,i_3\in\bbZ}
\wt{\rho}(2i_1-1)\wt{\rho}(2i_1-2i_2)\wh{\rho}(i_2-i_3)\wh{\rho}(2i_3-1)
\\&=&
\sum_{i_1,i_2,i_3\in\bbZ}
\wt{\rho}(-2i_1+2i_2-1)\wt{\rho}(-2i_1)\wh{\rho}(i_2-i_3)\wh{\rho}(2i_3-1)
\\&=&
\sum_{i_1,i_2,i_3\in\bbZ}
\wt{\rho}(2i_1-2i_2+1)\wt{\rho}(2i_1)\wh{\rho}(i_2-i_3)\wh{\rho}(2i_3-1)
\eeas
\end{en-text}

Let 
\beas 
\kappa(1,1;1,2)
&=& 
\frac{2^{2H-1}}{\big(4-2^{2H}\big)^4}
\sum_{i_1,i_2,i_3\in\bbZ}
\wh{\rho}(i_1)\wt{\rho}(2i_1-i_2)\wt{\rho}(i_2-2i_3)\wh{\rho}(i_3),
\eeas
\beas 
\kappa(1,2;2,2)
&=&
\begin{en-text}
\bluc{\frac{2^{2H-3}}{(4-2^{2H})^4}
\bigg\{
\sum_{i_1,i_2,i_3\in\bbZ}
\wh{\rho}(2i_1)\wh{\rho}(2i_1-i_2)\wt{\rho}(i_2-i_3)\wt{\rho}(i_3)}
\\&&\hspace{70pt}
\bluc{+\sum_{i_1,i_2,i_3\in\bbZ}
\wh{\rho}(2i_1+1)\wh{\rho}(2i_1-i_2)\wt{\rho}(i_2+1-i_3)\wt{\rho}(i_3)
\bigg\}}
\\&=&
\end{en-text}
\grnc{\frac{2^{2H-3}}{(4-2^{2H})^4}
\sum_{i_1,i_2,i_3\in\bbZ}
\wh{\rho}(i_1)\wh{\rho}(i_1-i_2)\wt{\rho}(i_2-i_3)\wt{\rho}(i_3)}
\eeas
and
\beas 
\kappa(1,2;1,2)
&=& 
\kappa(1,2;1,2)_1+\kappa(1,2;1,2)_2
\eeas
where 
\beas
\kappa(1,2;1,2)_1
&=& 
\begin{en-text}
\bluc{\frac{2^{2H-3}}{(4-2^{2H})^4}\bigg\{
\sum_{i_1,i_2,i_3\in\bbZ}
\redc{\xout{\wt{\rho}(i_1)}}\grnc{\wh{\rho}(i_1)}
\wt{\rho}(2i_1-i_2)\wh{\rho}(i_2-2i_3)\wt{\rho}(2i_3)}
\\&&\hspace{60pt}
\bluc{+\sum_{i_1,i_2,i_3\in\bbZ}
\redc{\xout{\wt{\rho}(i_1)}}\grnc{\wh{\rho}(i_1)}
\wt{\rho}(2i_1-i_2)\wh{\rho}(i_2-2i_3+1)\wt{\rho}(2i_3-1)
\bigg\}}
\\&=& 
\end{en-text}
\grnc{\frac{2^{2H-3}}{(4-2^{2H})^4}
\sum_{i_1,i_2,i_3\in\bbZ}
\wh{\rho}(i_1)\wt{\rho}(2i_1-i_2)\wh{\rho}(i_2-i_3)\wt{\rho}(i_3)}
\eeas
and 
\beas 
\kappa(1,2;1,2)_2
&=&
\frac{2^{4H-3}}{(4-2^{2H})^4}\bigg\{
\sum_{i_1,i_2,i_3\in\bbZ}\wt{\rho}(2i_1)\wt{\rho}(2i_1-i_2)\wt{\rho}(i_2-2i_3)\wt{\rho}(2i_3)
\\&&\hspace{60pt}
+\sum_{i_1,i_2,i_3\in\bbZ}\wt{\rho}(2i_1+1)\wt{\rho}(2i_1-i_2)\wt{\rho}(i_2-2i_3)\wt{\rho}(2i_3{\colorr+}1)
\bigg\}.
\eeas
\begin{lemma}\label{2108130613}
\bd\im[(a)] 
$\ds n\big\langle f_n^{[1]}\otimes_1f_n^{[1]},\>
f_n^{[1]}\odot_1f_n^{[2]}\big\rangle_{\mfh^{\otimes2}}
\yeq 
\kappa(1,1;1,2)+o(1)$. 
\im[(b)] 
$\ds 
n\big\langle f_n^{[2]}\otimes_1f_n^{[2]},\>
f_n^{[1]}\odot_1f_n^{[2]}\big\rangle_{\mfh^{\otimes2}}
\yeq 
\kappa(1,2;2,2)+o(1)
$. 
\im[(c)] 
$\ds
n\big\langle f_n^{[1]}\odot_1f_n^{[2]},\>
f_n^{[1]}\odot_1f_n^{[2]}\big\rangle_{\mfh^{\otimes2}}
\yeq
\kappa(1,2;1,2)+o(1)
$. 
\ed
\end{lemma}
\proof
{\sblue (a):} 
We have 
\beas&&
n\big\langle f_n^{[1]}\otimes_1f_n^{[1]},\>
f_n^{[1]}\odot_1f_n^{[2]}\big\rangle_{\mfh^{\otimes2}}
\\&=&
\frac{n^3}{E[V^{(2)}_{n,T}]^3E[V^{(2)}_{2n,T}]}\l(\frac{T^{2H}}{n^{2H}}\r)^4
\sum_{j_1,j_2,j_3=1}^{n-1}\sum_{k=1}^{2n-1}
\wh{\rho}(j_1-j_2)\wt{\rho}(2j_2-k)\wt{\rho}(k-2j_3)\wh{\rho}(j_3-j_1)
\\&=&
\frac{n^3}{\big(4-2^{2H}\big)^4(T/n)^{6H}(n-1)^3(T/(2n))^{2H}(2n-1)}\l(\frac{T^{2H}}{n^{2H}}\r)^4
\\&&\times
\sum_{j_1,j_2,j_3=1}^{n-1}\sum_{k=1}^{2n-1}
\wh{\rho}(j_1-j_2)\wt{\rho}(2j_2-k)\wt{\rho}(k-2j_3)\wh{\rho}(j_3-j_1)
\\&=&
\frac{2^{2H}n^3}{\big(4-2^{2H}\big)^4(n-1)^3(2n-1)}
\sum_{j_1,j_2,j_3=1}^{n-1}\sum_{k=1}^{2n-1}
\wh{\rho}(j_1-j_2)\wt{\rho}(2j_2-k)\wt{\rho}(k-2j_3)\wh{\rho}(j_3-j_1)
\\&=&
\kappa(1,1;1,2)+o(1). 
\eeas
%
The last equality can be verified by 
\beas &&
n^{-1}\sum_{j_1,j_2,j_3=1}^{n-1}\sum_{k=1}^{2n-1}
\wh{\rho}(j_1-j_2)\wt{\rho}(2j_2-k)\wt{\rho}(k-2j_3)\wh{\rho}(j_3-j_1)
\\&=&
n^{-1}\sum_{j_1,j_2,j_3=1}^{n-1}\sum_{k=1}^{n-1}
\wh{\rho}(j_1-j_2)\wt{\rho}(2j_2-2k)\wt{\rho}(2k-2j_3)\wh{\rho}(j_3-j_1)
\\&&
+n^{-1}\sum_{j_1,j_2,j_3=1}^{n-1}\sum_{k=1}^{n}
\wh{\rho}(j_1-j_2)\wt{\rho}(2j_2-2k+1)\wt{\rho}(2k-1-2j_3)\wh{\rho}(j_3-j_1)
\\&=&
\sum_{i_1,i_2,i_3\in\bbZ}
\wh{\rho}(i_1)\wt{\rho}(2i_1-2i_2)\wt{\rho}(2i_2-2i_3)\wh{\rho}(i_3)
\\&&
+\sum_{i_1,i_2,i_3\in\bbZ}
\wh{\rho}(i_1)\wt{\rho}(2i_1-2i_2+1)\wt{\rho}(2i_2-1-2i_3)\wh{\rho}(i_3)+o(1)
\\&=&
\sum_{i_1,i_2,i_3\in\bbZ}
\wh{\rho}(i_1)\wt{\rho}(2i_1-i_2)\wt{\rho}(i_2-2i_3)\wh{\rho}(i_3)+o(1). 
\eeas
\begin{en-text}
where 
\beas 
\lambda(2)^{*\prime}
&=& 
\sum_{i_1,i_2,i_3\in\bbZ}
\wh{\rho}(i_1)\wt{\rho}(2i_1-i_2)\wt{\rho}(i_2-2i_3)\wh{\rho}(i_3)
\eeas
\end{en-text}

\noindent
{\sblue (b):} 
We will consider the product
\beas&&
n\big\langle f_n^{[2]}\otimes_1f_n^{[2]},\>
f_n^{[1]}\odot_1f_n^{[2]}\big\rangle_{\mfh^{\otimes2}}
\\&=&
\frac{n^3}{E[V^{(2)}_{n,T}]E[V^{(2)}_{2n,T}]^3}
\l(\frac{T^{2H}}{(2n)^{2H}}\r)^2\l(\frac{T^{2H}}{n^{2H}}\r)^2
\sum_{k_1,k_2,k_3=1}^{2n-1}\sum_{j=1}^{n-1}
\wh{\rho}(k_1-k_2)\wh{\rho}(k_2-k_3)\wt{\rho}(k_3-2j)\wt{\rho}(2j-k_1)
\\&=&
\frac{n^3}{(4-2^{2H})^4(T/n)^{2H}(T/(2n))^{6H}(n-1)(2n-1)^3}
\l(\frac{T^{2H}}{(2n)^{2H}}\r)^2\l(\frac{T^{2H}}{n^{2H}}\r)^2
\\&&\times
\sum_{k_1,k_2,k_3=1}^{2n-1}\sum_{j=1}^{n-1}
\wh{\rho}(k_1-k_2)\wh{\rho}(k_2-k_3)\wt{\rho}(k_3-2j)\wt{\rho}(2j-k_1)
\\&=&
\frac{2^{2H}n^3}{(4-2^{2H})^4(n-1)(2n-1)^3}
\sum_{k_1,k_2,k_3=1}^{2n-1}\sum_{j=1}^{n-1}
\wh{\rho}(k_1-k_2)\wh{\rho}(k_2-k_3)\wt{\rho}(k_3-2j)\wt{\rho}(2j-k_1).
\eeas
\begin{en-text}
Shift $k_1,k_2,k_3$ by $1$. 
\beas &&
n^{-1}\sum_{k_1,k_2,k_3=1}^{2n-1}\sum_{j=1}^{n-1}
\wh{\rho}(k_1-k_2)\wh{\rho}(k_2-k_3)\wt{\rho}(k_3-2j)\wt{\rho}(2j-k_1)
\\&=&
\half n^{-1}\sum_{k_1,k_2,k_3=1}^{2n-1}\sum_{j=1}^{n-1}
\wh{\rho}(k_1-k_2)\wh{\rho}(k_2-k_3)\wt{\rho}(k_3-2j)\wt{\rho}(2j-k_1)
\\&&
+\half n^{-1}\sum_{k_1=1}^{2n-1}\sum_{k_2=1}^{2n-1}
\wh{\rho}(k_1-k_2)\wh{\rho}(k_2-k_3)\wt{\rho}(k_3-(2j-1))\wt{\rho}((2j-1)-k_1)+o(1)
\\&=&
\half n^{-1}\sum_{k_1,k_2,k_3,k_4=1}^{2n-1}
\wh{\rho}(k_1-k_2)\wh{\rho}(k_2-k_3)\wt{\rho}(k_3-k_4)\wt{\rho}(k_4-k_1)+o(1)
\\&=&
\half\sum_{i_1,i_2,i_3in\bbZ}
\wh{\rho}(i_1)\wh{\rho}(i_1-i_2)\wt{\rho}(i_2-i_3)\wt{\rho}(i_3)+o(1)
\eeas
\end{en-text}
Define $k^0$ and $k^1$ by 
\beas 
k^0=2k,\quad k^1=2k-1,
\eeas
respectively for an integer $k$. 
The sum $\wt{\sum}_{k_1,k_2,k_3}$ should read according to 
a proper configuration of 
$k^0_i$ and $k^1_i$, $i=1,2,3$.
To proceed,
\beas &&
n^{-1}\sum_{k_1,k_2,k_3=1}^{2n-1}\sum_{j=1}^{n-1}
\wh{\rho}(k_1-k_2)\wh{\rho}(k_2-k_3)\wt{\rho}(k_3-2j)\wt{\rho}(2j-k_1)
\\&=&
n^{-1}\wt{\sum}_{k_1,k_2,k_3}\sum_{j=1}^{n-1}
\wh{\rho}(k_1^0-k_2^0)\wh{\rho}(k_2^0-k_3^0)\wt{\rho}(k_3^0-2j)\wt{\rho}(2j-k_1^0)
\\&&
+n^{-1}\wt{\sum}_{k_1,k_2,k_3}\sum_{j=1}^{n-1}
\wh{\rho}(k_1^0-k_2^0)\wh{\rho}(k_2^0-k_3^1)\wt{\rho}(k_3^1-2j)\wt{\rho}(2j-k_1^0)
\\&&
+n^{-1}\wt{\sum}_{k_1,k_2,k_3}\sum_{j=1}^{n-1}
\wh{\rho}(k_1^0-k_2^1)\wh{\rho}(k_2^1-k_3^0)\wt{\rho}(k_3^0-2j)\wt{\rho}(2j-k_1^0)
\\&&
+n^{-1}\wt{\sum}_{k_1,k_2,k_3}\sum_{j=1}^{n-1}
\wh{\rho}(k_1^0-k_2^1)\wh{\rho}(k_2^1-k_3^1)\wt{\rho}(k_3^1-2j)\wt{\rho}(2j-k_1^0)
\\&&
+n^{-1}\wt{\sum}_{k_1,k_2,k_3}\sum_{j=1}^{n-1}
\wh{\rho}(k_1^1-k_2^0)\wh{\rho}(k_2^0-k_3^0)\wt{\rho}(k_3^0-2j)\wt{\rho}(2j-k_1^1)
\\&&
+n^{-1}\wt{\sum}_{k_1,k_2,k_3}\sum_{j=1}^{n-1}
\wh{\rho}(k_1^1-k_2^0)\wh{\rho}(k_2^0-k_3^1)\wt{\rho}(k_3^1-2j)\wt{\rho}(2j-k_1^1)
\\&&
+n^{-1}\wt{\sum}_{k_1,k_2,k_3}\sum_{j=1}^{n-1}
\wh{\rho}(k_1^1-k_2^1)\wh{\rho}(k_2^1-k_3^0)\wt{\rho}(k_3^0-2j)\wt{\rho}(2j-k_1^1)
\\&&
+n^{-1}\wt{\sum}_{k_1,k_2,k_3}\sum_{j=1}^{n-1}
\wh{\rho}(k_1^1-k_2^1)\wh{\rho}(k_2^1-k_3^1)\wt{\rho}(k_3^1-2j)\wt{\rho}(2j-k_1^1)
\eeas
\beas 
&=&
\sum_{i_1,i_2,i_3\in\bbZ}
\wh{\rho}(2i_1)\wh{\rho}(2i_1-2i_2)\wt{\rho}(2i_2-2i_3)\wt{\rho}(2i_3)
\\&&
+\sum_{i_1,i_2,i_3\in\bbZ}
\wh{\rho}(2i_1)\wh{\rho}(2i_1-2i_2+1)\wt{\rho}(2i_2-2i_3-1)\wt{\rho}(2i_3)
\\&&
+\sum_{i_1,i_2,i_3\in\bbZ}
\wh{\rho}(2i_1-1)\wh{\rho}(2i_1-2i_2-1)\wt{\rho}(2i_2-2i_3)\wt{\rho}(2i_3)
\\&&
+\sum_{i_1,i_2,i_3\in\bbZ}
\wh{\rho}(2i_1-1)\wh{\rho}(2i_1-2i_2)\wt{\rho}(2i_2-2i_3-1)\wt{\rho}(2i_3)
\\&&
+\sum_{i_1,i_2,i_3\in\bbZ}
\wh{\rho}(2i_1+1)\wh{\rho}(2i_1-2i_2)\wt{\rho}(2i_2-2i_3)\wt{\rho}(2i_3+1)
\\&&
+\sum_{i_1,i_2,i_3\in\bbZ}
\wh{\rho}(2i_1+1)\wh{\rho}(2i_1-2i_2+1)\wt{\rho}(2i_2-2i_3-1)\wt{\rho}(2i_3+1)
\\&&
+\sum_{i_1,i_2,i_3\in\bbZ}
\wh{\rho}(2i_1)\wh{\rho}(2i_1-2i_2-1)\wt{\rho}(2i_2-2i_3)\wt{\rho}(2i_3+1)
\\&&
+\sum_{i_1,i_2,i_3\in\bbZ}
\wh{\rho}(2i_1)\wh{\rho}(2i_1-2i_2)\wt{\rho}(2i_2-2i_3-1)\wt{\rho}(2i_3+1)
+o(1)
\eeas
\beas 
&=&
\sum_{i_1,i_2,i_3\in\bbZ}
\wh{\rho}(2i_1)\wh{\rho}(2i_1-i_2)\wt{\rho}(i_2-2i_3)\wt{\rho}(2i_3)
\\&&
+\sum_{i_1,i_2,i_3\in\bbZ}
\wh{\rho}(2i_1-1)\wh{\rho}(2i_1-2i_2-1)\wt{\rho}(2i_2-2i_3)\wt{\rho}(2i_3)
\\&&
+\sum_{i_1,i_2,i_3\in\bbZ}
\wh{\rho}(2i_1-1)\wh{\rho}(2i_1-2i_2)\wt{\rho}(2i_2-2i_3-1)\wt{\rho}(2i_3)
\\&&
+\sum_{i_1,i_2,i_3\in\bbZ}
\wh{\rho}(2i_1+1)\wh{\rho}(2i_1-2i_2)\wt{\rho}(2i_2-2i_3)\wt{\rho}(2i_3+1)
\\&&
+\sum_{i_1,i_2,i_3\in\bbZ}
\wh{\rho}(2i_1+1)\wh{\rho}(2i_1-2i_2+1)\wt{\rho}(2i_2-2i_3-1)\wt{\rho}(2i_3+1)
\\&&
+\sum_{i_1,i_2,i_3\in\bbZ}
\wh{\rho}(2i_1)\wh{\rho}(2i_1-i_2)\wt{\rho}(i_2-2i_3-1)\wt{\rho}(2i_3+1)
+o(1)
\eeas
\beas 
&=&
\sum_{i_1,i_2,i_3\in\bbZ}
\wh{\rho}(2i_1)\wh{\rho}(2i_1-i_2)\wt{\rho}(i_2-2i_3)\wt{\rho}(2i_3)
\\&&
+\sum_{i_1,i_2,i_3\in\bbZ}
\wh{\rho}(2i_1+1)\wh{\rho}(2i_1-2i_2+1)\wt{\rho}(2i_2-2i_3)\wt{\rho}(2i_3)
\quad(i_1\leftarrow i_1+1)
\\&&
+\sum_{i_1,i_2,i_3\in\bbZ}
\wh{\rho}(2i_1+1)\wh{\rho}(2i_1-2i_2)\wt{\rho}(2i_2-2i_3-1)\wt{\rho}(2i_3+2)
\quad(i_*\leftarrow i_*+1)
\\&&
+\sum_{i_1,i_2,i_3\in\bbZ}
\wh{\rho}(2i_1+1)\wh{\rho}(2i_1-2i_2)\wt{\rho}(2i_2-2i_3)\wt{\rho}(2i_3+1)
\\&&
+\sum_{i_1,i_2,i_3\in\bbZ}
\wh{\rho}(2i_1+1)\wh{\rho}(2i_1-2i_2+1)\wt{\rho}(2i_2-2i_3-1)\wt{\rho}(2i_3+1)
\\&&
+\sum_{i_1,i_2,i_3\in\bbZ}
\wh{\rho}(2i_1)\wh{\rho}(2i_1-i_2)\wt{\rho}(i_2-2i_3-1)\wt{\rho}(2i_3+1)
+o(1)
\eeas
\beas 
&=&
\sum_{i_1,i_2,i_3\in\bbZ}
\wh{\rho}(2i_1)\wh{\rho}(2i_1-i_2)\wt{\rho}(i_2-2i_3)\wt{\rho}(2i_3)
\\&&
+\sum_{i_1,i_2,i_3\in\bbZ}
\wh{\rho}(2i_1+1)\wh{\rho}(2i_1-2i_2+1)\wt{\rho}(2i_2-i_3)\wt{\rho}(i_3)
\quad(\text{combined with the 5th term})
\\&&
+\sum_{i_1,i_2,i_3\in\bbZ}
\wh{\rho}(2i_1+1)\wh{\rho}(2i_1-2i_2)\wt{\rho}(2i_2-i_3)\wt{\rho}(i_3+1)
\quad(\text{combined with the 4th term})
\\&&
+\sum_{i_1,i_2,i_3\in\bbZ}
\wh{\rho}(2i_1)\wh{\rho}(2i_1-i_2)\wt{\rho}(i_2-2i_3-1)\wt{\rho}(2i_3+1)
+o(1)
\eeas
\begin{en-text}
\redc{\beas 
&=&
\sum_{i_1,i_2,i_3\in\bbZ}
\wh{\rho}(2i_1)\wh{\rho}(2i_1-i_2)\wt{\rho}(i_2-2i_3)\wt{\rho}(2i_3)
\\&&
+\sum_{i_1,i_2,i_3\in\bbZ}
\wh{\rho}(2i_1+1)\wh{\rho}(2i_1-2i_2+1)\wt{\rho}(2i_2-i_3)\wt{\rho}(i_3)
\quad(\text{combined with the 5th term})
\\&&
+\sum_{i_1,i_2,i_3\in\bbZ}
\wh{\rho}(2i_1+1)\wh{\rho}(2i_1-2i_2)\wt{\rho}(2i_2-i_3)\wt{\rho}(i_3+1)
\quad(\text{combined with the 4th term})
\\&&
+\sum_{i_1,i_2,i_3\in\bbZ}
\wh{\rho}(2i_1)\wh{\rho}(2i_1-i_2)\wt{\rho}(i_2-2i_3-1)\wt{\rho}(2i_3+1)
+o(1)
\eeas}
\end{en-text}
\beas 
&=&
\sum_{i_1,i_2,i_3\in\bbZ}
\wh{\rho}(2i_1)\wh{\rho}(2i_1-i_2)\wt{\rho}(i_2-i_3)\wt{\rho}(i_3)
\quad(\text{combined with the 4th term})
\\&&
+\sum_{i_1,i_2,i_3\in\bbZ}
\wh{\rho}(2i_1+1)\wh{\rho}(2i_1-2i_2+1)\wt{\rho}(2i_2-i_3)\wt{\rho}(i_3)
\\&&
+\sum_{i_1,i_2,i_3\in\bbZ}
\wh{\rho}(2i_1+1)\wh{\rho}(2i_1-2i_2)\wt{\rho}(2i_2-i_3)\wt{\rho}(i_3+1)
+o(1)
\eeas
\beas 
&=&
\sum_{i_1,i_2,i_3\in\bbZ}
\wh{\rho}(2i_1)\wh{\rho}(2i_1-i_2)\wt{\rho}(i_2-i_3)\wt{\rho}(i_3)
\\&&
+\sum_{i_1,i_2,i_3\in\bbZ}
\wh{\rho}(2i_1+1)\wh{\rho}(2i_1-2i_2+1)\wt{\rho}(2i_2-i_3)\wt{\rho}(i_3)
\\&&
+\sum_{i_1,i_2,i_3\in\bbZ}
\wh{\rho}(2i_1+1)\wh{\rho}(2i_1-2i_2)\wt{\rho}(2i_2+1-i_3)\wt{\rho}(i_3)
\quad(i_3\leftarrow i_3-1)
\\&&
+o(1)
\eeas
\beas 
&=&
\sum_{i_1,i_2,i_3\in\bbZ}
\wh{\rho}(2i_1)\wh{\rho}(2i_1-i_2)\wt{\rho}(i_2-i_3)\wt{\rho}(i_3)
\\&&
+\sum_{i_1,i_2,i_3\in\bbZ}
\wh{\rho}(2i_1+1)\wh{\rho}(2i_1-i_2)\wt{\rho}(i_2+1-i_3)\wt{\rho}(i_3)
\\&&
+o(1)
\eeas
\grnc{
\beas 
&=&
\sum_{i_1,i_2,i_3\in\bbZ}
\wh{\rho}(2i_1)\wh{\rho}(2i_1-i_2)\wt{\rho}(i_2-i_3)\wt{\rho}(i_3)
\\&&
+\sum_{i_1,i_2,i_3\in\bbZ}
\wh{\rho}(2i_1+1)\wh{\rho}((2i_1+1)-i_2-1)\wt{\rho}(i_2+1-i_3)\wt{\rho}(i_3)
\\&&
+o(1)
\\&=&
\sum_{i_1,i_2,i_3\in\bbZ}
\wh{\rho}(2i_1)\wh{\rho}(2i_1-i_2)\wt{\rho}(i_2-i_3)\wt{\rho}(i_3)
\\&&
+\sum_{i_1,i_2,i_3\in\bbZ}
\wh{\rho}(2i_1+1)\wh{\rho}((2i_1+1)-i_2)\wt{\rho}(i_2-i_3)\wt{\rho}(i_3)
\\&&
+o(1)
\\&=&
\sum_{i_1,i_2,i_3\in\bbZ}
\wh{\rho}(i_1)\wh{\rho}(i_1-i_2)\wt{\rho}(i_2-i_3)\wt{\rho}(i_3)
+o(1).
\eeas
}
\begin{en-text}
\beas
\\&=&
\lambda(2)^{**\prime}+o(1)
\eeas
Remark. It is not clear whether the following representation of $\lambda(2)^{**\prime}$ has 
a more compact one: 
\beas 
\lambda(2)^{**\prime}&=&
\sum_{i_1,i_2,i_3\in\bbZ}
\wh{\rho}(2i_1)\wh{\rho}(2i_1-2i_2)\wt{\rho}(2i_2-i_3)\wt{\rho}(i_3)
\\&&
+\sum_{i_1,i_2,i_3\in\bbZ}
\wh{\rho}(2i_1)\wh{\rho}(2i_1-2i_2-1)\wt{\rho}(2i_2-2i_3+1)\wt{\rho}(2i_3)
\\&&
+\sum_{i_1,i_2,i_3\in\bbZ}
\wh{\rho}(2i_1-1)\wh{\rho}(2i_1-2i_2+1)\wt{\rho}(2i_2-2i_3)\wt{\rho}(2i_3)
\\&&
+\sum_{i_1,i_2,i_3\in\bbZ}
\wh{\rho}(2i_1-1)\wh{\rho}(2i_1-2i_2)\wt{\rho}(2i_2-2i_3+1)\wt{\rho}(2i_3)
\\&&
+\sum_{i_1,i_2,i_3\in\bbZ}
\wh{\rho}(2i_1+1)\wh{\rho}(2i_1-i_2)\wt{\rho}(i_2-2i_3)\wt{\rho}(2i_3-1)
\\&&
+\sum_{i_1,i_2,i_3\in\bbZ}
\wh{\rho}(2i_1)\wh{\rho}(2i_1-2i_2+1)\wt{\rho}(2i_2-2i_3)\wt{\rho}(2i_3-1)
\eeas
\end{en-text}
Therefore
\beas&&
n\big\langle f_n^{[2]}\otimes_1f_n^{[2]},\>
f_n^{[1]}\odot_1f_n^{[2]}\big\rangle_{\mfh^{\otimes2}}
\\&=&
\begin{en-text}
\bluc{\frac{2^{2H-3}}{(4-2^{2H})^4}
\bigg\{
\sum_{i_1,i_2,i_3\in\bbZ}
\wh{\rho}(2i_1)\wh{\rho}(2i_1-i_2)\wt{\rho}(i_2-i_3)\wt{\rho}(i_3)}
\\&&
\bluc{+\sum_{i_1,i_2,i_3\in\bbZ}
\wh{\rho}(2i_1+1)\wh{\rho}(2i_1-i_2)\wt{\rho}(i_2+1-i_3)\wt{\rho}(i_3)
\bigg\}}
\\&=&
\end{en-text}
\grnc{\frac{2^{2H-3}}{(4-2^{2H})^4}
\sum_{i_1,i_2,i_3\in\bbZ}
\wh{\rho}(i_1)\wh{\rho}(i_1-i_2)\wt{\rho}(i_2-i_3)\wt{\rho}(i_3)}+o(1)
\\&=&
\kappa(1,2;2,2)+o(1).
\eeas
\begin{en-text}
where 
\beas 
\lambda(2)^{**}=\frac{2^{4H-3}}{(4-2^H)^4}\lambda(2)^{**\prime}
\eeas
\end{en-text}

{\sblue (c):} 
For the product 
$n\big\langle f_n^{[1]}\odot_1f_n^{[2]},\>
f_n^{[1]}\odot_1f_n^{[2]}\big\rangle_{\mfh^{\otimes2}}$, we have
\beas&&
n\big\langle f_n^{[1]}\odot_1f_n^{[2]},\>
f_n^{[1]}\odot_1f_n^{[2]}\big\rangle_{\mfh^{\otimes2}}
\\&=&
\frac{n^3}{E[V^{(2)}_{n,T}]^2E[V^{(2)}_{2n,T}]^2}\frac{T^{4H}}{n^{4H}}
\\&&\times
\sum_{j_1,j_2=1}^{n-1}\sum_{k_1,k_2=1}^{2n-1}
\big\langle(1_{j_1+1}^n-1_{j_1}^n)\odot(1_{k_1+1}^{2n}-1_{k_1}^{2n}),\>
(1_{j_2+1}^n-1_{j_2}^n)\odot(1_{k_2+1}^{2n}-1_{k_2}^{2n})\big\rangle_{\mfh^{\otimes2}}
\\&&\hspace{80pt}\times
\wt{\rho}(k_1,j_1)\wt{\rho}(k_2,j_2)
\\&=&
\frac{n^3}{E[V^{(2)}_{n,T}]^2E[V^{(2)}_{2n,T}]^2}\frac{T^{4H}}{n^{4H}}
\sum_{j_1,j_2=1}^{n-1}\sum_{k_1,k_2=1}^{2n-1}\half\wt{\rho}(k_1,j_1)\wt{\rho}(k_2,j_2)
\\&&\times
\bigg\{
\big\langle1_{j_1+1}^n-1_{j_1}^n,\>1_{j_2+1}^n-1_{j_2}^n\big\rangle_{\mfh}
\big\langle1_{k_1+1}^{2n}-1_{k_1}^{2n},\>1_{k_2+1}^{2n}-1_{k_2}^{2n}\big\rangle_{\mfh}
\\&&
+
\big\langle1_{j_2+1}^n-1_{j_2}^n,\>1_{k_1+1}^{2n}-1_{k_1}^{2n}\big\rangle_{\mfh}
\big\langle1_{j_1+1}^n-1_{j_1}^n,\>1_{k_2+1}^{2n}-1_{k_2}^{2n}\big\rangle_{\mfh}
\bigg\}
\\&=&
\frac{n^3}{E[V^{(2)}_{n,T}]^2E[V^{(2)}_{2n,T}]^2}\frac{T^{4H}}{n^{4H}}
\sum_{j_1,j_2=1}^{n-1}\sum_{k_1,k_2=1}^{2n-1}\half\wt{\rho}(k_1-2j_1)\wt{\rho}(k_2-2j_2)
\\&&\times
\bigg\{
\frac{T^{2H}}{n^{2H}}\frac{T^{2H}}{(2n)^{2H}}\wh{\rho}(j_1-j_2)\wh{\rho}(k_1-k_2)
+\frac{T^{2H}}{n^{2H}}\frac{T^{2H}}{n^{2H}}\wt{\rho}(k_2-2j_1)\wt{\rho}(k_1-2j_2)
\bigg\}
\\&=&
\frac{n^3}{E[V^{(2)}_{n,T}]^2E[V^{(2)}_{2n,T}]^2}\frac{T^{8H}}{n^{8H}}
\sum_{j_1,j_2=1}^{n-1}\sum_{k_1,k_2=1}^{2n-1}\half\wt{\rho}(k_1-2j_1)\wt{\rho}(k_2-2j_2)
\\&&\times
\bigg\{
2^{-2H}\wh{\rho}(j_1-j_2)\wh{\rho}(k_1-k_2)
+\wt{\rho}(k_2-2j_1)\wt{\rho}(k_1-2j_2)
\bigg\}
\\&=&
\frac{2^{4H-3}}{(4-2^{2H})^4n}
\sum_{j_1,j_2=1}^{n-1}\sum_{k_1,k_2=1}^{2n-1}\wt{\rho}(k_1-2j_1)\wt{\rho}(k_2-2j_2)
\\&&\times
\bigg\{
2^{-2H}\wh{\rho}(j_1-j_2)\wh{\rho}(k_1-k_2)
+\wt{\rho}(k_2-2j_1)\wt{\rho}(k_1-2j_2)
\bigg\}+o(1)
\\&=&
\frac{2^{2H-3}}{(4-2^{2H})^4n}
\sum_{j_1,j_2=1}^{n-1}\sum_{k_1,k_2=1}^{2n-1}\wt{\rho}(k_1-2j_1)\wt{\rho}(k_2-2j_2)
\wh{\rho}(j_1-j_2)\wh{\rho}(k_1-k_2)
\\&&
+\frac{2^{4H-3}}{(4-2^{2H})^4n}
\sum_{j_1,j_2=1}^{n-1}\sum_{k_1,k_2=1}^{2n-1}\wt{\rho}(k_1-2j_1)\wt{\rho}(k_2-2j_2)
\wt{\rho}(k_2-2j_1)\wt{\rho}(k_1-2j_2)+o(1).
\eeas
Now the convergence in (c) is verified as follows. 
To compute the first sum, 
\beas &&
n^{-1}\sum_{j_1,j_2=1}^{n-1}\sum_{k_1,k_2=1}^{2n-1}\wt{\rho}(k_1-2j_1)\wt{\rho}(k_2-2j_2)
\wh{\rho}(j_1-j_2)\wh{\rho}(k_1-k_2)
\\&=&
n^{-1}\sum_{j_1,j_2=1}^{n-1}\sum_{k_1,k_2=1}^{2n-1}
\wh{\rho}(j_1-j_2)\wt{\rho}(2j_2-k_2)\wh{\rho}(k_2-k_1)\wt{\rho}(k_1-2j_1)
\\&=&%
n^{-1}\wt{\sum}_{j_1,j_2}\wt{\sum}_{k_1,k_2}
\wh{\rho}(j_1-j_2)\wt{\rho}(2j_2-k_2^0)\wh{\rho}(k_2^0-k_1^0)\wt{\rho}(k_1^0-2j_1)
\\&&
+n^{-1}\wt{\sum}_{j_1,j_2}\wt{\sum}_{k_1,k_2}
\wh{\rho}(j_1-j_2)\wt{\rho}(2j_2-k_2^1)\wh{\rho}(k_2^1-k_1^0)\wt{\rho}(k_1^0-2j_1)
\\&&
+n^{-1}\wt{\sum}_{j_1,j_2}\wt{\sum}_{k_1,k_2}
\wh{\rho}(j_1-j_2)\wt{\rho}(2j_2-k_2^0)\wh{\rho}(k_2^0-k_1^1)\wt{\rho}(k_1^1-2j_1)
\\&&
+n^{-1}\wt{\sum}_{j_1,j_2}\wt{\sum}_{k_1,k_2}
\wh{\rho}(j_1-j_2)\wt{\rho}(2j_2-k_2^1)\wh{\rho}(k_2^1-k_1^1)\wt{\rho}(k_1^1-2j_1)
\\&=&%
\sum_{i_1,i_2,i_3\in\bbZ}
\grnc{\wh{\rho}(i_1)}
\wt{\rho}(2i_1-2i_2)\wh{\rho}(2i_2-2i_3)\wt{\rho}(2i_3)
\\&&
+\sum_{i_1,i_2,i_3\in\bbZ}
\grnc{\wh{\rho}(i_1)}
\wt{\rho}(2i_1-2i_2+1)\wh{\rho}(2i_2-2i_3-1)\wt{\rho}(2i_3)
\\&&
+\sum_{i_1,i_2,i_3\in\bbZ}
\grnc{\wh{\rho}(i_1)}
\wt{\rho}(2i_1-2i_2)\wh{\rho}(2i_2-2i_3+1)\wt{\rho}(2i_3-1)
\\&&
+\sum_{i_1,i_2,i_3\in\bbZ}
\grnc{\wh{\rho}(i_1)}
\wt{\rho}(2i_1-2i_2+1)\wh{\rho}(2i_2-2i_3)\wt{\rho}(2i_3-1)+o(1)
\\&=&%
\sum_{i_1,i_2,i_3\in\bbZ}
\grnc{\wh{\rho}(i_1)}
\wt{\rho}(2i_1-i_2)\wh{\rho}(i_2-2i_3)\wt{\rho}(2i_3)
\\&&
+\sum_{i_1,i_2,i_3\in\bbZ}
\grnc{\wh{\rho}(i_1)}
\wt{\rho}(2i_1-i_2)\wh{\rho}(i_2-2i_3+1)\wt{\rho}(2i_3-1)
+o(1).
\\&=&%
\grnc{\sum_{i_1,i_2,i_3\in\bbZ}
\wh{\rho}(i_1)\wt{\rho}(2i_1-i_2)\wh{\rho}(i_2-i_3)\wt{\rho}(i_3)
+o(1).}
\eeas
For the second sum, 
\beas &&
n^{-1}\sum_{j_1,j_2=1}^{n-1}\sum_{k_1,k_2=1}^{2n-1}\wt{\rho}(k_1-2j_1)\wt{\rho}(k_2-2j_2)
\wt{\rho}(k_2-2j_1)\wt{\rho}(k_1-2j_2)
\\&=&
n^{-1}\sum_{j_1,j_2=1}^{n-1}\sum_{k_1,k_2=1}^{2n-1}
\wt{\rho}(k_1-2j_1)\wt{\rho}(2j_1-k_2)\wt{\rho}(k_2-2j_2)\wt{\rho}(2j_2-k_1)
\\&=&
n^{-1}\wt{\sum}_{j_1,j_2,k_1,k_2}
\wt{\rho}(k_1^0-2j_1)\wt{\rho}(2j_1-k_2^0)\wt{\rho}(k_2^0-2j_2)\wt{\rho}(2j_2-k_1^0)
\\&&
+n^{-1}\wt{\sum}_{j_1,j_2,k_1,k_2}
\wt{\rho}(k_1^0-2j_1)\wt{\rho}(2j_1-k_2^1)\wt{\rho}(k_2^1-2j_2)\wt{\rho}(2j_2-k_1^0)
\\&&
+n^{-1}\wt{\sum}_{j_1,j_2,k_1,k_2}
\wt{\rho}(k_1^1-2j_1)\wt{\rho}(2j_1-k_2^0)\wt{\rho}(k_2^0-2j_2)\wt{\rho}(2j_2-k_1^1)
\\&&
+n^{-1}\wt{\sum}_{j_1,j_2,k_1,k_2}
\wt{\rho}(k_1^1-2j_1)\wt{\rho}(2j_1-k_2^1)\wt{\rho}(k_2^1-2j_2)\wt{\rho}(2j_2-k_1^1)
\eeas
\beas
&=&%
\sum_{i_1,i_2,i_3\in\bbZ}\wt{\rho}(2i_1)\wt{\rho}(2i_1-2i_2)\wt{\rho}(2i_2-2i_3)\wt{\rho}(2i_3)
\\&&
+\sum_{i_1,i_2,i_3\in\bbZ}\wt{\rho}(2i_1)\wt{\rho}(2i_1-2i_2{\colorr+}1)\wt{\rho}(2i_2-2i_3{\colorr-}1)\wt{\rho}(2i_3)
\\&&
+\sum_{i_1,i_2,i_3\in\bbZ}\wt{\rho}(2i_1+1)\wt{\rho}(2i_1-2i_2)\wt{\rho}(2i_2-2i_3)\wt{\rho}(2i_3{\colorr+}1)
\\&&
+\sum_{i_1,i_2,i_3\in\bbZ}\wt{\rho}(2i_1+1)\wt{\rho}(2i_1-2i_2{\colorr+}1)\wt{\rho}(2i_2-2i_3{\colorr-}1)\wt{\rho}(2i_3{\colorr+}1)+o(1)
\\&=&%
\sum_{i_1,i_2,i_3\in\bbZ}\wt{\rho}(2i_1)\wt{\rho}(2i_1-i_2)\wt{\rho}(i_2-2i_3)\wt{\rho}(2i_3)
\\&&
+\sum_{i_1,i_2,i_3\in\bbZ}\wt{\rho}(2i_1+1)\wt{\rho}(2i_1-i_2)\wt{\rho}(i_2-2i_3)\wt{\rho}(2i_3{\colorr+}1)
+o(1).
\eeas
\qed
\begin{en-text}
Therefore 
\beas
n\big\langle f_n^{[1]}\odot_1f_n^{[2]},\>
f_n^{[1]}\odot_1f_n^{[2]}\big\rangle_{\mfh^{\otimes2}}
&=&
\frac{2^{2H-3}}{(4-2^{2H})^4}\lambda(4_1)
+\frac{2^{4H-3}}{(4-2^{2H})^4}\lambda(4_2)
+o(1)
\eeas
\end{en-text}

\subsubsection{Central limit theorems}
We are now on the point of getting a central limit theorem for $U_n$. 
Recall 
\beas 
U_n&=& \big(M_n^{[1]},M_n^{[2]},n^{1/2}\wt{G}_n\big). 
\eeas
{\sblue Denote by $N_k(0,C)$ the $k$-dimensional centered normal distribution with covariance matrix $C$.} 
\begin{proposition}\label{2108140105}
The random vector $U_n$ is asymptotically normal, that is, 
\beas 
U_n &\to^d& N_3(0,{\mathfrak U})
\eeas
as $n\to\infty$, where 
${\mathfrak U}=({\mathfrak U}_{ij})_{i,j=1}^3$ is a symmetric matrix 
with components 
\beas 
{\mathfrak U}_{11}&=& \Sigma_{11},\quad
{\mathfrak U}_{12}\yeq\Sigma_{12},\quad
{\mathfrak U}_{22}\yeq\Sigma_{22},\y
{\mathfrak U}_{13}&=& {\colorr 4}\kappa(1;2,2)-{\colorr 8}\kappa(1;1,2)+{\colorr 4}\kappa(1;1,1),\y
{\mathfrak U}_{23}&=& {\colorr 4}\kappa(2;2,2)-{\colorr 8}\kappa(2;1,2)+{\colorr 4}\kappa(2;1,1),\y
{\mathfrak U}_{33}&=& {\colorr 8}\kappa(2,2;2,2)+{\colorr 32}\kappa(1,2;1,2)+{\colorr 8}\kappa(1,1:1,1)\y
&&-{\colorr 32}\kappa(1,2;2,2)+{\colorr 16}\kappa(1,1;2,2)-{\colorr 32}\kappa(1,1;1,2).
\eeas
\end{proposition}
\proof 
The Cram\'er-Wold device is used to show the three-dimensional central limit theorem. 
It follows from Lemmas 
\ref{2108121234}, \ref{2108130356}, \ref{2108130456} and \ref{2108130613} 
with the representations (\ref{2108140111}) and (\ref{2108121206}) 
that the asymptotic covariance matrix of $v\cdot U_n$ is 
$^tv{\mathfrak U}v$ for every $v\in\bbR^3$. 
To apply the fourth moment theorem (Peccati and Nualart \cite{nualart2005central}), 
what we need to show is 
$\text{Var}\big[|Dv\cdot U_n|_\mfh^2\big]\to0$ as $n\to\infty$. 
We will demonstrate $\text{Var}[|DF_n|_\mfh^2]\to0$ 
only for the component $F_n=2^{-1}n^{1/2}I_2\big(f_n^{[1]}\otimes_1f_n^{[1]}\big)$ 
because proof of the convergence of the other components is quite simlar. 
%
For example, we consider a component $F_n=2^{-1}n^{1/2}I_2\big(f_n^{[1]}\otimes_1f_n^{[1]}\big)$. 
By definition of $F_n$, we have 
\beas 
DF_n 
&=& 
n^{1/2}I_1\big(f_n^{[1]}\otimes_1f_n^{[1]}\big)
\\&=&
\frac{n^{1.5}}{E[V^{(2)}_{n,T}]^2}\frac{T^{2H}}{n^{2H}}
\sum_{j=1}^{n-1}\sum_{j'=1}^{n-1}
I_1\big(1^{n}_{j+1}-1^{n}_j\big)(1^{n}_{j'+1}-1^{n}_{j'})
\wh{\rho}(j-j')
\eeas
and
\beas 
|DF_n |_\mfh^2
&=& 
{\colorr\frac{n^3}{(4-2^{2H})^4(T/n)^{8H}(n-1)^4}\l(\frac{T^{2H}}{n^{2H}}\r)^2}
\sum_{j_1,j_2,j_3,j_4=1}^{n-1} 
I_1\big(1_{j_1+1}^n-1_{j_1}^n\big)I_1\big(1_{j_4+1}^n-1_{j_4}^n\big)
\\&&\times
\big\langle 1_{j_2+1}^n-1_{j_2}^n,\>1_{j_3+1}^n-1_{j_3}^n\rangle_\mfh
\wh{\rho}(j_1-j_2)\wh{\rho}(j_3-j_4)
\\&=& 
{\colorr\frac{n^3}{(4-2^{2H})^4(T/n)^{8H}(n-1)^4}\l(\frac{T^{2H}}{n^{2H}}\r)^3}
\\&&\times
\sum_{j_1,j_2,j_3,j_4=1}^{n-1} 
I_1\big(1_{j_1+1}^n-1_{j_1}^n\big)I_1\big(1_{j_4+1}^n-1_{j_4}^n\big)
\wh{\rho}(j_2-j_3)
\wh{\rho}(j_1-j_2)\wh{\rho}(j_3-j_4)
\\&=& 
{\colorr\frac{n^3}{(4-2^{2H})^4(T/n)^{8H}(n-1)^4}\l(\frac{T^{2H}}{n^{2H}}\r)^3}
\\&&\times
\sum_{j_1,j_2,j_3,j_4=1}^{n-1} 
I_2\big((1_{j_1+1}^n-1_{j_1}^n)\odot(1_{j_4+1}^n-1_{j_4}^n)\big)
\wh{\rho}(j_2-j_3)\wh{\rho}(j_1-j_2)\wh{\rho}(j_3-j_4)
\\&&
+E[|DF_n |_\mfh^2]. 
\eeas
Therefore 
\beas 
\text{Var}\big[|DF_n |_\mfh^2\big]
&=& 
{\colorr\frac{n^6}{(4-2^{2H})^8(T/n)^{16H}(n-1)^8}\l(\frac{T^{2H}}{n^{2H}}\r)^6}
\sum_{j_1,...,j_8=1}^{n-1} 
\\&&\times
E\bigg[
I_2\big((1_{j_1+1}^n-1_{j_1}^n)\odot(1_{j_4+1}^n-1_{j_4}^n)\big)
I_2\big((1_{j_8+1}^n-1_{j_8}^n)\odot(1_{j_5+1}^n-1_{j_5}^n)\big)\bigg]
\\&&\times
\wh{\rho}(j_2-j_3)\wh{\rho}(j_1-j_2)\wh{\rho}(j_3-j_4)
\wh{\rho}(j_6-j_7)\wh{\rho}(j_5-j_6)\wh{\rho}(j_7-j_8)
\\&=& 
{\colorr\frac{n^6}{(4-2^{2H})^8(T/n)^{16H}(n-1)^8}\l(\frac{T^{2H}}{n^{2H}}\r)^8}
\\&&\times
\sum_{j_1,...,j_8=1}^{n-1} 
\half\big[\wh{\rho}(j_1-j_8)\wh{\rho}(j_4-j_5)+\wh{\rho}(j_4-j_8)\wh{\rho}(j_1-j_5)
\big]
\\&&\times
\wh{\rho}(j_2-j_3)\wh{\rho}(j_1-j_2)\wh{\rho}(j_3-j_4)
\wh{\rho}(j_6-j_7)\wh{\rho}(j_5-j_6)\wh{\rho}(j_7-j_8)
\\&=& 
{\colorr\frac{n^6}{(4-2^{2H})^8(n-1)^8}}
\sum_{j_1,...,j_8=1}^{n-1} 
\wh{\rho}(j_1-j_2)\wh{\rho}(j_2-j_3)\wh{\rho}(j_3-j_4)\wh{\rho}(j_4-j_5)
\\&&\hspace{130pt}\times
\wh{\rho}(j_5-j_6)\wh{\rho}(j_6-j_7)\wh{\rho}(j_7-j_8)\wh{\rho}(j_8-j_1).
\eeas

\im We have $\bbV\big(\overbrace{\wh{\rho},...,\wh{\rho}}^{k}\big) < \infty$ since $4-2H>2>\frac{k-1}{k}$ 
for every $k\geq2$.
Therefore
\beas 
\text{Var}\big[|DF_n |_\mfh^2\big]
&=& 
O(n^{-1})
\eeas
by Lemma \ref{2108111404}. 
\qed\halflineskip

A central limit theorem for 
$
T_n\yeq\big(M_n,n^{1/2}\wt{G}_n\big)
$
follows immediately from Proposition \ref{2108140105} 
since 
\grnc{$M_n=M_n^{[1]}-M_n^{[2]}$. }
\begin{corollary}\label{2108140225}
$
T_n \to^d N_2(0,{\mathfrak T})
$
as $n\to\infty$, where ${\mathfrak T}=({\mathfrak T}_{ij})_{i,j=1}^2$ is a symmetric matrix 
with components 
\beas 
{\mathfrak T}_{11}&=& {\mathfrak U}_{11}-2 {\mathfrak U}_{12}+ {\mathfrak U}_{22},\quad
\grnc{{\mathfrak T}_{12}\yeq{\mathfrak U}_{13}-{\mathfrak U}_{23},}\quad
{\mathfrak T}_{22}\yeq{\mathfrak U}_{33}.
\eeas
\end{corollary}

{\sred 
We remark that 
\bea\label{2108211140} 
G_\infty\yeq {\mathfrak T}_{11}
\yeq
\Sigma_{22}-2\Sigma_{12}+\Sigma_{11}
\yeq 
\frac{3}{2}\Sigma_{11}-2\Sigma_{12}
\ygeq 
\bigg(\frac{3}{2}-\sqrt{2}\bigg)\Sigma_{11}
\>>\>0,
\eea
where $G_\infty$ was defined by (\ref{2108211136}). 
Indeed, 
from Proposiion \ref{2108140105}, 
$(M_n^{[1]},M_n^{[2]})\to^d(M_\infty^{[1]},M_\infty^{[2]})$, 
a two-dimensional Gaussian random vector with covariance matrix $(\Sigma_{ij})_{i,j=1,2}$. 
\begin{en-text}
If $G_\infty=\Sigma_{22}-2\Sigma_{12}+\Sigma_{11}=\bbE\big[\big(M_\infty^{[2]}-M_\infty^{[1]}\big)^2\big]=0$, then $M_\infty^{[2]}=M_\infty^{[1]}$, however it is impossible 
because $\Sigma_{22}=\half\Sigma_{11}$ by (\ref{2108211148}). 
\end{en-text}
We applied the Schwarz inequality to obtain (\ref{2108211140}). 
}

\begin{en-text}
\beas
E\big[M_n^{[i]}\wt{G}_n\big]
&=&
E\big[I_2(f_n^{[i]})\big\{
2I_2\big(f_n^{[2]}\otimes_1f_n^{[2]}\big)
-4I_2\big(f_n^{[2]}\odot_1f_n^{[1]}\big)
+2I_2\big(f_n^{[1]}\otimes_1f_n^{[1]}\big)\big\}
\big]
\nn\\&=&
2\big\langle f_n^{[i]},f_n^{[2]}\otimes_1f_n^{[2]}\big\rangle_{\mfh^{\otimes2}}
-4\big\langle f_n^{[i]},f_n^{[2]}\odot_1f_n^{[1]}\big\rangle_{\mfh^{\otimes2}}
+2\big\langle f_n^{[i]},f_n^{[1]}\otimes_1f_n^{[1]}\big\rangle_{\mfh^{\otimes2}}
\eeas
\end{en-text}

\subsection{Proof of Proposition \ref{2108101817}}\label{2108110837}
Let $\psi:\bbR\to[0,1]$ be a smooth function satisfying 
$\psi(x)=1$ on $\{x;|x|\leq1/2\}$, and $\psi(x)=0$ on $\{x;|x|\geq1\}$. 
We define $\psi_n$ by $\psi_1=0$ and 
\beas 
\psi_n 
&=& 
\psi\bigg(\frac{4|V_{n,T}^{(2)}-E[V_{n,T}^{(2)}]|^2}
{
\grnc{\eta_0^2}E[V_{n,T}^{(2)}]^2}\bigg)
\psi\bigg(\frac{4|V_{2n,T}^{(2)}-E[V_{2n,T}^{(2)}]|^2}
{
\grnc{\eta_0^2}E[V_{2n,T}^{(2)}]^2}\bigg)
\eeas
for $n\geq2$. 
Whenever $\psi_n>0$, 
\bea\label{2108140203}
\bigg|\frac{V^{(2)}_{n,T}}{E[V^{(2)}_{n,T}]}-1\bigg|<\half\eta_0
&\text{and}&
\bigg|\frac{V^{(2)}_{2n,T}}{E[V^{(2)}_{2n,T}]}-1\bigg|<\half\eta_0.
\eea
where $\eta_0$ is a positive constant. 
The properties of the truncation functional $\psi_n$ in (i) of Proposition \ref{2108101817} 
are easy to verified. 
We choose a sufficiently small $\eta_0$ and a positive integer $n_0$ both 
depending on $H$ so that 
$\wh{H}^{(2)}_n\in(0,1)$ whenever $\psi_n>0$ and $n\geq n_0$. 
We will only consider $n\geq n_0$ in what follows. 
On the event $\{\psi_n>0\}$, we have
\begin{en-text}
\beas 
\wh{H}^{(2)}_n
&=&
H
-\frac{1}{2\log 2}\log\frac{V^{(2)}_{2n,T}/E[V^{(2)}_{2n,T}]}{V^{(2)}_{n,T}/E[V^{(2)}_{n,T}
]}-\frac{1}{4n\log 2}+O(n^{-2})
\\&=&
H
-\frac{1}{2\log 2}\bigg\{\log \bigg(1-\bigg[\frac{V^{(2)}_{2n,T}}{E[V^{(2)}_{2n,T}]}-1\bigg]\bigg)
-\log \bigg(1-\bigg[\frac{V^{(2)}_{n,T}}{E[V^{(2)}_{n,T}]}-1\bigg]\bigg)
\bigg\}
\\&&
-\frac{1}{4n\log 2}+O(n^{-2})
\\&=&
H
+\frac{1}{2\log 2}\bigg\{\bigg[\frac{V^{(2)}_{2n,T}}{E[V^{(2)}_{2n,T}]}-1\bigg]
-\bigg[\frac{V^{(2)}_{n,T}}{E[V^{(2)}_{n,T}]}-1\bigg]
\\&&
{\colorr+}\half\bigg[\frac{V^{(2)}_{2n,T}}{E[V^{(2)}_{2n,T}]}-1\bigg]^2
{\colorr-}\half\bigg[\frac{V^{(2)}_{n,T}}{E[V^{(2)}_{n,T}]}-1\bigg]^2
\\&&
+\frac{1}{3}\bigg[\frac{V^{(2)}_{2n,T}}{E[V^{(2)}_{2n,T}]}-1\bigg]^3
-\frac{1}{3}\bigg[\frac{V^{(2)}_{n,T}}{E[V^{(2)}_{n,T}]}-1\bigg]^3
\bigg\}
\\&&
+O_{M}(1)\bigg[\frac{V^{(2)}_{2n,T}}{E[V^{(2)}_{2n,T}]}-1\bigg]^4
+O_{M}(1)\bigg[\frac{V^{(2)}_{n,T}}{E[V^{(2)}_{n,T}]}-1\bigg]^4
-\frac{1}{4n\log 2}+O(n^{-2})
\eeas
where $O_{M}(1)$ stands for a bounded sequence in $\bbD_\infty$. 
\end{en-text}
\beas 
\wh{H}^{(2)}_n
&=&
H
-\frac{1}{2\log 2}\log\frac{V^{(2)}_{2n,T}/E[V^{(2)}_{2n,T}]}{V^{(2)}_{n,T}/E[V^{(2)}_{n,T}
]}-\frac{1}{4n\log 2}+O(n^{-2})
\\&=&
H
-\frac{1}{2\log 2}\bigg\{\log \bigg(1\grnc{+}\bigg[\frac{V^{(2)}_{2n,T}}{E[V^{(2)}_{2n,T}]}-1\bigg]\bigg)
-\log \bigg(1\grnc{+}\bigg[\frac{V^{(2)}_{n,T}}{E[V^{(2)}_{n,T}]}-1\bigg]\bigg)
\bigg\}
\\&&
-\frac{1}{4n\log 2}+O(n^{-2})
\\&=&
H
\grnc{-}
\frac{1}{2\log 2}\bigg\{\bigg[\frac{V^{(2)}_{2n,T}}{E[V^{(2)}_{2n,T}]}-1\bigg]
-\bigg[\frac{V^{(2)}_{n,T}}{E[V^{(2)}_{n,T}]}-1\bigg]
\\&&
\grnc{-}\half\bigg[\frac{V^{(2)}_{2n,T}}{E[V^{(2)}_{2n,T}]}-1\bigg]^2
\grnc{+}\half\bigg[\frac{V^{(2)}_{n,T}}{E[V^{(2)}_{n,T}]}-1\bigg]^2
+R_n
\bigg\}
\\&&
-\frac{1}{4n\log 2}+O(n^{-2})
\eeas
where 
\grnc{\beas 
R_n 
&=& 
\bigg[\frac{V^{(2)}_{2n,T}}{E[V^{(2)}_{2n,T}]}-1\bigg]^3
\int_0^1(1-s)^2\bigg(1+ s\bigg[\frac{V^{(2)}_{2n,T}}{E[V^{(2)}_{2n,T}]}-1\bigg]\bigg)^{-3}ds
\\&&
-\bigg[\frac{V^{(2)}_{n,T}}{E[V^{(2)}_{n,T}]}-1\bigg]^3
\int_0^1(1-s)^2\bigg(1+s\bigg[\frac{V^{(2)}_{n,T}}{E[V^{(2)}_{n,T}]}-1\bigg]\bigg)^{-3}ds.
\eeas}
Since the family ${\mathfrak F}=\big\{n^{1/2}\big(V_{n,T}^{(2)}/E[V_{n,T}^{(2)}]-1\big)\big\}_{n\in\bbN}$ 
of Wiener functionals is 
bounded in $\bbD_\infty$, as is known by the hypercontractivity and 
stability of the ${\mathfrak F}$ under the Malliavin operator 
(i.e., $-2^{-1}L$ is the identity on the second chaos), 
we see
$R_n\psi_n = O_M(n^{-1.5})$ with the help of (\ref{2108140203}). 
\qed

\subsection{Asymptotic expansion of $M_n$}\label{2108140212}
We will apply Theorem 3 of Tudor and Yoshida \cite{tudor2019asymptotic} 
to the sequence $(M_n)_{n\in\bbN}$ 
by checking Conditions $[C1]$-$[C3]$ therein. 
Condition $[C1]$ is satisfied by Corollary \ref{2108140225}. 
Conditions $[C2]$ and $[C3]$ for 
boundedness of $\{M_n\}_{n\in\bbN}$ and $\{n^{1/2}(G_n-G_\infty)\}_{n\in\bbN}$ 
are verified with the hypercontractivity and stability of a fixed chaos under 
the Malliavin operator, 
as mentioned in the last part of the proof of Proposition \ref{2108101817}. 
Remark that 
$n^{1/2}(G_n-G_\infty)=n^{1/2}\wt{G}_n+O(n^{-1/2})$; 
$n^{1/2}\wt{G}_n$ is in the second chaos and 
the error term is deterministic. 

According to Corollary \ref{2108140225}, 
$T_n\to^dT_\infty$ as $n\to\infty$, where 
$T_\infty=(T_{\infty,1},T_{\infty,2})$ is a centered two-dimensional Gaussian random variable defined on some probability space with variance matrix ${\mathfrak T}$. 
Then 
\beas 
\bbE[T_{\infty,2}|T_{\infty,1}] &=& \theta T_{\infty,1}
\eeas
where $\theta={\mathfrak T}_{12}/{\mathfrak T}_{11}$. 
Suggested by Formula (23) of Tudor and Yoshida \cite{tudor2019asymptotic}, 
we define a symbol ${\mathfrak S}_n(\tti z)$ by 
\beas 
{\mathfrak S}_n(\tti z) 
&=& 
1+ \frac{1}{3}\theta{\mathfrak T}_{11}n^{-1/2}(\tti z)^3. 
\eeas
We note that introducing the parameter $\theta$ makes the resulting formula look slightly simple, that involves many ${\mathfrak T}_{11}^{-1}$ 
or $G_\infty^{-1}$ by derivatives. 
An approximate density $p_n(z)$ is then defined by 
\beas 
p_n^M(z)
\yeq
{\mathfrak S}_n(-\partial_z) \phi\big(z;0,{\mathfrak T}_{11}\big)
\yeq 
{\mathfrak S}_n(\partial_z)^* \phi\big(z;0,{\mathfrak T}_{11}\big)
\eeas
In other words, 
\bea\label{2108140701}
p_n^M(z)
&=&
\phi(z;0,G_\infty)+
\frac{1}{3}n^{-1/2}\theta G_\infty\bigg(\frac{z^3}{G_\infty^3}-\frac{3z}{G_\infty^2}\bigg)
\phi(z;0,G_\infty)
\eea
The function $p_n^M(z)$ is the Fourier inversion of the function 
\beas 
\varphi_n(u) 
&=& 
e^{-\half {\mathfrak T}_{11}u^2}+n^{-1/2}\frac{1}{3}\theta{\mathfrak T}_{11}(\tti u)^3e^{-\half {\mathfrak T}_{11}u^2}.
\eeas
Recall that $\cale(K,\gamma)$ is the set of measurable functions on $\bbR$ 
such that $|f(z)|\leq K(1+|z|^\gamma)$ for all $z\in\bbR$. 
Theorem 3 of Tudor and Yoshida \cite{tudor2019asymptotic} gives the following result. 
\begin{theorem}\label{2108140519}
For any positive numbers $K$ and $\gamma$, 
\beas 
\sup_{f\in\cale(K,\gamma)}
\bigg|E[f(M_n)]-\int_\bbR f(z)p_n^M(z)dz\bigg|
&=& 
o(n^{-1/2})
\eeas
as $n\to\infty$. 
\end{theorem}


\subsection{Proof of Theorem \ref{2108141203}}\label{2108140213}
In this section, we will combine the asymptotic expansion of $M_n$ 
with the perturbation method to give asymptotic expansion for $Z_n$. 
The perturbation method was used in Yoshida \cite{yoshida1997malliavin,yoshida2001malliavin,yoshida2013martingale}. 
For this method, we refer the reader to Sakamoto and Yoshida \cite{SakamotoYoshida2003} 
and Yoshida \cite{yoshida2020asymptotic}. 
Following the formulation in Sakamoto and Yoshida \cite{SakamotoYoshida2003}, 
it is possible to derive asymptotic expansion of $\bbS_n=\bbX_n+r_n\bbY_n$ 
from the joint convergence of $(\bbX_n,\bbY_n)$ 
once that of $\bbX_n$ was obtained. 
We set 
$\bbX_n=M_n$, $\bbY_n=N_n$, $\bbS_n=Z_n$ and $r_n=n^{-1/2}$. 
Let 
\beas 
\xi_n 
&=& 
\frac{4|G_n-G_\infty|^2}{G_\infty^2}. 
\eeas
Then, 
\beas 
|DM_n|_\mfh^2 
\yeq 
2G_n
>
G_\infty
\eeas
whenever $|\xi_n|<1$. 
Therefore 
$\sup_nE\big[1_{\{|\xi_n|<1\}}|DM_n|_\mfh^{-p}\big]<\infty$ for any $p>1$. 
Moreover, $P\big[|\xi_n|>1/2\big]=O(n^{-L})$ for every $L>0$ 
since $\wt{G}_n$ is bounded in $L^\inftym=\cap_{p>1}L_p$. 
It is easy to see $\xi_n$, as well as $M_n$ and $N_n$, is bounded in $\bbD_\infty$. 
On the other hand, 
since 
\beas 
\big(M_n^{[1]},M_n^{[2]}\big)&\to^d& \big(M_\infty^{[1]},M_\infty^{[2]}\big)
\quad(n\to\infty)
\eeas
from Proposition \ref{2108140105}, we obtain 
\begin{en-text}
\redc{\beas
(M_n,N_n)&=& \bigg(M_n^{[2]}-M_n^{[1]},\>\half\big\{ (M_n^{[2]})^2-(M_n^{[1]})^2\big\}
-\half\bigg)+o_p(1)
\nn\\&\to^d&
\bigg(M_\infty^{[2]}-M_\infty^{[1]},\>\half\big\{ (M_\infty^{[2]})^2-(M_\infty^{[1]})^2\big\}
-\half\bigg). 
\eeas}
\end{en-text}
\grnc{\bea\label{2108150310}
(M_n,N_n)&=& \bigg(-M_n^{[2]}+M_n^{[1]},\>\half\big\{ (M_n^{[2]})^2-(M_n^{[1]})^2\big\}
-\half\bigg)+o_p(1)
\nn\\&\to^d&
\bigg(-M_\infty^{[2]}+M_\infty^{[1]},\>\half\big\{ (M_\infty^{[2]})^2-(M_\infty^{[1]})^2\big\}
-\half\bigg). 
\eea}
Therefore
\begin{en-text}
\redc{\beas
g_\infty(z) 
&=& 
-\partial_z\bigg\{\bbE\bigg[\half\big\{ (M_\infty^{[2]})^2-(M_\infty^{[1]})^2\big\}
-\half\bigg|M_\infty^{[2]}-M_\infty^{[1]}=z\bigg]\phi(z;0,G_\infty)\bigg\}
\nn\\&=&
-\half\partial_z\bigg\{\bbE\bigg[z\big\{ M_\infty^{[2]}+M_\infty^{[1]}\big\}
-1\bigg|M_\infty^{[2]}-M_\infty^{[1]}=z\bigg]\phi(z;0,G_\infty)\bigg\}
\nn\\&=&
-\half\partial_z\big\{\big[\tau z^2-1\big]\phi(z;0,G_\infty)\big\}
\nn\\&=&
\bigg(\frac{\tau}{2G_\infty}z^3-(\tau+\frac{1}{2G_\infty})z\bigg)
\phi(z;0,G_\infty)
\eeas}
\end{en-text}
\grnc{\bea\label{2108140714}
g_\infty(z) 
&=& 
-\partial_z\bigg\{\bbE\bigg[\half\big\{ (M_\infty^{[2]})^2-(M_\infty^{[1]})^2\big\}
-\half\bigg|-M_\infty^{[2]}+M_\infty^{[1]}=z\bigg]\phi(z;0,G_\infty)\bigg\}
\nn\\&=&
-\half\partial_z\bigg\{\bbE\bigg[-z\big\{ M_\infty^{[2]}+M_\infty^{[1]}\big\}
-1\bigg|-M_\infty^{[2]}+M_\infty^{[1]}=z\bigg]\phi(z;0,G_\infty)\bigg\}
\nn\\&=&
-\half\partial_z\big\{\big[\tau z^2-1\big]\phi(z;0,G_\infty)\big\}
\nn\\&=&
\bigg(\frac{\tau}{2G_\infty}z^3-(\tau+\frac{1}{2G_\infty})z\bigg)
\phi(z;0,G_\infty)
\eea}
with the coefficient 
\beas 
\tau
&=& 
\frac{\Sigma_{22}-\Sigma_{11}}{G_\infty}
\yeq 
\frac{{\mathfrak U}_{22}-{\mathfrak U}_{11}}{G_\infty}.
\eeas

From (\ref{2108140701}) and (\ref{2108140714}), we obtain 
the expansion formula for $Z_n$ as 
\bea\label{2108141141}
p_n^Z(z)
&=&
\big(1+n^{-1/2}q^Z(z)\big)\phi(z;0,G_\infty),
\eea
where
\bea\label{2108141142}
q^Z(z)
&=&
\bigg(\frac{\theta}{3G_\infty^2}+\frac{\tau}{2G_\infty}\bigg)z^3
-\bigg(\frac{2\theta+1}{2G_\infty}+\tau\bigg)z;
\eea
recall $\theta={\mathfrak T}_{12}/{\mathfrak T}_{11}$. 
That is, 
\bea\label{2108141446}
\sup_{f\in\cale(K,\gamma)}\bigg|
E[f(Z_n)] -\int_\bbR f(z)p_n^Z(z)dz\bigg|
&=&
o(n^{-1/2})
\eea
as $n\to\infty$. 
However, since $\wh{H}_n^{(2)}\in[0,1]$ by definition, we know 
\bea\label{21081447}
\sup_{f\in\cale(K,\gamma)}\big|
E\big[f\big(c\sqrt{n}(\wh{H}_n^{(2)}-H)\big)\big]-E[f(Z_n)]\big|
&=&
\sup_{f\in\cale(K,\gamma)}\big|
E\big[f\big(c\sqrt{n}(\wh{H}_n^{(2)}-H)\big)\big(1-\psi_n\big)\big]\big|
\nn\\&\leq&
M\big(1+(c\sqrt{n})^\gamma\big)
\big\|1-\psi_n\big\|_1
\nn\\&=&
O(n^{-L})
\eea
as $n\to\infty$, for every $L>0$, where $c=2\log2$. 
Consequently 
\bea\label{2108141446}
\sup_{f\in\cale(K,\gamma)}\bigg|
E\big[f\big(c\sqrt{n}(\wh{H}_n^{(2)}-H)\big)\big]-\int_\bbR f(z)p_n^Z(z)dz\bigg|
&=&
o(n^{-1/2})
\eea
as $n\to\infty$. 
\begin{en-text}
\bea\label{21081447}&&
\sup_{f\in\cale(K,\gamma)}\big|
E\big[f\big(c\sqrt{n}(\wh{H}_n^{(2)}-H)\big)\big]-E[f(Z_n)]\big|
\nn\\&=&
\sup_{f\in\cale(K,\gamma)}\big|
E\big[f\big(c\sqrt{n}(\wh{H}_n^{(2)}-H)\big)\big(1-\psi_n\big)\big]\big|
\nn\\&\simleq&
E\big[\big(1+\big|\log V_{n,T}^{(1)}\big|+\big|\log V_{n,T}^{(2)}\big|\big)^\gamma\big|1-\psi_n\big|\big]
\nn\\&\simleq&
E\big[\big(1+\big|\log V_{n,T}^{(1)}\big|+\big|\log V_{n,T}^{(2)}\big|\big)^\gamma\big|1-\psi_n\big|\big]
\nn\\&\simleq&
\sum_{\alpha=1}^2 \bigg\{1+
E\big[|V_{n,T}^{(\alpha)}|^{2\gamma}\big]^{1/2}
+E\big[|V_{n,T}^{(\alpha)}|^{-2\gamma}\big]^{1/2}\bigg\}E[|1-\psi_n|^2]^{1/2}
\nn\\&=&
o(n^{-L})
\eea
as $n\to\infty$, for every $L>0$, where $c=2\log2$. 
Here $\simleq$ is standing for an inequality up to a constant depending 
only on $(M,\gamma)$, and the estimates are based on 
the estimates 
\beas 
E\big[|V_{n,T}^{(\alpha)}|^{2\gamma}\big] &=& O(n^{c(\gamma)})
\eeas
for some constant $c(\gamma)$ depending on $\gamma$, 
\bea\label{21081413}
\sup_{n\in\bbN}E\big[|V_{n,T}^{(\alpha)}|^{-2\gamma}\big] &<& \infty \koko
\eea
and 
\beas 
E[|1-\psi_n|^2] &=& O(n^{-L})
\eeas
for any $L>0$. 
For the estimate (\ref{21081413}) is as follows. 
Take a positive constant $\eta$ sufficiently small compared with $T^{2H}$. 
Then, by applying Lemma 4.1 of \koko 
to the sequence $(n^Hd_{n,j})_{j=1,...,n}$ 
\bea\label{2108141443}
P\bigg[n^{2H}\sum_{j=1}^{n-1}d_{n,j}^2\leq \eta k^{-1}\bigg]
&\leq&
\exp\bigg[-\frac{\big(\sum_{i=1}^{n-1}n^{2H}E[d_{n,j}^2]- \eta k^{-1}\big)^2}
{4\sum_{i,j=1}^{n-1}\big(n^{2H}E[d_{n,i}d_{n,j}]\big)^2}\bigg]
\nn\\&=&
\exp\bigg[-\frac{\big(T^{2H}(n-1)- \eta k^{-1}\big)^2}
{4\sum_{i,j=1}^{n-1}T^{4H}n^{-4H}\wh{\rho}(i-j)^2}\bigg]
\nn\\&\leq&
\exp(-\eta_1n)\qquad(\forall k\in\bbN)
\eea
for some positive constant $\eta_1$, since 
$n^{-1}\sum_{i,j=1}^{n-1}\wh{\rho}(i-j)^2$ tends to a positive constant. 
The estimate (\ref{21081413}) is a consequence of (\ref{2108141443}). 
\end{en-text}

Finally we rescale the argument of $p_n^Z(z)$ by $c$ to obtain 
the approximate density $p_n(z)$ for $\wh{H}_n^{(2)}$ 
appearing in (\ref{2108141155}) with
\bea\label{2108141156}
q(z)\yeq q^Z({\colorr c}z)&\text{and}&v_H=c^{-2}G_\infty.
\eea
This completes the proof of Theorem \ref{2108141203}. 
\qed\halflineskip

\subsection{Proof of Theorem \ref{2108150211}}
The effect of the modification in (\ref{2108150237}) 
appears as the shift of $\wh{H}_n^{(2)}$ by the constant $-b(H)/n$ 
in the asymptotic expansion. 
So, the estimate (\ref{2108150242}) is almost obvious if we replace 
the function $f(z)$ by $f\big(z-n^{-1/2}b(H)\big)$ in (\ref{2108150241}) 
with some modified $(M,\gamma)$ and next by expansion after change of variables. 
Rigorous justification does not matter. 
{\sred 
We may cut off the event $\{\wh{H}_n^{(2)}\not\in (H/2,(H+1)/2)\}$ 
since $\big\|\sqrt{n}\big(\wh{H}_n^{(b)}-H\big)\big\|_\infty=O(n^{1/2})$.}
Just start with $N_n$ including $-cb\big(\wh{H}_n^{(2)})$. 
Then a modification is adding $-cb(H)$ in (\ref{2108150310}) 
to the limit of $N_n$. 
This causes addition of $-cb(H)G_\infty^{-1}z\phi(z;0,G_\infty)$ to (\ref{2108140714}) 
to give the same result as the one by the above intuition. 
\qed\halflineskip
{\sred
\begin{remark}\label{2109051131}\rm 
It is not difficult to show that $b^*$ and $b^{**}\in C^\infty(0,1)$ 
by using the representation of $\Sigma_{ij}$ ($i,j=1,2$) in $H$, and 
the representation (\ref{2108141156}) of $q(z)$ depending on $H$, 
as well as the uniform positivity (\ref{2108211140}) of $G_\infty$ depending on $H$. 
\end{remark}
}

\section{Numerical study}\label{2208290905}
{\sred 
We give some results of simulation. 
The following two figures compare the histogram of $\hat H_n^{(2)}$ with normal approximation and asymptotic expansion, 
in the case of $H=0.5$ and $n=32$. 
The figure on the left plots the histogram of $\hat H_n^{(2)}$ 
and the one on the right the histogram but without cutoff as (\ref{2108101802}), 
i.e., 
$\half-\frac{1}{2\log2}\log\big({V_{2n,T}^{(2)}}/{V_{n,T}^{(2)}}\big)$.
The green curves are obtained by the asymptotic expansion, and the red ones by the normal approximation. 
%
\begin{figure}[H]
    \centering
  \begin{subfigure}[t]{0.45\textwidth}
    \centering
		\includegraphics[width=\linewidth]{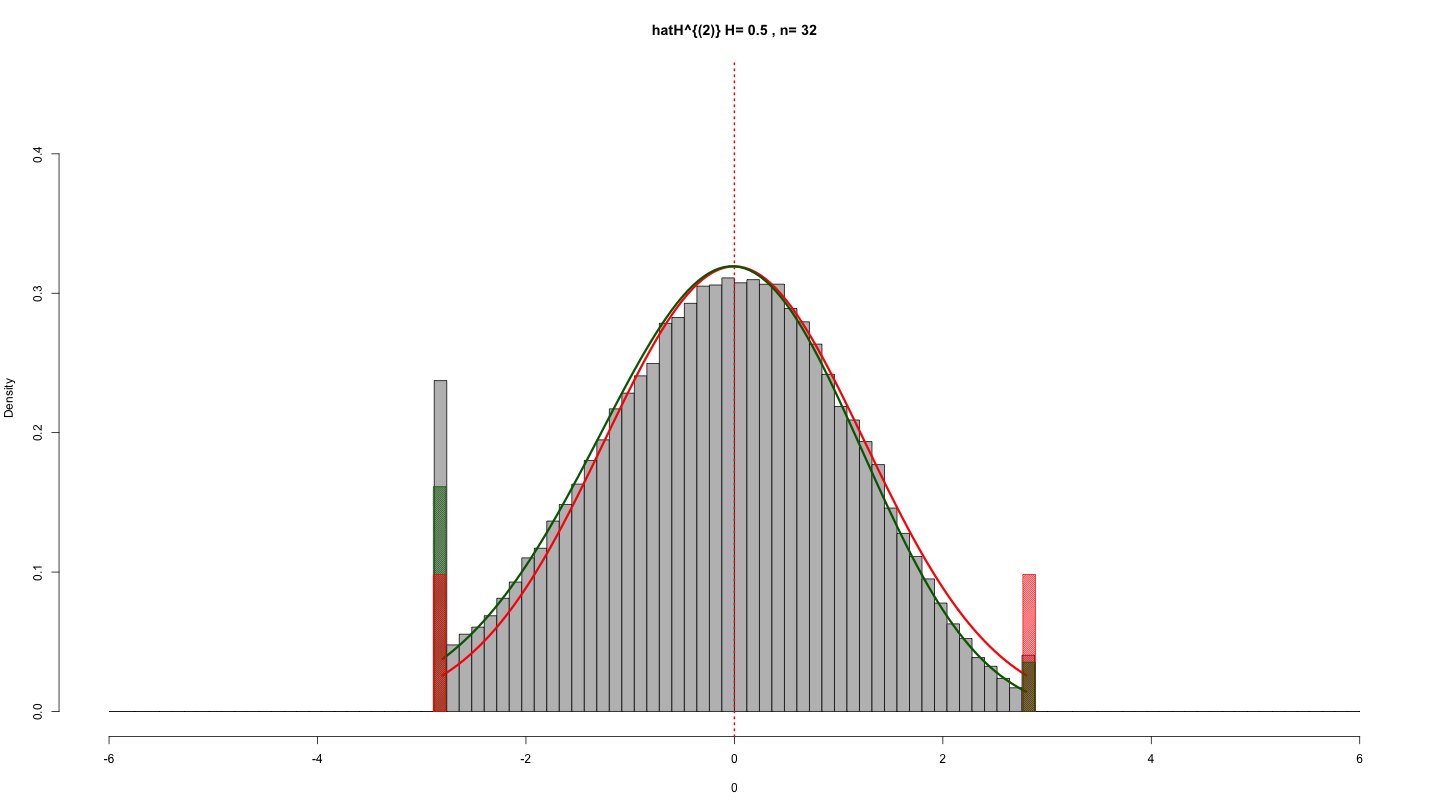}
		\caption{$H=0.5$, $n=32$}
  \end{subfigure}
  \begin{subfigure}[t]{0.45\textwidth}
    \centering
		\includegraphics[width=\linewidth]{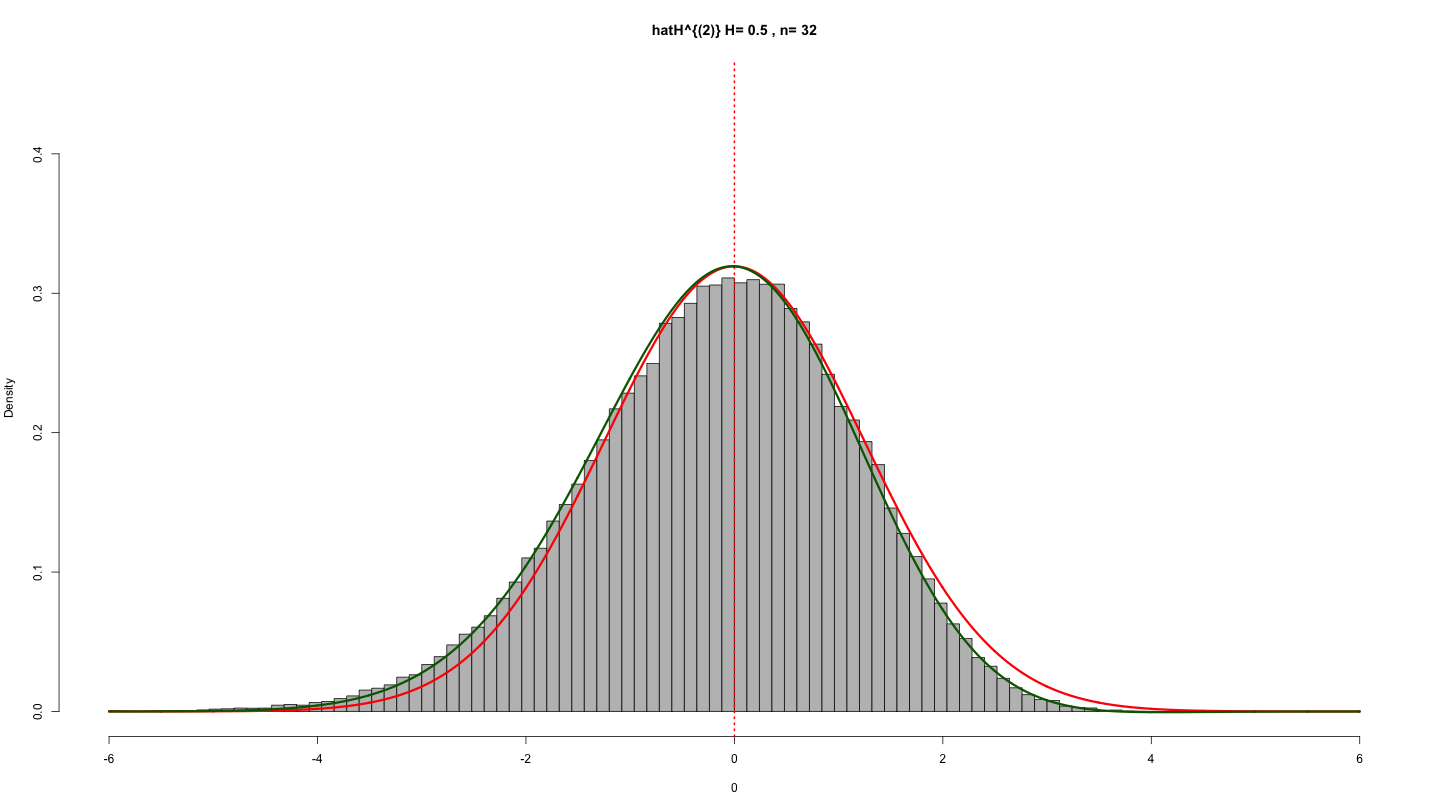}
		\caption{$H=0.5$, $n=32$}
  \end{subfigure}
\end{figure}

When $n$ increases, the atoms become smaller as well as the difference between the asymptotic expansion and the normal approximation decreases.  
%
\begin{figure}[H]
    \centering
  \begin{subfigure}[t]{0.45\textwidth}
    \centering
		\includegraphics[width=\linewidth]{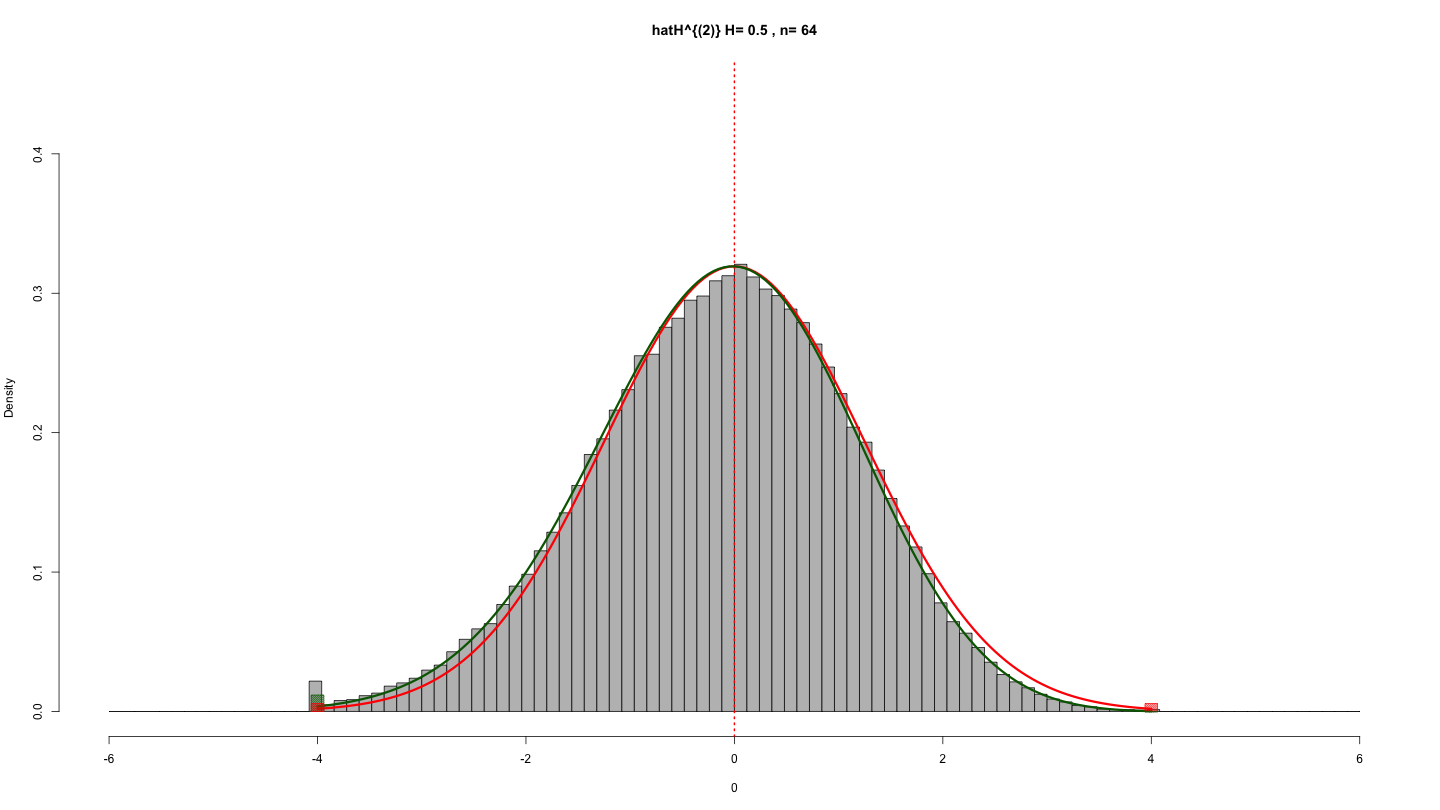}
		\caption{$H=0.5$, $n=64$}
  \end{subfigure}
  \begin{subfigure}[t]{0.45\textwidth}
    \centering
		\includegraphics[width=\linewidth]{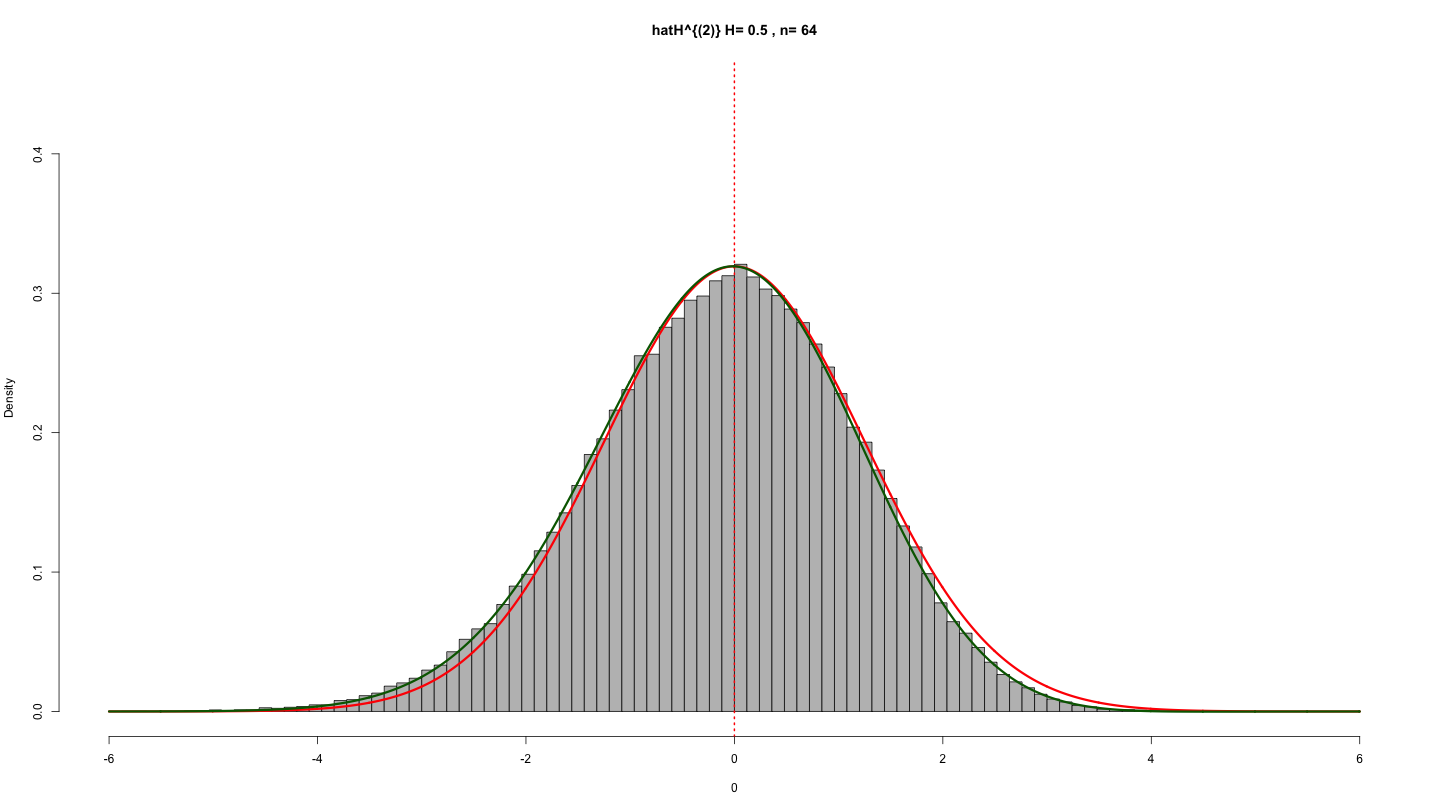}
		\caption{$H=0.5$, $n=64$}
  \end{subfigure}
\end{figure}

Since the precision of approximation of the atoms is determined by that of density, we will leave the curves and omit to draw atoms by the approximations, 
in the following several plots for different values of $H$. 
In practical use, the density should be integrated on the tails to approximate 
the distribution of $\wh{H}_n^{(2)}$ of (\ref{2108101802}), that has atoms on the endpoints $0$ and $1$. 
}
\begin{en-text}
{\sred atomが見えない？どのようなプロットにするか検討}

The following figure on the left is the histogram of $\hat H_n^{(2)}$ in the case of 
$H=0.05$ and $n=32$, 
and the next one is the histgram without cutoff at $0$ and $1$, 
that is 
$\half-\frac{1}{2\log2}\log\frac{V_{2n,T}^{(2)}}{V_{n,T}^{(2)}}$.
\end{en-text}
\begin{figure}[H]
	\centering
\begin{subfigure}[t]{0.45\textwidth}
		\centering
		\includegraphics[width=\linewidth]{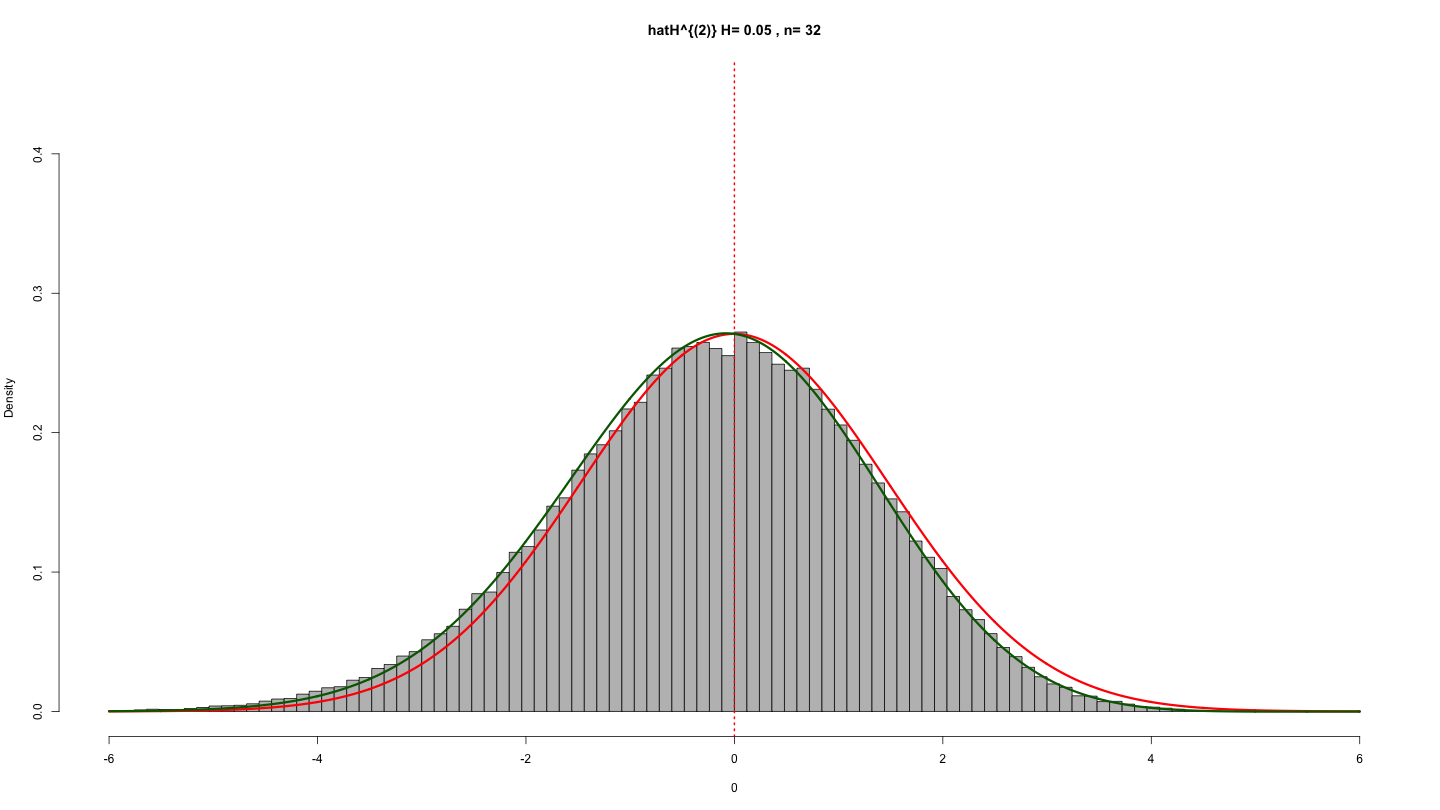}
		\caption*{$H=0.05$, $n=32$}
\end{subfigure}
\begin{subfigure}[t]{0.45\textwidth}
	\centering
	\includegraphics[width=\linewidth]{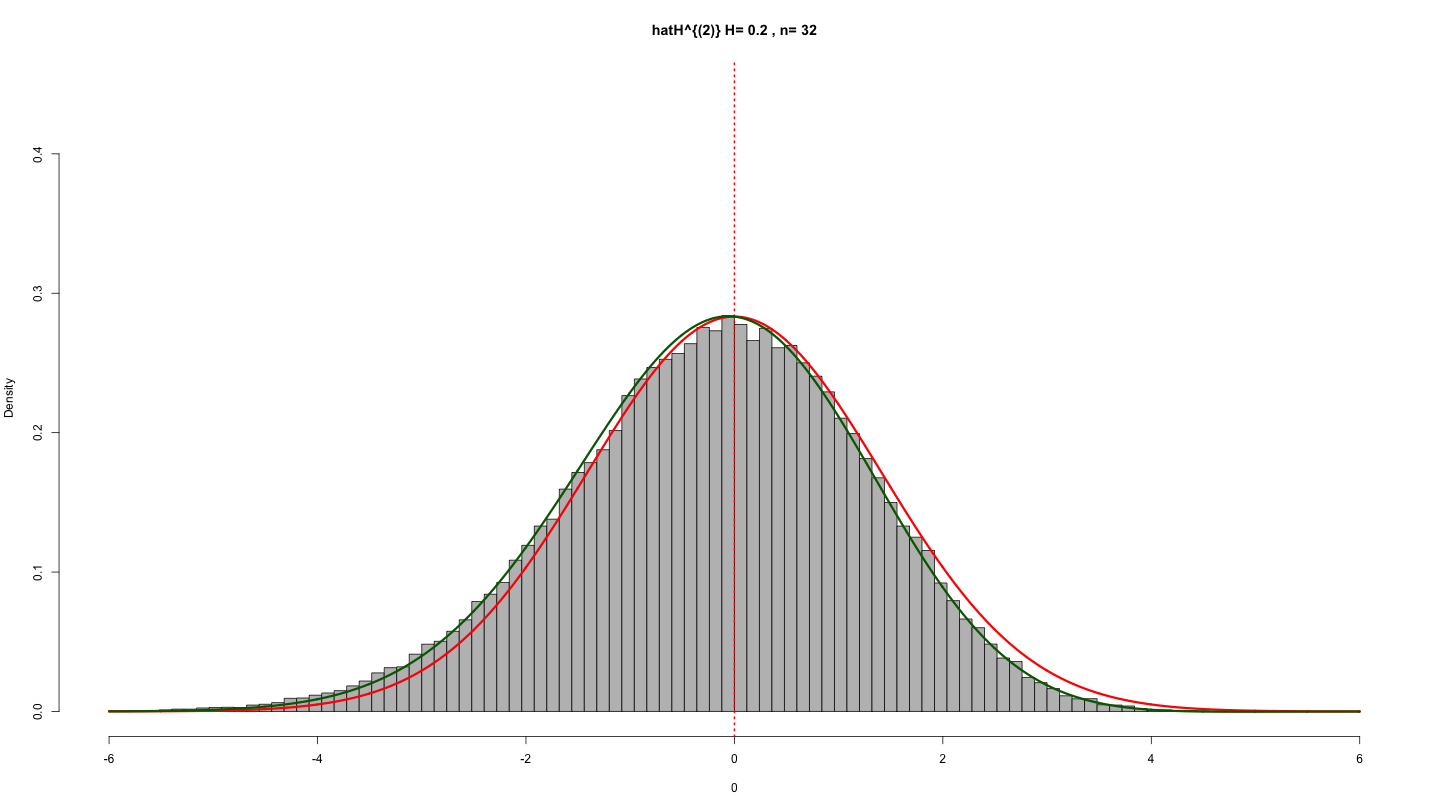}
	\caption*{$H=0.20$, $n=32$}
\end{subfigure}
\end{figure}

\begin{figure}[H]
	\centering
\begin{subfigure}[t]{0.45\textwidth}
		\centering
		\includegraphics[width=\linewidth]{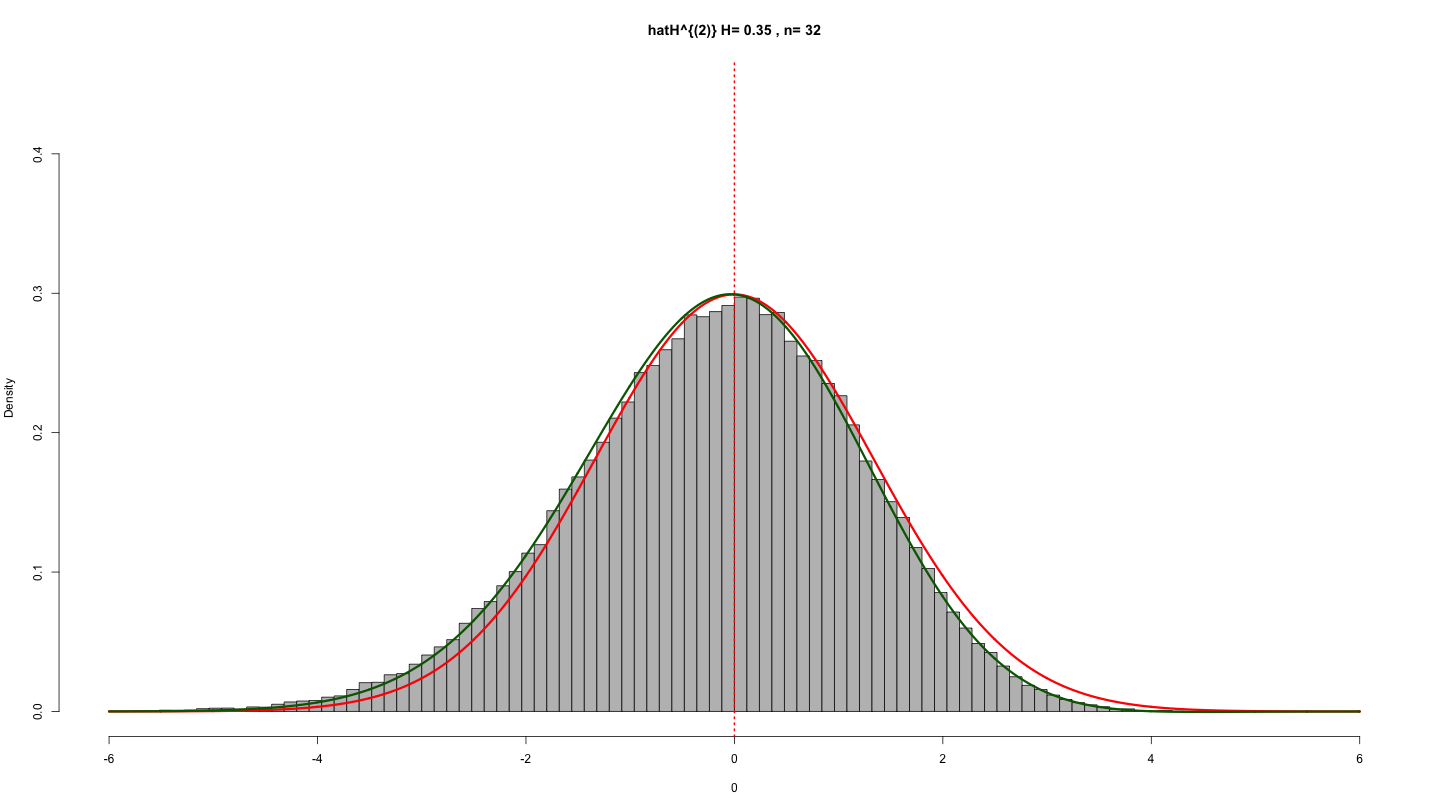}
		\caption*{$H=0.35$, $n=32$}
\end{subfigure}
\begin{subfigure}[t]{0.45\textwidth}
	\centering
	\includegraphics[width=\linewidth]{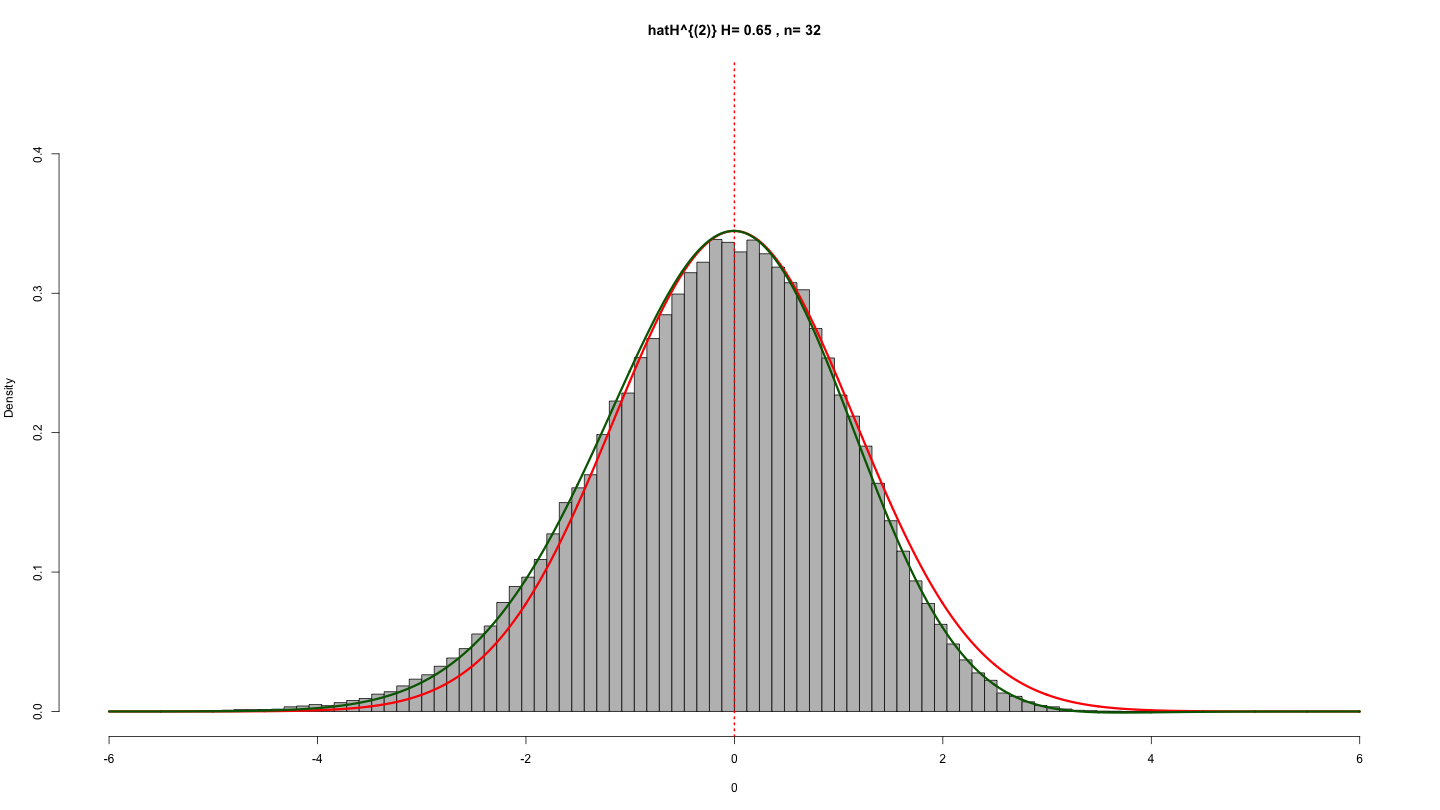}
	\caption*{$H=0.65$, $n=32$}
\end{subfigure}
\end{figure}

\begin{figure}[H]
	\centering
\begin{subfigure}[t]{0.45\textwidth}
		\centering
		\includegraphics[width=\linewidth]{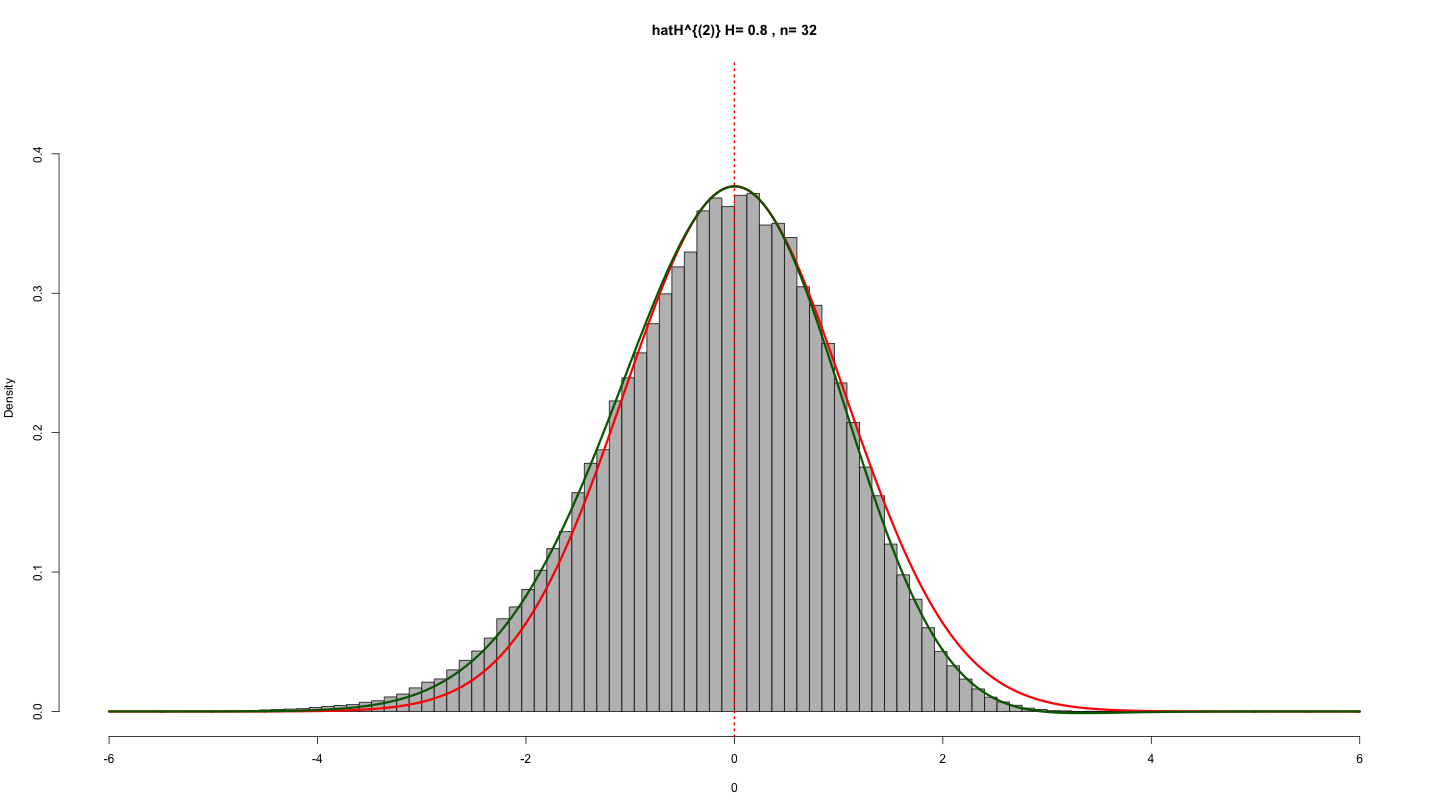}
		\caption*{$H=0.80$, $n=32$}
\end{subfigure}
\begin{subfigure}[t]{0.45\textwidth}
	\centering
	\includegraphics[width=\linewidth]{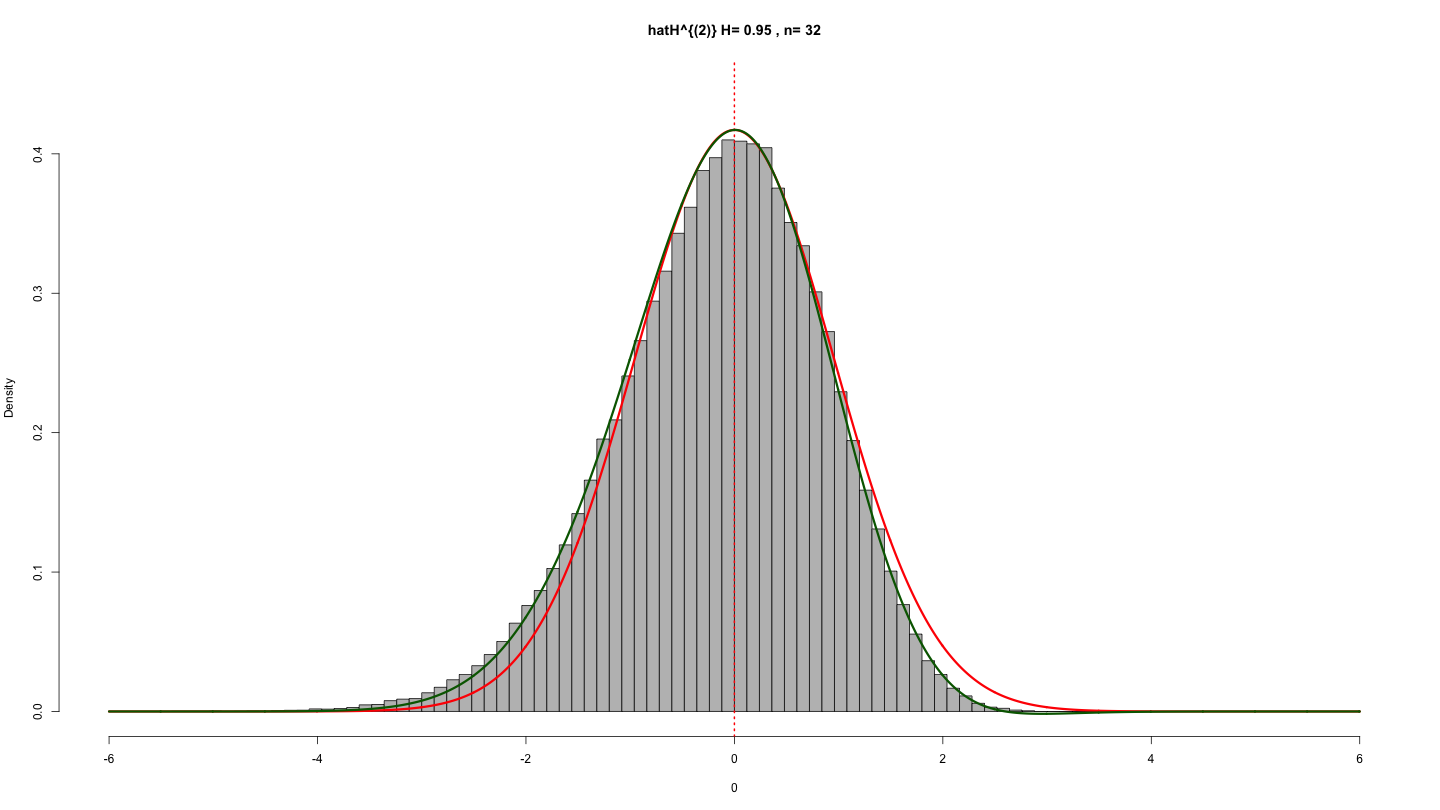}
	\caption*{$H=0.95$, $n=32$}
\end{subfigure}
\end{figure}
\begin{en-text}
Below are for $H=0.95$.
\begin{figure}[H]
	\centering
\begin{subfigure}[t]{0.45\textwidth}
		\centering
		\includegraphics[width=\linewidth]{hatH2.H.95,n32_.png}
		\caption{$H=0.95$, $n=32$}
\end{subfigure}
\begin{subfigure}[t]{0.45\textwidth}
	\centering
	\includegraphics[width=\linewidth]{hatH2.H.95,n32_nocutoff.png}
	\caption{$H=0.95$, $n=32$}
\end{subfigure}
\end{figure}
\end{en-text}

\bibliographystyle{spmpsci}      
\bibliography{bibtex-20210212-20210810+}   

\end{document}